\documentclass[12pt,reqno]{amsart}

\usepackage[dvipsnames, svgnames]{xcolor}
\usepackage{amsmath, amssymb,graphicx,amsthm,latexsym, amsfonts, enumitem, mathtools, tensor, MnSymbol}
\usepackage{hyperref}
\usepackage[all, color]{xy}
\usepackage{color}
\usepackage{amssymb, wasysym}
\usepackage{xcolor,cancel}
\usepackage{float}

\newcommand\hcancel[2][black]{\setbox0=\hbox{$#2$}
\rlap{\raisebox{.45\ht0}{\textcolor{#1}{\rule{\wd0}{1pt}}}}#2}
\usepackage{float}
\usepackage{caption}
\usepackage{subcaption}

\newcommand{\knotlinewidth}{1pt}
\usepackage{tikz}
\usetikzlibrary{arrows,decorations.pathmorphing,decorations.pathreplacing,positioning,shapes.geometric,shapes.misc,decorations.markings,decorations.fractals,calc,patterns,braids,math,arrows.meta}
\usetikzlibrary{calc}
\usetikzlibrary{intersections}

\tikzset{overcross/.style={double, line width=2, white, double=black, double distance=\knotlinewidth}}

\newcommand{\MO}{\mathcal{O}}

\theoremstyle{plain}
\newtheorem{theorem}{Theorem}[section]
\newtheorem*{theorem*}{Theorem}

\theoremstyle{definition}
\newtheorem{defn}[theorem]{Definition}

\newtheorem{remark}[theorem]{Remark}

\newtheorem{lemma}[theorem]{Lemma}
\newtheorem{notation}[theorem]{Notation}

\newtheorem{setup}[theorem]{Setup}
\newtheorem{proposition}[theorem]{Proposition}

\newtheorem{mainthm}{Theorem}

\date{}

\thanks{
This work was supported by the Engineering and Physical Sciences Research Council [grant number EP/W007509/1] and by the Centre for Advanced Study (CAS) in Oslo, Norway.}

\begin{document}
\setlength{\parindent}{0pt}
\setlength{\parskip}{7pt}
\author{Francesca Fedele}
\address{School of Mathematics, University of Leeds, Leeds LS2 9JT, U.K.}
\email{f.fedele@leeds.ac.uk}
\author{Bethany Rose Marsh}
\address{School of Mathematics, University of Leeds, Leeds LS2 9JT, U.K.}
\email{B.R.Marsh@leeds.ac.uk}
\title{Presentations of the braid group of the complex reflection group $G(d,d,n)$}
\begin{abstract}
    We show that the braid group associated to the complex reflection group $G(d,d,n)$ is an index $d$ subgroup of the braid group of the orbifold quotient of the complex numbers by a cyclic group of order $d$. We also give a compatible presentation of $G(d,d,n)$ and its braid group for each tagged triangulation of the disk with $n$ marked points on its boundary and an interior marked point (interpreted as a cone point of degree $d$) in such a way that the presentations of Brou\'{e}-Malle-Rouquier correspond to a special tagged triangulation.
\end{abstract}
\maketitle
\tableofcontents

\section{Introduction}
\newcommand{\DD}{{\mathcal{D}}} 
\newcommand{\AAA}{{\mathcal{A}}} 

Our main aim is to give a family of presentations of the braid group $B(d,d,n)$ of the complex reflection group $G(d,d,n)$, for positive integers $d,n$ (see~\cite[\S B]{BMR98}), with one presentation associated to each tagged triangulation (in the sense of~\cite[\S 7]{FST08}) of an orbifold given by a disk with a single cone point of degree $d$. In addition, we show that $B(d,d,n)$ can be embedded in the $n$-strand braid group of the orbifold as a subgroup of index $d$ (a result obtained independently in~\cite{F23B}; see the comment after Theorem A below), generalising a result of Allock~\cite[Theorem 1.1]{A}. This allows us to give a geometric interpretation of the generators in each presentation in the family. This generalises a family of presentations of the Artin braid group of type $D_n$ given in~\cite{GM}, which can be regarded as the case $d=2$.

Recently, there have been a number of articles giving presentations of braid groups using the theory of cluster algebras, and these form part of the motivation for this paper. In~\cite{BM}, a family of presentations of finite Weyl groups was given, one for each seed in the corresponding cluster algebra; the subsequent article~\cite{GM} gave alternative presentations in the simply-laced case, which lifted to the corresponding braid groups (see also~\cite{QZ20}). An independent proof of this was found by Alastair King and Qiu Yu (see~\cite[Prop. 10.3]{Qiu16}). Braid group presentations (for all finite cases) were also given in~\cite{HHLP}. Presentations for types $H$ and $I$ were given in~\cite[Thm.\ 3.5]{HHQ24}. Presentations for affine Coxeter groups were given in~\cite{FT16} and groups of a similar kind were associated to surfaces in~\cite{FLST21}. This article can be regarded as providing presentations similar in style to these cluster algebra-theoretic presentations, but we note that there is no cluster algebra associated to a complex braid group, and that the mutation considered here (see Section~\ref{subsection_mutation_q_triang}), although related, is not the same as Fomin-Zelevinsky
mutation~\cite[Defn.\ 4.2]{FZ02} (or the corresponding diagram mutation~\cite[\S8]{FZ}); in particular, the quivers considered here have additional decorations that do not appear in the theory of cluster algebras.

It is also interesting to note the article~\cite{KingQiu20}, which associates a groupoid, known as the cluster exchange groupoid, to a cluster algebra of Dynkin type, showing that the fundamental group is isomorphic to the corresponding Artin braid group~\cite[Thm.\ 2.16]{KingQiu20}, giving an alternative construction of the Artin braid group in these cases. There are also strong relationships with mapping class groups and groups generated by spherical twists; see, for example~\cite{qiuyusummary19} and the references therein.

In the remainder of Section $1$ we recall the relevant theory and background for real and complex reflection groups and braid groups, and state our main results in more detail.
In Section $2$ we give an orbifold realisation of the braid group of the complex reflection group $G(d,d,n)$. In Section $3$ we give the promised family of presentations of $G(d,d,n)$ and the corresponding braid group, and in Section $4$, we give a geometric interpretation of the generators in these presentations in terms of the geometric description in Section $2$.

\textbf{Acknowledgement}: We would like to thank Paul P. Martin for several useful discussions related to this work.

\subsection{(Real) reflection and braid groups.}

In (a special case of) \cite[Theorem~5.4]{BM}, Barot and Marsh proved that if $Q$ is a \textit{mutation-Dynkin quiver}, i.e. a quiver that can be obtained by mutating a Dynkin quiver $\Delta$ of type $ADE$ in the sense of~\cite[Defn.\ 4.2]{FZ02} finitely many times, then the associated group $W(Q)$ is isomorphic to the Weyl group $W(\Delta)$.
Let $n$ be the number of vertices in $\Delta$. As shown in \cite{brieskorn1971fundamentalgruppe} (and recalled in \cite[Section~2]{A}), the Artin braid group $\mathcal{A}(\Delta)$ of type $\Delta$ is isomorphic to the fundamental group 
\begin{align*}
\pi_1 \left(\left(V-\bigcup_{s\in\Sigma}H_s\right)/W(\Delta)\right),
\end{align*}
where $\Sigma$ is the set of reflections in $W(\Delta)$, $V$ is the complexification of $\mathbb{R}^n$ and $H_s$ the complexification of the set of fixed points of $s$ in $\mathbb{R}^n$. 

More abstractly, the Artin braid group $\mathcal{A}(\Delta)$ can be defined in terms of generators and relations associated to the corresponding graph of type $\Delta$. For example, the Coxeter graph of type $D_n$:
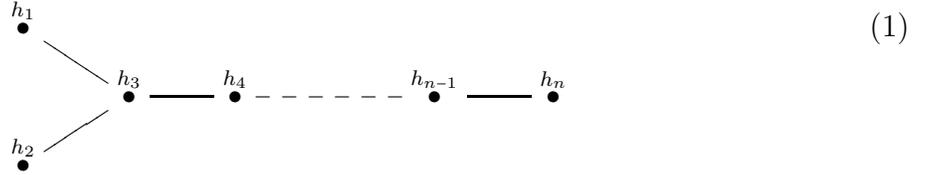
\begin{align}\label{eqn_DynkinDn}
    \xymatrix@R=5pt{
\stackrel{h_1}{\bullet}\ar@{-}[dr]^{}
     \\ 
     & \stackrel{h_3}{\bullet} \ar@{-}[r] & \stackrel{h_4}{\bullet} \ar@{--}[rr]&& \stackrel{h_{n-1}}{\bullet} \ar@{-}[r] &\stackrel{h_n}{\bullet}
     \\
     \stackrel{h_2}{\bullet}\ar@{-}[ur]^{}
     }
\end{align}
gives the standard presentation of $\mathcal{A}(D_n)=\langle h_1,h_2,\dots,h_n\mid R\rangle$, where $R$ is the set of relations $h_ih_jh_i=h_jh_ih_j$ if there is an edge between $h_i$ and $h_j$ and $h_ih_j=h_jh_i$ otherwise. The Weyl group $W(D_n)$ is then the quotient of $\mathcal{A}(D_n)$ obtained by adding the relations $h_i^2=e$ for all $i$, where $e$ is the identity element.

Allcock described the connection between some Artin braid groups and orbifold fundamental groups. In particular, in \cite[Theorem~1.1]{A}, he proved that $\mathcal{A}(D_n)$ is isomorphic to a subgroup of index $2$ of the the orbifold fundamental group
\begin{align*}
    Z_n(\mathcal{O}_2)=\pi_1 ((\mathcal{O}^n_2-\Delta_n)/S_n),
\end{align*}
where $\MO_2$ is the orbifold $\mathbb{C}/C_2$, $\Delta_n=\{ (x_1,x_2,\dots,x_n)\in\mathcal{O}^n_2 : x_i=x_j \text{ for some } i\neq j\}$ and $S_n$ is the symmetric group of degree $n$.

Subsequently, Grant and Marsh studied presentations of Artin braid groups of type $ADE$. In \cite[Theorem A]{GM} they showed that if $Q$ is a mutation-Dynkin quiver, then the associated braid group is isomorphic to the Artin braid group $\mathcal{A}(\Delta)$ of the corresponding Dynkin type. This way one obtains many presentations of the Artin braid groups of type $ADE$.

Moreover, they showed that an orientation of (\ref{eqn_DynkinDn}) coincides with the quiver $Q_{T_0}$ associated with the \textit{initial (tagged) triangulation} $T_0$  of $(X, M)$, where $X$ is the disk $S$ with an interior marked point interpreted as a cone point of degree $2$, and $M$ is a set of $n$ marked points on the boundary of $X$; see \cite[page 91 and Figure 5]{GM}. See~\cite[\S7]{FST08} for the definition of tagged triangulations; see also \cite[Section 3]{GM}. Note that the interior of $X$ is isomorphic as an orbifold to $\mathcal{O}_2$.

Flipping a triangulation corresponds to mutating the quiver associated to it and, using the fact that the graph of flips of (tagged) triangulations of the disk is connected, it was shown in \cite[Theorem~A]{GM} that any (tagged) triangulation $T$ of $(X,M)$ gives a presentation of $\mathcal{A}(D_n)$. 
Moreover, $T$ has an associated braid graph, the edges of which correspond to elements $\sigma_i$ in $Z_n(\mathcal{O}_2)$.
In \cite[Theorem~3.6]{GM}, Grant and Marsh proved that the subgroup $B_T$ of $Z_n(\mathcal{O}_2)$ generated by the elements $\sigma_i$ is isomorphic to the group $G_{Q_T}$ associated to the quiver corresponding to $T$. Hence the group presentation associated to the triangulation $T$ gives a presentation of $\mathcal{A}(D_n)$ as a subgroup of index $2$ of $Z_n(\mathcal{O}_2)$.

\subsection{Complex reflection and braid groups.}

In this paper, we are interested in studying a ``complex'' version of the above.

A \textit{pseudo-reflection $s$} is a non-trivial element in the general linear group GL$(\mathbb{C}^n)$ which fixes a hyperplane $H_s$ pointwise, known as \textit{the reflecting hyperplane of $s$}.
A group generated by pseudo-reflections is known as a \textit{complex reflection group}.
The irreducible finite complex reflection groups were classified by Shepard and Todd in \cite{shephard1954finite}. Brou\'{e}, Malle and Rouquier provided presentations of all such groups using Coxeter-like diagrams; see \cite[Tables 1-4 in Appendix~2]{BMR98}.
Here we focus on the complex reflection groups of the form $G(de,d,n)$ for positive integers $d,n$ and $e$. We use the same notation as in \cite{shi2005}.
For $\sigma\in S_n$, denote by $[(x_1,x_2,\dots,x_n)\mid\sigma]$ the $n\times n$ monomial matrix with non-zero entries $x_i$ in the $i,\sigma(i)$ positions. Then
$$G(de,d,n):=\left\{[(x_1,x_2,\ldots ,x_n)|\sigma]\,:\,
x_i\in \mathbb{C}^*, x_i^{de}=1, \Big(\prod_{j=1}^n x_j\Big)^e=1, \sigma\in S_n\right\}.$$

Note the close relationship to Weyl groups, which can be seen as a special case of the above. In particular, note that $W(A_n)=G(1,1,n)$, $W(B_n)=G(2,1,n)$ and $W(D_n)=G(2,2,n)$.

Similarly to the real case above, one can construct the \textit{braid group}, denoted by $B(de,d,n)$, associated to the complex reflection group $G(de,d,n)$. This is defined as the fundamental group
\begin{align*}
    \pi_1\left(\left(\mathbb{C}^n-\bigcup_{s\in\Sigma}H_s\right)/G(de,d,n)\right),
\end{align*}
where $\Sigma$ is the set of pseudo-reflections in $G(de,d,n)$.
See \cite[Tables 1,2 and 5]{BMR98} for presentations of both $G(de,d,n)$ and $B(de,d,n)$.
Similarly to the real case, $G(de,d,n)$ is a quotient of $B(de,d,n)$, obtained by making all generators of finite order. For general $d$ and $e$ some generators have order larger than $2$, while for $e=1$ they all have order $2$.

Moreover, note that $B(1,1,n)=\mathcal{A}(A_n)$, $B(d,1,n)=\mathcal{A}(B_n)$ for any $d\geq 2$ and $B(2,2,n)=\mathcal{A}(D_n)$. In particular, \cite[Theorem~1.1]{A} states that $B(2,2,n)$ is isomorphic to a subgroup of order $2$ of $Z_n(\mathcal{O}_2)$. Let $d\geq2$ be an integer and $\mathcal{O}_d$ be the orbifold $\mathbb{C}/C_d$.  In our first main result, we generalise Allcock's inclusion of groups to the case of arbitrary $d$, and fit it into a commutative diagram.

\begin{mainthm}\label{mainthmA}
There is a commutative diagram of group homomorphisms
\begin{figure}[H]
\[
\xymatrix@C=6em{
N=\langle st_2s^{-1},t_2,t_3,...,t_n \rangle
\ar@{->}[r]^-{\varphi}_-\cong 
\ar@{^{(}->}[d]^-{index \; d} _-{\alpha}
                    & B(d,d,n) 
                    \ar@{^{(}->}[d]_-{\beta}^-{index \;d\; }
                       \\
{\frac{\AAA(B_n)}{\langle s^d =e\rangle}} 
\ar@{->}[r]^-{\gamma}_-{\cong} 
            & Z_n(\MO_d) 
}
\]
\caption{Commutative diagram of group homomorphisms}
\label{fig:maindiagram}
\end{figure}
where $\gamma$, $\varphi$ are isomorphisms and $\alpha$, $\beta$ are monomorphisms,
and $\AAA(B_n)$ is the Artin braid group of type $B_n$ with presentation by generators:
\begin{align*}
\begin{tikzpicture}[scale=0.8,
    vertex/.style={black},
  db/.style={thick, double, double distance=1.3pt, shorten <=-6pt}
  ]
\node[vertex, label=
{[label distance=-5pt]90:{$s$}}] (w8) at  (-0.5,0) {$\bullet$};
 \node[vertex, label=
{[label distance=-5pt]90:{$t_2$}}] (w6) at  (1.5,0) {$\bullet$};
 \node[vertex, label=
{[label distance=-5pt]90:{$t_3$}}] (w5) at  (3.5,0) {$\bullet$};
 \node[vertex, label=
{[label distance=-5pt]90:{$t_{n-1}$}}] (w4) at  (5,0) {$\bullet$};
 \node[vertex, label=
{[label distance=-5pt]90:{$t_n$}}] (w3) at  (7,0) {$\bullet$};
 \draw[dotted, thick] (w5)--(w4);
 \draw (w4)--(w3);
 \draw (w5)--(w6);
 \draw[db] (0,0)--(w6);
       \end{tikzpicture}
\end{align*}

That is, $\AAA(B_n)=\langle s,t_2,t_3,\dots, t_n \mid R \rangle$, where $R$ is the set of relations $t_it_{i+1}t_i=t_{i+1}t_it_{i+1}$ for $2\leq i\leq n-1$, $t_it_j=t_jt_i$ if $|i-j|>1$ and $st_2st_2=t_2st_2s$.
\end{mainthm}

See Section~\ref{section_commdiagram} for more details of the groups and the morphisms, and the proof of the theorem.

We note that Theorem A also follows from~\cite[Cor.\ 5.7]{F23A} and~\cite[Theorem A(2)]{F23B}. The proof here was obtained independently. It is more direct (for this special case), avoiding use of the mapping class group and giving a construction specifically related to the approach of~\cite{BMR98} (i.e.\ more in the style of~\cite{A}).

Our second main result generalises \cite[Theorem~5.4]{BM} and \cite[Theorem~A]{GM} for the groups $G(d,d,n)$ and $B(d,d,n)$ with $d\geq 2$, where the case $d=2$ recovers the classical results. Consider the marked surface $(X, M)$, where $X$ is the disk $S$ with an interior marked point interpreted as a cone point of degree $d$. Note that the interior of $X$ is isomorphic to $\mathcal{O}_d$ as an orbifold. Let $M$ be a set of $n$ marked points on the boundary of $X$.

In Section~\ref{subsection_associating_quiver}, we associate a decorated quiver $Q_T$ to any tagged triangulation $T$ of $(X,M)$ and a group $G_{Q_T}$ to $Q_T$. In Section~\ref{subsection_mutation_q_triang}, we introduce a mutation rule for such a quiver with respect to a chosen vertex, which corresponds to flipping the associated triangulation.

In particular, the initial triangulation $T_0$, illustrated in Figure~\ref{fig:BMR_embedded}, has associated quiver $Q_{T_0}$, which is an orientation of the presentation of $B(d,d,n)$ from \cite[Table~5]{BMR98}. Proving that at each mutation step we obtain an isomorphic group, and using the fact that the flipping graph of (tagged) triangulations of $(X,M)$ is connected, we obtain the following result, providing a family of new presentations of the groups $B(d,d,n)$ and $G(d,d,n)$.

\begin{mainthm} \textbf{(=Theorems \ref{thm:Bddnpresentation} and \ref{thm:Gddnpresentation}).}
Let $T$ be a tagged triangulation of $(X,M)$ and let $G'_{Q_T}$ be the group defined in the same way as $G_{Q_T}$ with the additional relations that all generators square to the identity element. Then
\begin{itemize}
    \item $G_{Q_T}\cong B(d,d,n)$ and $G_{Q_T}$ gives a presentation of $B(d,d,n)$,
    \item $G'_{Q_T}\cong G(d,d,n)$ and $G'_{Q_T}$ gives a presentation of $G(d,d,n)$.
\end{itemize}
\end{mainthm}

Finally, we re-interpret Theorem~B by assigning explicit elements to the abstract generators in the newfound presentations both for the braid and reflection groups.

We combine the two theorems above to give a geometric interpretation of the new presentations of $B(d,d,n)$.
Similarly to the real case above, each tagged triangulation $T$ of $(X,M)$ has an associated braid graph $D_T$ such that the edges of $D_T$, one for each vertex $i$ in $Q_T$, correspond to braids $\sigma_i$ in $Z_n(\mathcal{O}_d)$. Let $B_T$ be the subgroup of $Z_n(\mathcal{O}_d)$ generated by these braids $\sigma_i$. The following result generalises \cite[Theorem~3.6]{GM} to the case $d\geq 2$, concluding that the group presentation associated to the triangulation $T$ gives a presentation of $B(d,d,n)$ as a subgroup of index $d$ of $Z_n(\mathcal{O}_d)$.
\begin{mainthm} \textbf{(=Theorem \ref{thm_braid_interpretation}.)}
     Let $T$ be a tagged triangulation of $(X,M)$. Then there is an isomorphism from $B_T$ to $G_{Q_{T}}$ taking the braid $\sigma_i$ to the generator $s_i$ of $G_{Q_{T}}$ corresponding to the vertex $i$ in $Q_T$. Furthermore, $B_{T}$ is a subgroup of index $d$ of $Z_n(\mathcal{O}_d)$.
\end{mainthm}

Finally, combining Theorem~B with results from \cite{shi2005}, we assign explicit reflections to the generators of the new presentations of $G(d,d,n)$ as follows.
For the definition of the reflections appearing in the following result, we refer the reader to Section~\ref{s:buildingbeta}.

\begin{mainthm}\textbf{(=Theorem \ref{thm_reflection_interpretation}.)}
    Let $T$ be a tagged triangulation of $(X,M)$ and fix a numbering $1,2,\dots,n$ of the $n$ vertices of $D_T$.
    Associate a reflection $s(e)=s(a,b;c(e))$ to each edge $e$ between vertices $a$ and $b$ in $D_T$, where for the edges appearing in the unique cycle of $D_T$, the integers $c(e)$ have to obey the condition explained in Setup~\ref{setup_reflections}. Then, there is an isomorphism of groups $\nu: G'_{Q_T}\rightarrow G(d,d,n)$ sending the generator of $G'_{Q_T}$ associated to vertex $v$ in $Q_T$ to the reflection associated to the edge in $D_T$ that is the dual of $v$.
\end{mainthm}

\section{The complex braid group \texorpdfstring{$B(d,d,n)$}{Bddn} as a subgroup of \texorpdfstring{$Z_n(\MO_d)$}{Zn(Od)}: Proof of Theorem~\ref{mainthmA}.}\label{section_commdiagram}

The aim of this section is to prove Theorem~\ref{mainthmA}, i.e. to construct the commutative diagram of group homomorphisms in Figure~\ref{fig:maindiagram}. Note that the maps $\alpha$ and $\varphi$ are mainly due to \cite{BMR98}. For the remaining two maps, we proceed with a geometric argument. Details of the groups and presentations appearing in the diagram are given in the following sections.

\subsection{Building \texorpdfstring{$\beta$}{beta}}
\label{s:buildingbeta}

We follow~\cite[\S2]{A}, using the notation from~\cite{shi2005}.
We set $V=\mathbb{C}^n$, and denote the fixed hyperplane of a reflection $s:V\rightarrow V$, by $H_s$.

We use the notation from~\cite{shi2005}.
Let $d,n$ be positive integers.
Let $S_n$ denote the symmetric group of degree $n$,
and $\mathbb{C}^*=\mathbb{C}\setminus \{0\}$.
For $\sigma\in S_n$ and $(x_1,x_2,\ldots, x_n)\in (\mathbb{C}^*)^n$, let $[(x_1,x_2,\ldots ,x_n)]$ denote the $n\times n$ monomial matrix with $x_i$ in the $i,\sigma(i)$ position for $i=1,2,\ldots ,n$. The entries in such a matrix are powers of $\omega_d=e^{2\pi i/d}$.

Let $\Gamma(d,n)$ denote the group of all such matrices, and let $G(d,d,n)$ denote the complex reflection group:
$$G(d,d,n)=\left\{[(x_1,x_2,\ldots ,x_n)|\sigma]\,:\,
x_i\in \mathbb{C}^*, x_i^d=1, \prod_{j=1}^n x_j=1, \sigma\in S_n\right\},$$
which is a normal subgroup of $\Gamma(d,n)$ of index $d$.
For $1\leq a<b\leq n$ and $0\leq c\leq d-1$, set
$$s(a,b;c)=[(1,\ldots ,1,\omega_d^{-k},1,\ldots ,1,
\omega_d^k,1,\ldots ,1)|(a,b)],$$
and for $a>b$, set $s(a,b;c)=s(b,a,-c)$.
Thus, for $a<b$,
$$s(a,b;c)(z_1,\ldots, z_n)=
(z_1,\ldots ,z_{i-1},\omega_d^{-k}z_b,\ldots ,\omega_d^k z_a,\ldots ,z_n).$$
These elements all have order two and, as remarked in~\cite[2.1]{shi2005}, they constitute the reflections in $G(d,d,n)$; we denote this set by $\Sigma$. Thus $G(d,d,n)$ is generated by $\Sigma$, the set of reflections it contains.

\begin{remark}\label{hypeplane_line}
    For $i<j$, the reflection $s(a,b;c)$ fixes the \textit{hyperplane}
    \begin{align*}
        H_{s(a,b;c)}=H(a,b;c)=\{(z_1,\ldots ,z_n)\in \mathbb{C}^n\,:\,
z_a=\omega_d^{-k}z_b\}=\ker(s(a,b;c)-id_{\mathbb{C}^n}).
    \end{align*}
    Moreover, $s(a,b;c)$ has associated \textit{hyperline}
    \begin{align*}
        L_{s(a,b;c)}=L(a,b;c)=\text{im}(s(a,b;c)-id_{\mathbb{C}^n}).
    \end{align*}
    Note that $\mathbb{C}^n=H_{s(a,b;c)}\oplus L_{s(a,b;c)}$, and so each element $x\in\mathbb{C}^n$ can be written uniquely as $x=x_H + x_L$ with $x_H\in H_{s(a,b;c)}$ and $x_L\in L_{s(a,b;c)}$.
\end{remark}

Let
$$V_0=V\setminus \cup_{s\in \Sigma} H_s.$$
As noted in~\cite[2.1]{A} for the $D_n$ case, $V_0$ is connected since each $H_s$ has real codimension $2$ in $V$. 

It is well-known that $G(d,d,n)$ acts freely on $V_0$ (see e.g.~\cite{garnier23}). Since $G(d,d,n)$ is finite, it acts properly discontinuously on $V_0$, and it is clear the action is smooth, so we can form the quotient manifold $V_0/G(d,d,n)$ (by e.g.~\cite[7.10]{hatcher02}).
Let $p:V_0\rightarrow V_0/G(d,d,n)$ be the canonical surjection, which is a manifold covering map~\cite[Thm.\ 21.13]{lee13}. Choose $x_0\in V_0$. Then the fundamental group $\pi_1(V_0/G(d,d,n),p(x_0))$ is known as the \textbf{braid group} of $G(d,d,n)$ and denoted $B(d,d,n)$ (see~\cite[2B]{BMR98}).

The cyclic group $C_d$ acts on $\mathbb{C}$, with a generator sending $z$ to $\omega_d z$. Let $\MO_d$ be the orbifold $\mathbb{C}/C_d$. The underlying space of $\mathbb{C}/C_d$ is $\mathbb{C}$, and it has a single cone point of degree $d$ at the origin.
The $n$-strand \textbf{pure braid space} of $\MO_d$ is
$\MO_d^n-\Delta_n$, where
$$\Delta_n=\{(z_1,z_2,\ldots ,z_n)\in \MO_d^n\,:\,z_i=z_j\text{ for some }i\not=j\}.$$
The symmetric group $S_n$ of degree $n$ acts freely on $\MO_d^n-\Delta_n$ and we can form the quotient 
$$X_n=(\MO_d^n-\Delta_n)/S_n,$$
which is the $n$-strand \textbf{braid space} of $\MO_d$. Then the $n$-strand \textbf{braid group} $Z_n(\MO_d)$ of $\MO_d$ is the orbifold fundamental group (in the sense of~\cite[Defn.\ 13.2.5]{thurston22}) of $X_n$ with respect to a choice of basepoint $b=(b_1,\ldots ,b_n)\in X_n$ which does not lie on the orbifold locus. The $n$-strand \textbf{pure} braid group $P_n(\MO_d)$ of $\MO_d$ is the orbifold fundamental group of $\Delta_n$.

\begin{proposition}\label{prop_d-foldcovering}
There is an isomorphism of orbifolds $V/C_d^n\cong \MO_d^n$
given by $(x_1,x_2,\ldots ,x_n)\mapsto (x_1^d,x_2^d,\ldots ,x_n^d)$.
This induces an isomorphism $\varphi:V_0/\Gamma(d,n)\cong X_n$ and hence a $d$-fold orbifold covering map $V_0/G(d,d,n)\rightarrow X_n$ and an embedding $\beta$ of $B(d,d,n)=\pi_1(V_0/G(d,d,n))$ as a subgroup of index $d$ in $Z_n(\MO_d)=\pi_1(X_n)$.
\end{proposition}
\begin{proof}
Recall that
$$V_0=\{(z_1,z_2,\ldots,z_n)\in \mathbb{C}^n\,:\,z_i\not=\omega_d^k z_j,\ \text{ for all }i\not=j\text{ and }0\leq k\leq d-1\}.$$
The group $G(d,d,n)$ has a normal subgroup isomorphic to $C_d^{n-1}$ consisting of the elements where the permutation is the identity, that is, using the notation $\omega_d:=e^{2\pi i/d}$, the elements of the form
\begin{align*}
[(\omega_d^{k_1},\omega_d^{k_2},\dots, \omega_d^{k_n})\mid id]
\end{align*}
where $k_1,k_2,\dots,k_n\in\{0,1,\dots, d-1\}$ satisfy $k_1+k_2+\dots +k_n=0 \text{ modulo } d$. Note that $k_1,k_2,\dots,k_{n-1}$ can be chosen freely and they determine $k_n$. It is easy to see that
$G(d,d,n)=C_d^{n-1}\rtimes S_n$.
The rest of the argument goes through as in~\cite[Proof of Thm.\ 1.1]{A}.
\end{proof}

We use the same generating set of reflections for $G(d,d,n)$ as \cite[pp~151-152]{BMR98} with the following notation. Note that these give a presentation of $G(d,d,n)$ by \cite[Proposition~3.2]{BMR98}.
\begin{notation}\label{notation_generators}
   We set
     \begin{align*}
     t'_2&:=s(1,2;1): (z_1,z_2,z_3, z_4\dots, z_n)\mapsto (e^{-2\pi i/d}z_2,e^{2\pi i/d} z_1,z_3,z_4\dots , z_{n-2}, z_{n-1}, z_n), \\
     t_2&:=s(1,2;0): (z_1,z_2,z_3, z_4\dots, z_n)\mapsto (z_2,z_1,z_3,z_4\dots , z_{n-2}, z_{n-1}, z_n), \\
     t_3&:=s(2,3;0): (z_1,z_2,z_3, z_4\dots, z_n)\mapsto (z_1,z_3,z_2,z_4\dots, z_{n-2}, z_{n-1}, z_n), \\
     &\vdots\\
     t_n&:=s(n-1,n;0): (z_1,z_2,z_3, z_4\dots, z_n)\mapsto (z_1,z_2,z_3,z_4\dots, z_{n-2}, z_{n}, z_{n-1}).
     \end{align*}
     In other words, we are taking $n$ reflections: $s(a-1,a;0)$ for $a\in\{1,2,\dots,n\}$ and $s(1,2;1)$ and renaming them as above. As pointed out in \cite{shi2005}, these are all reflections of type I and hence they have order $2$ and they lie in $G(d,d,n)$.
\end{notation}

We next need some paths in $V_0$ as defined in~\cite{BMR98}.

\begin{defn} \label{def:perturbedpath}
\cite[\S B, Eq.\ (2.13)]{BMR98}.
Let $s\in G(d,d,n)$ be a reflection. Note that $s$ has order $2$. Let $x\in V_0$, with decomposition $x=x_H+x_L$ with $x_H\in H_s$ and $x_L\in L_s$, as in Remark~\ref{hypeplane_line}. Then $y=s(x)=x_H-x_L$. Then the straight path $\tilde{p}_s:[0,1]\rightarrow V$ sending $t$ to $x_H+(1-2t)x_L$ does not lie in $V_0$, since $\tilde{p}_s(1/2)\in H$. So we take instead a version of the path which is perturbed close to $H$:
$$p_s(t)=\begin{cases} x+t(y-x)=x_H+(1-2t)x_L, & 0\leq t\leq \frac{1}{3}; \\
x+(\frac{1}{2}-\frac{1}{6}e^{3i\pi(t-\frac{1}{3})})(y-x)=
x_H+\frac{1}{3}e^{3i\pi(t-1/3)}x_L, & \frac{1}{3}\leq t\leq \frac{2}{3}; \\
x+t(y-x)=x_H+(1-2t)x_L, & \frac{2}{3}\leq t\leq 1,
\end{cases}
$$
which is a special case of the construction
in~\cite[\S B, Eq.\ (2.13)]{BMR98}.
\end{defn}

\begin{remark}
\label{rem:perturbedentries}
In Definition~\ref{def:perturbedpath}, the $i$th entry of $p_s(t)$ is either constant (if $x_i=y_i$) or a path from $x_i$ to $y_i$ which is the first third of the straight path in $\mathbb{C}$ from $x_i$ to $y_i$ ending at $x_i+\frac{1}{3}(y_i-x_i)=2$, followed by an anticlockwise semicircle of radius $\frac{1}{6}|y_i-x_i|$ centred at the mid-point between $x_i$ and $y_i$ and ending at $x_i+\frac{2}{3}(y_i-x_i)$, followed by the last third of the straight path from $x_i$ to $y_i$.
\end{remark}

The following result is an instance of  \cite[Proposition~3.2 and Theorem~2.27]{BMR98}.

\begin{proposition}\label{prop_BMR_Bddn_presentation}
    The set $\{t_2',t_2,\dots,t_n\}$ together with the relations described in \cite[Appendix~2 and Table~2]{BMR98} give a presentation by generators and relations of $G(d,d,n)$.
    Moreover, for $s$ equal to, respectively, $t_2',t_2,t_3\dots,t_n$,
    the paths $p_s$, regarded as paths in $V_0/G(d,d,n)$,
    are $s$-generators of the monodromy, 
    denoted respectively by $\tau_2',\tau_2,\tau_3,\dots,\tau_n$, giving the presentation by generators and relations of $B(d,d,n)$ illustrated in Figure~\ref{fig:BMR_presentation_Bddn}, with relations:
    \begin{align*}
    &\tau_i\tau_{i+1}\tau_i=\tau_{i+1}\tau_i\tau_{i+1},\quad \tau_2'\tau_3\tau_2'=\tau_3\tau_2'\tau_3, \quad
    \tau_i\tau_j=\tau_j\tau_i \text{ for } |i-j|>1,\quad
    \tau_2'\tau_i=\tau_i\tau_2' \text{ for } i>3\\
    &\tau_2\tau_2'\tau_3\tau_2\tau_2'\tau_3=\tau_3\tau_2\tau_2'\tau_3\tau_2\tau_2', \quad
    \underbrace{\tau_2\tau_2'\cdots }_{\text{$d$ terms}} =\underbrace{\tau_2'\tau_2\cdots }_{\text{$d$ terms}}.
    \end{align*}
   \begin{figure}[H]
       \centering
       \begin{tikzpicture}[scale=0.8,
  quiverarrow/.style={black, -latex},
  iquiverarrow/.style={black, -latex, <-},
  mutationarc/.style={dashed, red, very thick},
  arc/.style={dashed, black},
  point/.style={gray},
  vertex/.style={black},
  conepoint/.style={gray, circle, draw=gray!100, fill=white!100, thick, inner sep=1.5pt},
  db/.style={thick, double, double distance=1.3pt, shorten <=-6pt}
  ]
           \node[vertex, label=
{[label distance=-5pt]90:{$\tau_2$}}] (w8) at  (0,1) {$\bullet$};
 \node[vertex, label=
{[label distance=-5pt]-90:{$\tau_2'$}}] (w7) at  (0,-1) {$\bullet$};
 \node[vertex, label=
{[label distance=-5pt]90:{$\tau_3$}}] (w6) at  (1.5,0) {$\bullet$};
 \node[vertex, label=
{[label distance=-5pt]90:{$\tau_4$}}] (w5) at  (3.5,0) {$\bullet$};
 \node[vertex, label=
{[label distance=-5pt]90:{$\tau_{n-1}$}}] (w4) at  (5,0) {$\bullet$};
 \node[vertex, label=
{[label distance=-5pt]90:{$\tau_n$}}] (w3) at  (7,0) {$\bullet$};
 \draw[dotted, thick] (w5)--(w4);
 \draw (w4)--(w3);
 \draw (w5)--(w6);
 \draw (w7)--(w6);
 \draw (w8)--(w6);
 \draw[-] (w8)--(w7);
 \draw[db] (1.1,0)--(0.7,0);
 \node[left] at  ($(w8)! 0.5!(w7)$) {\tiny $d$};
       \end{tikzpicture}
       \caption{The presentation of $B(d,d,n)$ from \cite{BMR98}.}
       \label{fig:BMR_presentation_Bddn}
   \end{figure}
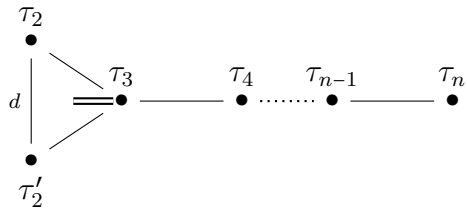
\end{proposition}

\begin{remark}
\label{rem:RS}
We recall part of the proof of Proposition~\ref{prop_BMR_Bddn_presentation} from~\cite[Section 3]{BMR98}.
Recall the standard presentation of the Artin braid group of type $A_{n+1}$, that is $\mathcal{A}(A_{n+1})$ or in \cite{BMR98} notation $B(n+1)$, is:
\begin{align*}
\begin{tikzpicture}[scale=0.7,
  vertex/.style={black},
  db/.style={thick, double, double distance=1.3pt, shorten <=-6pt}
  ]
 \node[vertex, label=
{[label distance=-5pt]90:{$\xi_1$}}] (w6) at  (1.5,0) {$\bullet$};
 \node[vertex, label=
{[label distance=-5pt]90:{$\xi_2$}}] (w5) at  (3.5,0) {$\bullet$};
 \node[vertex, label=
{[label distance=-5pt]90:{$\xi_{n-1}$}}] (w4) at  (5,0) {$\bullet$};
 \node[vertex, label=
{[label distance=-5pt]90:{$\xi_n$}}] (w3) at  (7,0) {$\bullet$};
 \draw[dotted, thick] (w5)--(w4);
 \draw (w4)--(w3);
 \draw (w5)--(w6);
       \end{tikzpicture}
\end{align*}

By~\cite[Thm.\ 3.6]{BMR98}, the braid group $B(d,1,n)$ associated to the complex reflection group $G(d,1,n)$ is, for any $d>1$, isomorphic to the subgroup of $\mathcal{A}(A_{n+1})$ generated by $\{\xi_1^2, \xi_2,\xi_3,\dots,\xi_n\}$.
By the discussion following the proof of~\cite[Thm.\ 3.6]{BMR98}, an application of the Reidemeister-Schreier method shows that this subgroup is isomorphic to the Artin braid group $\mathcal{A}(B_n)$, with generators associated to the vertices of its Dynkin diagram as follows:
\begin{align*}
\begin{tikzpicture}[scale=0.8,
    vertex/.style={black},
  db/.style={thick, double, double distance=1.3pt, shorten <=-6pt}
  ]
\node[vertex, label=
{[label distance=-5pt]90:{$\xi_1^2$}}] (w8) at  (-0.5,0) {$\bullet$};
 \node[vertex, label=
{[label distance=-5pt]90:{$\xi_2$}}] (w6) at  (1.5,0) {$\bullet$};
 \node[vertex, label=
{[label distance=-5pt]90:{$\xi_3$}}] (w5) at  (3.5,0) {$\bullet$};
 \node[vertex, label=
{[label distance=-5pt]90:{$\xi_{n-1}$}}] (w4) at  (5,0) {$\bullet$};
 \node[vertex, label=
{[label distance=-5pt]90:{$\xi_n$}}] (w3) at  (7,0) {$\bullet$};
 \draw[dotted, thick] (w5)--(w4);
 \draw (w4)--(w3);
 \draw (w5)--(w6);
 \draw[db] (0,0)--(w6);
       \end{tikzpicture}
\end{align*}
Recall that, if $m>1$, by \cite[Lemma~3.3]{BMR98}, the complement in $\mathbb{C}^n$ of the union of the reflecting hyperplanes of $G(md,d,n)$ is
\begin{align*}
    \mathcal{M}^\#(md,n)=\{ (z_1,z_2,\dots,z_n)\mid(\forall j, k, 1\leq j\neq k\leq n)(\forall a\in \mathbb{Z})(z_j\neq 0)(z_j\neq e^{\frac{i\pi a}{md}}z_k) \},
\end{align*}
while if $m=1$, it is
\begin{align*}
    \mathcal{M}(d,n)=\{ (z_1,z_2,\dots,z_n)\mid(\forall j, k, 1\leq j\neq k\leq n)(\forall a\in \mathbb{Z})(z_j\neq e^{\frac{i\pi a}{d}}z_k) \}.
\end{align*}

As remarked in~\cite[Section 3C]{BMR98},
\cite[Proposition 3.8]{BMR98} could be stated in a more general way, obtaining, by an application of the Reidemester-Schreier algorithm, an injective group homomorphism
\begin{align*}
    \phi: \pi_1(\mathcal{M}^\#(md,n)/G(md,d,n))\hookrightarrow \mathcal{A}(B_{n}), 
\end{align*}
where, letting $\xi_2':=\xi_1^2\xi_2\xi_1^{-2}$, the left hand side group has presentation
\begin{align}\label{eqn_presentationBdeer}
    \begin{tikzpicture}[scale=0.7,
  quiverarrow/.style={black, -latex},
  iquiverarrow/.style={black, -latex, <-},
  mutationarc/.style={dashed, red, very thick},
  arc/.style={dashed, black},
  point/.style={gray},
  vertex/.style={black},
  conepoint/.style={gray, circle, draw=gray!100, fill=white!100, thick, inner sep=1.5pt},
  db/.style={thick, double, double distance=1.3pt, shorten <=-6pt}
  ]
           \node[vertex, label=
{[label distance=-5pt]90:{$\xi_2$}}] (w8) at  (0,1) {$\bullet$};
\node[vertex, label=
{[label distance=-5pt]180:{$\xi_1^{2d}$}}] (w9) at  (-1,0) {$\bullet$};
 \node[vertex, label=
{[label distance=-5pt]-90:{$\xi_2'$}}] (w7) at  (0,-1) {$\bullet$};
 \node[vertex, label=
{[label distance=-5pt]90:{$\xi_3$}}] (w6) at  (1.5,0) {$\bullet$};
 \node[vertex, label=
{[label distance=-5pt]90:{$\xi_4$}}] (w5) at  (3.5,0) {$\bullet$};
\node[vertex, label=
{[label distance=-5pt]90:{$\xi_{n-1}$}}] (w4) at  (5,0) {$\bullet$};
 \node[vertex, label=
{[label distance=-5pt]90:{$\xi_n$}}] (w3) at  (7,0) {$\bullet$};
 \draw[dotted, thick] (w5)--(w4);
 \draw (w4)--(w3);
 \draw (w5)--(w6);
 \draw (w7)--(w6);
 \draw (w8)--(w6);
 \draw[db] (1.1,0)--(0.7,0);
 \node at  (-1,-1) {\tiny $d+1$};
 \draw([shift=(90:1cm)]0,0) arc (90:270:1cm); 
       \end{tikzpicture}
\end{align}
In particular, if $m>1$, the left hand side is isomorphic to $B(md,d,n)$, but here we are interested in the case $m=1$. By \cite[Section 3C]{BMR98}, in this case there is an isomorphism of groups
\begin{align*}
    B(d,d,n)\cong  {\frac{\pi_1(\mathcal{M}^\#(d,n)/G(d,d,n))}{\langle \xi_1^{2d} =e\rangle}}
\end{align*}
and the presentation of $B(d,d,n)$ is obtained from (\ref{eqn_presentationBdeer}) by suppressing the node corresponding to $\xi_1^{2d}$ and adding an edge labelled $d$ between $\xi_2$ and $\xi_2'$.
\end{remark}

\begin{remark} \label{rem:BMRbraidgroup}
Note that \cite[Section~3]{BMR98} uses
the base point $(x_1,x_2,x_3\dots,x_{n+1})$ with $x_1<x_2<x_3<\dots<x_{n+1}$ real numbers. Instead, we choose basepoint $(0,-1,1,2,\dots, n-1)$.
We also modify
the $(2\ 3)$-generator of the monodromy $\xi_2$ (where $(2\ 3)$ is a generating transposition of $S_{n+1}$), taking the path as in Figure~\ref{fig:fig_path2_power} rather than as in~\cite[Section 3]{BMR98}: the arguments of~\cite{BMR98} go through unchanged with these choices. See Remark~\ref{rem:choice} for an explanation of this choice.
\end{remark}

\begin{proposition}\label{prop_paths_drawn}
The injective group homomorphism $\beta$ from $B(d,d,n)$ to $Z_n(\MO_d)$ from Proposition is given by:
    \begin{align*}
        \beta: B(d,d,n)\rightarrow Z_n(\MO_d): \tau_2'\mapsto h_1, \tau_i\mapsto h_i \text{ for } 2\leq i\leq n,
    \end{align*}
    where $\tau_2'$, $\tau_2$, $\tau_3,\dots$, and $\tau_n$ are as in the presentation in Proposition~\ref{prop_BMR_Bddn_presentation}
    and the braids $h_1,h_2,\dots,h_n$ are as illustrated in Figure~\ref{fig:braids_hi}.
    \end{proposition}
\begin{proof}
Let us fix the basepoint in $V_0$ to be $b=(e^{-i\pi/d},1,2,\dots,n-1)$. We compute the paths $p_s$ associated to each of the reflections $s$ from Notation~\ref{notation_generators} following Definition~\ref{def:perturbedpath}.

Note that $t_2(b)=(1, e^{-i\pi/d},2,\dots, n-1)$.
By Remark~\ref{rem:perturbedentries}, the first entry in the path 
$p_{t_2}$ from Definition~\ref{def:perturbedpath} is the first third of the straight path from $e^{-i\pi/d}$ to $1$, followed by a semicircular path centred at the midpoint between $e^{-i\pi/d}$ 
and $1$ of radius $\frac{1}{6}|1-e^{-i\pi/d}|$, followed by the 
last third of the straight path from $e^{-i\pi/d}$ to $1$.
The second entry is a similarly-defined path going in the other 
direction, while all the other entries are constant.
We sketch this, for $d=3$, in Figure~\ref{fig:fig_path1}.

We have $t'_2(b)=(e^{-2\pi i/d},e^{i\pi/d},2,\ldots ,n-1)$, so
the first entry of $p_{t'_2}$ is a path from $e^{-i\pi/d}$ to 
$e^{-2\pi i/d}$, while the second entry is a path from
$1$ to $e^{i\pi/d}$ and the other entries remain constant. We 
sketch this, for $d=3$, in Figure~\ref{fig:fig_path2}.

For $3\leq j\leq n$, we have $t_j(b)=(e^{-i\pi/d},1,\ldots ,j-3,j-1,j-2,j,\ldots ,n-1)$, so the $(j-1)$st entry of $p_{t_j}$ is a path from $j-2$ to $j-1$ and the $j$th entry of $p_{t_j}$ is a path from $j-1$ to $j-2$.
We sketch this, for $d=3$, in Figure~\ref{fig:fig_pathi}.

By Proposition~\ref{prop_BMR_Bddn_presentation}, for $s$ equal to, respectively, $t_2',t_2,t_3\dots,t_n$, the paths $p_s$,
regarded as paths in $V_0/G(d,d,n)$, are $s$-generators of the monodromy, denoted respectively by $\tau_2',\tau_2,\tau_3,\dots,\tau_n$. Their images under
the covering map $V_0/G(d,d,n)\rightarrow X_n$ in Proposition~\ref{prop_d-foldcovering} are given by taking their $d$th powers, entry by entry.

These images are illustrated in Figures~\ref{fig:fig_path1_power},\ref{fig:fig_path2_power} and \ref{fig:fig_pathi_power}, for $d=3$. Note also that the basepoint for $\MO_d^n-\Delta_n$ is $(-1,1,2^d,\dots, (n-1)^d)$.
It is easy to see, by writing the paths in polar coordinates in $\mathbb{C}$, that the paths will be in the same half plane (i.e. above or below the real axis) as shown in these figures, for general $d$.
Drawing these paths as braids using the same conventions as in \cite{A}, that is placing the point of view from below, and recalling that the cone point of order $d$ at $0$ is interpreted as a pole of order $d$, we see that the elements $\tau_2',\tau_2,\tau_3,\dots,\tau_n$
are sent respectively to the braids $h_1,h_2,h_3,\dots,h_n$ illustrated in Figure~\ref{fig:braids_hi}.
Hence, the injective group homomorphism $\beta$ from
Proposition~\ref{prop_d-foldcovering} sends $B(d,d,n)$ to the subgroup $\langle h_1,h_2,h_3,\dots,h_n\rangle$ of $Z_n(\MO_d)$ of index $d$, as claimed.
\end{proof}

\begin{figure}[H]
  \centering
    \begin{tikzpicture}[scale=2.5]\pgfmathsetmacro{\length}{2}
      \draw (0,0) node{$\star$};
      \draw (0,0.15) node{$0$};
      \coordinate (A) at (-60:1cm) ;
\coordinate (B) at (1,0);
\coordinate (stop1) at ($(A)!0.33333!(B)$);
\coordinate (mid) at ($(A)!0.5!(B)$);
\coordinate (stop2) at ($(A)!0.66666!(B)$);
\draw[thick, Green] ([shift=(60:0.17cm)]0.75,-0.4330127) arc (60:240:0.17cm); 
\draw[->, thick, Green] ([shift=(60:0.17cm)]0.75,-0.4330127) arc (60:200:0.17cm); 
\draw[thick, red] ([shift=(240:0.17cm)]0.75,-0.433) arc (240:410:0.17cm); 
\draw[->,thick, red] ([shift=(240:0.17cm)]0.75,-0.433) arc (240:330:0.17cm); 
\path[name path=line,draw=red, thick] (A) node {$\bullet$} -- (stop1)  node {\textcolor{gray}{$\bullet$}};
\path[name path=line,draw=red, thick] (stop2) node {$\bullet$} -- (B)  node {\textcolor{gray}{$\bullet$}};
\path[dashed, name path=line,draw=Green, thick] (A) node {$\bullet$} -- (stop1)  node {\textcolor{gray}{$\bullet$}};
\path[dashed,name path=line,draw=Green, thick] (stop2) node {\textcolor{gray}{$\bullet$}} -- (B)  node {$\bullet$};
    \draw[Green] (1,0.15) node{$1$};
    \draw[red] (-60:1.15cm) node{$e^{-i\pi/3}$};
    \draw (2,0.15) node{$2$};
    \draw (2,-0) node{$\bullet$};
    \draw (3,-0) node{$\dots$};
    \draw (4,0.15) node{$n-1$};
    \draw (4,-0) node{$\bullet$};
    \end{tikzpicture}
    \caption{The path $p_{t_2}$ associated to the reflection $t_2$. Note that only the first two entries of the $n$-tuple are non-constant paths. The drawing is for $d=3$. For larger $d$, the endpoints of the paths are closer.}
    \label{fig:fig_path1}
\end{figure}
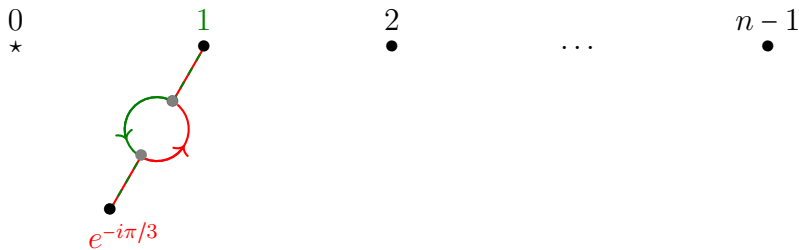

\begin{figure}[H]
  \centering
    \begin{tikzpicture}[scale=2.5]\pgfmathsetmacro{\length}{2}
      \draw (0,0) node{$\star$};
      \draw (0,0.15) node{$0$};
      \coordinate (A) at (60:1cm) ;
\coordinate (B) at (-120:1cm);
\coordinate (C) at (1,0) ;
\coordinate (D) at (-60:1cm);
\coordinate (stop1) at ($(D)!0.33333!(B)$);
\coordinate (mid) at ($(D)!0.5!(B)$);
\coordinate (stop2) at ($(D)!0.66666!(B)$);
\coordinate (stop3) at ($(C)!0.33333!(A)$);
\coordinate (mid2) at ($(C)!0.5!(A)$);
\coordinate (stop4) at ($(C)!0.66666!(A)$);
\draw[thick, Green] ([shift=(-60:0.17cm)]0.75,0.4330127) arc (-60:110:0.17cm); 
\draw[->, thick, Green] ([shift=(-60:0.17cm)]0.75,0.4330127) arc (-60:70:0.17cm); 
\draw[thick, red] ([shift=(0:0.17cm)]0,-0.86602) arc (0:180:0.17cm); 
\draw[->,thick, red] ([shift=(0:0.17cm)]0,-0.86602) arc (0:90:0.17cm); 
\path[name path=line,draw=red, thick] (D) node {$\bullet$} -- (stop1)  node {\textcolor{gray}{$\bullet$}};
\path[name path=line,draw=red, thick] (stop2) node {\textcolor{gray}{$\bullet$}} -- (B)  node {$\bullet$};
\path[name path=line,draw=Green, thick] (C) node {$\bullet$} -- (stop3)  node {\textcolor{gray}{$\bullet$}};
\path[name path=line,draw=Green, thick] (stop4) node {\textcolor{gray}{$\bullet$}} -- (A)  node {$\bullet$};
    \draw[Green] (1.03,0.15) node{$1$};
    \draw (60:1.15cm) node{$e^{i\pi/3}$};
    \draw (-120:1.15cm) node{$e^{-i2\pi/3}$};
    \draw[red] (-60:1.15cm) node{$e^{-i\pi/3}$};
    \draw (2,0.15) node{$2$};
    \draw (2,-0) node{$\bullet$};
    \draw (3,-0) node{$\dots$};
    \draw (4,0.15) node{$n-1$};
    \draw (4,-0) node{$\bullet$};
    \end{tikzpicture}
    \caption{The path $p_{t_2'}$ associated to the reflection $t_2'$. Note that only the first two entries of the $n$-tuple are non-constant paths. The drawing is for $d=3$. For larger $d$, the red path lies entirely within the fourth quadrant of the plane and the endpoints of both paths are closer.}
    \label{fig:fig_path2}
\end{figure}
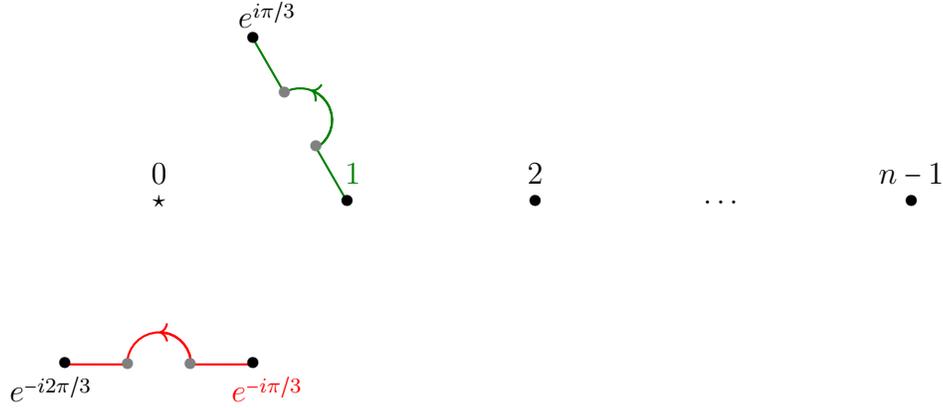

\begin{figure}[H]
  \centering
    \begin{tikzpicture}[scale=2.5]\pgfmathsetmacro{\length}{2}
      \draw (0,0) node{$\star$};
      \draw (1.5,-0) node{$\dots$};
      \draw (0,0.15) node{$0$};
      \draw (1,-0) node{$\bullet$};
      \draw (1,0.15) node{$1$};
      \coordinate (A) at (2,0) ;
\coordinate (B) at (3,0);
\coordinate (stop1) at ($(A)!0.33333!(B)$);
\coordinate (mid) at ($(A)!0.5!(B)$);
\coordinate (stop2) at ($(A)!0.66666!(B)$);
\draw[thick, Green] ([shift=(0:0.17cm)]2.5,0) arc (0:180:0.17cm);
\draw[->,thick, Green] ([shift=(0:0.17cm)]2.5,0) arc (0:90:0.17cm); 
\draw[thick, red] ([shift=(180:0.17cm)]2.5,0) arc (180:360:0.17cm);
\draw[->,thick, red] ([shift=(180:0.17cm)]2.5,0) arc (180:270:0.17cm); 
\path[name path=line,draw=red, thick] (A) node {$\bullet$} -- (stop1)  node {$\bullet$};
\path[name path=line,draw=red, thick] (stop2) node {$\bullet$} -- (B)  node {$\bullet$};
\path[dashed, name path=line,draw=Green, thick] (A) node {$\bullet$} -- (stop1)  node {\textcolor{gray}{$\bullet$}};
\path[dashed,name path=line,draw=Green, thick] (stop2) node {\textcolor{gray}{$\bullet$}} -- (B)  node {$\bullet$};
    \draw[red] (2,0.15) node{$j-2$};
    \draw (-60:1.15cm) node{$e^{-i\pi/3}$};
    \draw (-60:1.cm) node{$\bullet$};
    \draw[Green] (3,0.15) node{$j-1$};
    \draw (2,-0) node{$\bullet$};
    \draw (3.5,-0) node{$\dots$};
    \draw (4,0.15) node{$n-1$};
    \draw (4,-0) node{$\bullet$};
    \end{tikzpicture}
    \caption{The path $p_{t_j}$ associated to the reflection $t_j$ for $3\leq j\leq n$. Note that only the $(j-1)$st and $j${th} entries of the $n$-tuple (with initial values $j-2$ and $j-1$ respectively) are non-constant paths. We draw the case $d=3$.}
    \label{fig:fig_pathi}
\end{figure}
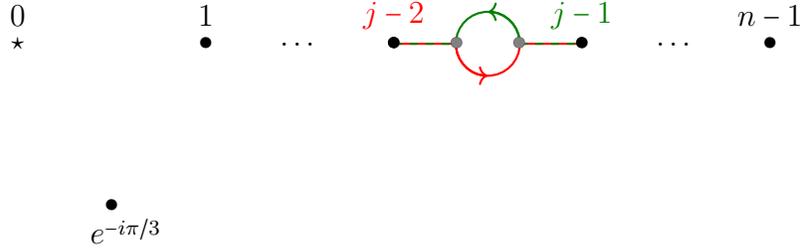

\begin{remark} \label{rem:choice}
Our choice of basepoint, together with the modified choice of $(2\ 3)$-generator for the monodromy for the braid group of type $A$ (see Remark~\ref{rem:BMRbraidgroup}) allows to use the same generating set of reflections for $G(d,d,n)$ as in~\cite[Section 3A]{BMR98} (see Notation~\ref{notation_generators}) while ensuring that the diagram in Figure~\ref{fig:maindiagram} commutes and the images of the generators under $\beta$ in
Proposition~\ref{prop_paths_drawn} are the $h_i$.

If we did not modify the $(2\ 3)$-
generator of the monodromy, we would 
need to replace the reflection 
$t_2'=s(1,2;1)$ with the reflection 
$t_2''=s(1,2;-1)$; by~\cite[Theorem 
2.19]{shi2005}, we would still have a 
generating set of $G(d,d,n)$. We would 
also need to replace the basepoint $b=
(e^{-i\pi/d},1,2,\ldots ,n-1)$ in the 
proof of 
Proposition~\ref{prop_paths_drawn} 
with $(e^{i\pi/d},1,2,\ldots ,n-1)$ 
and the element $\xi_2'$ in 
Remark~\ref{rem:RS} with 
$\xi_1^{-2}\xi_2\xi_1^2$, also 
switching the labels $\xi_2$ and 
$\xi'_2$ on the presentation of 
$B(d,d,n)$ given there. Modifying the 
morphisms $\alpha,\phi$ and $\gamma$ 
in Figure~\ref{fig:maindiagram} 
appropriately (so that 
$\alpha(a_2)=t_2$, 
$\alpha(b_2)=s^{-1}t_2s$, 
$\varphi(a_2)=\tau_2$, 
$\varphi(b_2)=\tau'_2$, 
$\beta(\tau_2)=h_1$, 
$\beta(\tau'_2)=h_2$, $\gamma(t_2)=h_1$ and $\gamma(s^{-1}t_2s)=h_2$),
this would also be a valid construction.
For $d=2$, the paths $p_s$ we consider 
here do not coincide with the 
paths $g_1,\dots g_n$ from \cite[proof 
of Theorem 1.1]{A} (even choosing 
$\epsilon=1$). However, with this 
change, the basepoint would be 
$(i,1,2,\ldots ,n-1)$ as in~\cite{A}, 
and we would recover exactly the same 
paths as in~\cite{A}.
\end{remark}

\begin{figure}[H]
  \centering
    \begin{tikzpicture}[scale=2]\pgfmathsetmacro{\length}{2}
      \draw (0,0) node{$\star$};
      \draw (0,0.2) node{$0$};
      \coordinate (A) at (-1,0) ;
\coordinate (B) at (1,0);
    \draw[red, thick] (A) arc
	[start angle=60,
		end angle=0,
		x radius=1cm,
		y radius =0.7cm
	]  node (stop1) {\textcolor{gray}{$\bullet$}};
    \draw[red, thick] (B) arc
	[start angle=120,
		end angle=180,
		x radius=1cm,
		y radius =0.7cm
	] node (stop2) {\textcolor{gray}{$\bullet$}};
 \draw[Green, thick, dashed] (A) arc
	[start angle=60,
		end angle=0,
		x radius=1cm,
		y radius =0.7cm
	];
    \draw[Green, thick, dashed] (B) arc
	[start angle=120,
		end angle=180,
		x radius=1cm,
		y radius =0.7cm
	];
 \draw[red, thick] (stop1) arc
	[start angle=-120,
		end angle=-60,
		x radius=1cm,
		y radius =6cm
	];
 \draw[red, thick, ->] (stop1) arc
	[start angle=-120,
		end angle=-90,
		x radius=1cm,
		y radius =6cm
	];
  \draw[Green, thick] (stop2) arc
	[start angle=60,
		end angle=120,
		x radius=1cm,
		y radius =3cm
	];
 \draw[Green, thick, ->] (stop2) arc
	[start angle=60,
		end angle=90,
		x radius=1cm,
		y radius =3cm
	];
    \draw (-1,0) node{$\bullet$};
    \draw (1,0) node{$\bullet$};
    \draw[Green] (1.1,-0.2) node{$1$};
    \draw[red] (-1.1,-0.2) node{$-1$};
    \draw (2,-0.2) node{$2^d$};
    \draw (2,-0) node{$\bullet$};
    \draw (3,-0) node{$\dots$};
    \draw (4,-0.2) node{$(n-1)^d$};
    \draw (4,-0) node{$\bullet$};
    \end{tikzpicture}
    \caption{Path $p_{t_2}$ to power $d$.}
    \label{fig:fig_path2_power}
\end{figure}
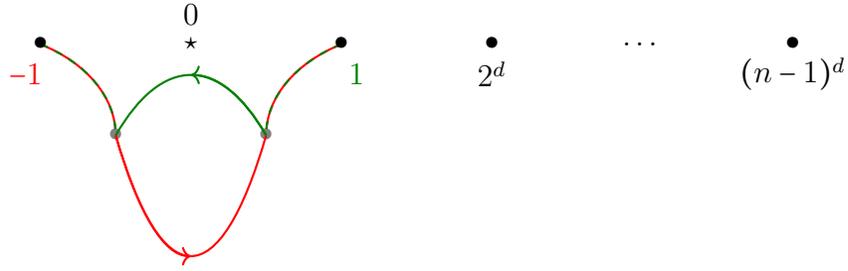

\begin{figure}[H]
  \centering
    \begin{tikzpicture}[scale=2]\pgfmathsetmacro{\length}{2}
      \draw (0,0) node{$\star$};
      \draw (0,-0.2) node{$0$};
      \coordinate (A) at (-1,0) ;
\coordinate (B) at (1,0);
    \draw[red, thick] (A) arc
	[start angle=180,
		end angle=120,
		x radius=1cm,
		y radius =0.7cm
	]  node (stop1) {\textcolor{gray}{$\bullet$}};
    \draw[red, thick] (B) arc
	[start angle=0,
		end angle=60,
		x radius=1cm,
		y radius =0.7cm
	] node (stop2) {\textcolor{gray}{$\bullet$}};
 \draw[Green, thick, dashed] (A) arc
	[start angle=180,
		end angle=120,
		x radius=1cm,
		y radius =0.7cm
	];
    \draw[Green, thick, dashed] (B) arc
	[start angle=0,
		end angle=60,
		x radius=1cm,
		y radius =0.7cm
	];
     \draw[red, thick, ->] (stop1) arc
	[start angle=-120,
		end angle=-80,
		x radius=1cm,
		y radius =3cm
	];
 \draw[red, thick] (stop1) arc
	[start angle=-120,
		end angle=-60,
		x radius=1cm,
		y radius =3cm
	];
  \draw[Green, thick] (stop2) arc
	[start angle=60,
		end angle=120,
		x radius=1cm,
		y radius =6cm
	];
 \draw[Green, thick, ->] (stop2) arc
	[start angle=60,
		end angle=100,
		x radius=1cm,
		y radius =6cm
	];
    \draw (-1,0) node{$\bullet$};
    \draw (1,0) node{$\bullet$};
    \draw[Green] (1,-0.2) node{$1$};
    \draw[red] (-1,-0.2) node{$-1$};
    \draw (2,-0.2) node{$2^d$};
    \draw (2,-0) node{$\bullet$};
    \draw (3,-0) node{$\dots$};
    \draw (4,-0.2) node{$(n-1)^d$};
    \draw (4,-0) node{$\bullet$};
    \end{tikzpicture}
    \caption{Path $p_{t_2'}$ to the power $d$.}
    \label{fig:fig_path1_power}
\end{figure}
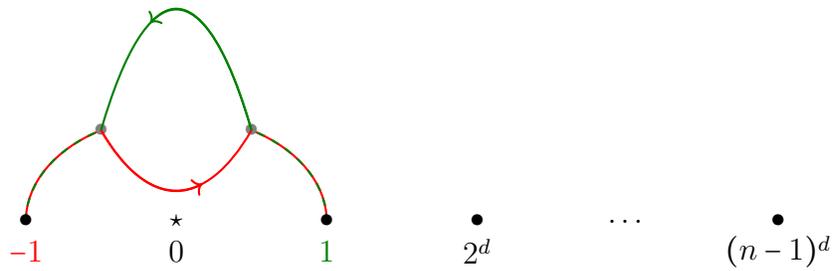

\begin{figure}[H]
  \centering
    \begin{tikzpicture}[scale=2]\pgfmathsetmacro{\length}{2}
      \draw (0,0) node{$\star$};
      \draw (1.5,-0) node{$\dots$};
      \draw (0,-0.2) node{$0$};
      \draw (1,-0) node{$\bullet$};
      \draw (1,-0.2) node{$1$};
      \coordinate (A) at (2,0) ;
\coordinate (B) at (3,0);
\draw[red, thick] (A) arc
	[start angle=180,
		end angle=120,
		x radius=0.5cm,
		y radius =0.7cm
	]  node (stop1) {\textcolor{gray}{$\bullet$}};
    \draw[red, thick] (B) arc
	[start angle=0,
		end angle=60,
		x radius=0.5cm,
		y radius =0.7cm
	] node (stop2) {\textcolor{gray}{$\bullet$}};
 \draw[Green, thick, dashed] (A) arc
	[start angle=180,
		end angle=120,
		x radius=0.5cm,
		y radius =0.7cm
	];
    \draw[Green, thick, dashed] (B) arc
	[start angle=0,
		end angle=60,
		x radius=0.5cm,
		y radius =0.7cm
	];
     \draw[red, thick, ->] (stop1) arc
	[start angle=-120,
		end angle=-80,
		x radius=0.5cm,
		y radius =2cm
	];
 \draw[red, thick] (stop1) arc
	[start angle=-120,
		end angle=-60,
		x radius=0.5cm,
		y radius =2cm
	];
  \draw[Green, thick] (stop2) arc
	[start angle=60,
		end angle=120,
		x radius=0.5cm,
		y radius =4cm
	];
 \draw[Green, thick, ->] (stop2) arc
	[start angle=60,
		end angle=100,
		x radius=0.5cm,
		y radius =4cm
	];
    \draw[red] (2,-0.2) node{$(j-2)^d$};
    \draw (3,0) node{$\bullet$};
    \draw (-1,-0.2) node{$-1$};
    \draw (-1,0) node{$\bullet$};
    \draw[Green] (3,-0.2) node{$(j-1)^d$};
    \draw (2,-0) node{$\bullet$};
    \draw (3.5,-0) node{$\dots$};
    \draw (4,-0.2) node{$(n-1)^d$};
    \draw (4,-0) node{$\bullet$};
    \end{tikzpicture}
    \caption{Path $p_{t_j}$ to the power $d$ for $3\leq j\leq n$.}
    \label{fig:fig_pathi_power}
\end{figure}
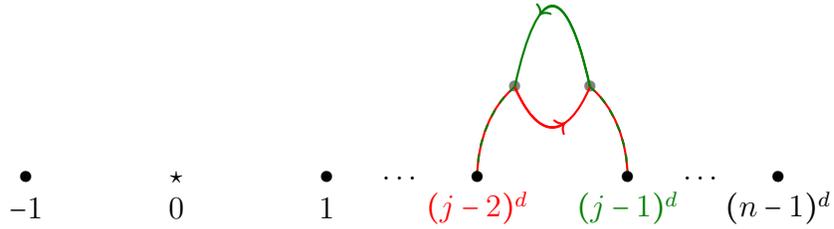

\subsection{Building \texorpdfstring{$\alpha$}{alpha} and \texorpdfstring{$\varphi$}{phi}}

\begin{lemma}\label{lemma_Nnormal_indexd}
The subgroup $N:=\langle st_2s^{-1},t_2,t_3,...,t_n \rangle\subseteq {\frac{\AAA(B_n)}{\langle s^d =e\rangle}}$ is a normal subgroup of index $d$ and
\begin{align*}
        {\frac{\AAA(B_n)}{\langle s^d =e\rangle}}=N\cupdot Ns\cupdot Ns^2\cupdot \dots Ns^{d-1}.
    \end{align*}
\end{lemma}
\begin{proof}
We first prove $N$ is a normal subgroup by showing that its normaliser is the whole ambient group. Note that by construction $st_2 s^{-1}\in N$ and for $i=3,\dots,n$ we have $st_i s^{-1}=t_i\in N$. Moreover, using the relation $st_2st_2=t_2st_2s$, or equivalently $st_2s^{-1}=t_2^{-1}s^{-1}t_2st_2$, we have that
    \begin{align*}
        s(st_2 s^{-1}) s^{-1}=st_2^{-1}s^{-1}t_2st_2s^{-1}=(st_2s^{-1})^{-1}(t_2)(st_2s^{-1}),
    \end{align*}
and this is an element in $N$ since it is a product of three elements in $N$. Hence $sNs^{-1}\subseteq N$. Then the normaliser of $N$ contains both $s$ and all the elements of $N$. In particular it contains the generators $s,t_2,\dots,t_n$ of the ambient group. Hence $N$ is a normal subgroup of ${\frac{\AAA(B_n)}{\langle s^d =e\rangle}}$.

We now show that ${\frac{\AAA(B_n)}{\langle s^d =e\rangle}}$ is the semidirect product of $\langle s \rangle$ and $N$. First we show that $s\not\in N$. Recall that $s^{-1}=s^{d-1}$ and the relations in the group preserve the sum of the exponents of copies of $s$ modulo $d$ in the expression of any element of the group. Then, we have that the sum of the exponents of copies of the element $s$ in the expression for any element in $N$ is always a multiple of $d$. Hence $s\not\in N$ and $N\cap \langle s\rangle =\{ e\}$, where $e$ is the identity element.
    Now, we have
    \begin{align*}
        {\frac{\AAA(B_n)}{\langle s^d =e\rangle}}=N\cupdot Ns\cupdot Ns^2\cupdot \dots Ns^{d-1}
    \end{align*}
    and ${\frac{\AAA(B_n)}{\langle s^d =e\rangle}}=\langle N, \langle s\rangle \rangle$ and $N$ is a normal subgroup, so
    \begin{align*}
        {\frac{\AAA(B_n)}{\langle s^d =e\rangle}}= N \langle s\rangle
    \end{align*}
    and $N$ is a subgroup of index $d$.
\end{proof}

We define $\alpha$ to be the embedding of $N$ into ${\frac{\AAA(B_n)}{\langle s^d =e\rangle}}$.

\begin{theorem}
\label{thm_varphi_iso}
There is a group isomorphism
    \begin{align*}
        \varphi: N\longrightarrow B(d,d,n): st_2s^{-1}\mapsto \tau'_2,\, t_i\mapsto \tau_i, \textup{ for } 2\leq i\leq n.
    \end{align*}
\end{theorem}
\begin{proof}
This can be seen by applying the Reidemester-Schreier algorithm to find a presentation of the subgroup $N$, using the set of coset representatives
$\{e,s,s^2,\dots,s^{d-1}\}$, and noting the presentation of $B(d,d,n)$ in Proposition~\ref{prop_BMR_Bddn_presentation} (from~\cite{BMR98}).
This is very similar to the proof of Proposition~\ref{prop_BMR_Bddn_presentation} in~\cite[Sections 3.7 and 3C]{BMR98}; see Remark~\ref{rem:RS}.
\end{proof}

\subsection{Building \texorpdfstring{$\gamma$}{gamma}}

\begin{figure}[ht]
\begin{tikzpicture}[scale=0.8]
\begin{scope}
    \node at (0,6.4) {\tiny$-1$};
    \node at (2,6.4) {\tiny$1$};
    \node at (1,2.6) {$d$};
    \node at (3,6.4) {\tiny$2^d$};
    \node at (5,6.4) {\tiny$(n-1)^d$};
    \node at (-1,4.5) {$h_1=$};
    \node[font=\huge] at (4,4.5){$\cdots$};
    \draw[overcross](2,4)--(2,3);
    \draw[overcross, looseness=1.3](2,6)to[out=-90,in=90](0,3)--(0,3);
    \draw[overcross, looseness=1.3](2,4)to[out=90,in=-90](0,6);
    \draw[line width=0.25cm, white] (1,3)--(1,6);
    \draw[overcross, thick, black] (1,3)--(1,6);
    \draw[overcross] (3,3)--(3,6);
    \draw[overcross] (5,3)--(5,6);
    \end{scope}
\begin{scope}[xshift=9cm]
    \node at (0,6.4) {\tiny$-1$};
    \node at (2,6.4) {\tiny$1$};
    \node at (1,2.6) {$d$};
    \node at (3,6.4) {\tiny$2^d$};
    \node at (5,6.4) {\tiny$(n-1)^d$};
    \node at (-1,4.5) {$h_2=$};
    \node[font=\huge] at (4,4.5){$\cdots$};
    \draw[overcross, thick, black] (1,3)--(1,6);
    \draw[overcross](2,4)--(2,3);
    \draw[overcross, looseness=1.3](2,6)to[out=-90,in=90](0,3)--(0,3);
    \draw[overcross, looseness=1.3](2,4)to[out=90,in=-90](0,6);
    \draw[overcross] (3,3)--(3,6);
    \draw[overcross] (5,3)--(5,6);
    \end{scope}
\begin{scope}[yshift=-6cm]
    \node at (0,6.4) {\tiny$-1$};
    \node at (2,6.4) {\tiny$1$};
    \node at (1,2.6) {$d$};
    \node at (3,6.4) {\tiny$2^d$};
    \node at (4,6.4) {\tiny$3^d$};
    \node at (6,6.4) {\tiny$(n-1)^d$};
    \node at (-1,4.5) {$h_3=$};
    \node[font=\huge] at (5,4.5){$\cdots$};
    \draw[overcross, thick, black] (1,3)--(1,6);
    \draw[overcross] (0,3)--(0,6);
    \draw[overcross] (3,4)--(3,3);
    \draw[overcross, looseness=1.3](3,6)to[out=-90,in=90](2,4);
    \draw[overcross, looseness=1.3](3,4)to[out=90,in=-90](2,6);
    \draw[overcross] (2,4)--(2,3);
    \draw[overcross] (4,3)--(4,6);
    \draw[overcross] (6,3)--(6,6);
\end{scope}
\begin{scope}[xshift=9cm, yshift=-6cm]
    \node at (0,6.4) {\tiny$-1$};
    \node at (2,6.4) {\tiny$1$};
    \node at (1,2.6) {$d$};
    \node at (3.8,6.4) {\tiny{$(i-2)^d$}};
    \node at (5.2,6.4) {\tiny{$(i-1)^d$}};
    \node at (7,6.4) {\tiny{$(n-1)^d$}};
    \node at (-1,4.5) {$h_i=$};
    \node[font=\huge] at (3,4.5){$\cdots$};
    \node[font=\huge] at (6,4.5){$\cdots$};
    \draw[overcross, thick, black] (1,3)--(1,6);
    \draw[overcross] (0,3)--(0,6);
    \draw[overcross] (5,4)--(5,3);
    \draw[overcross, looseness=1.3](5,6)to[out=-90,in=90](4,4);
    \draw[overcross, looseness=1.3](5,4)to[out=90,in=-90](4,6);
    \draw[overcross] (4,4)--(4,3);
    \draw[overcross] (2,3)--(2,6);
    \draw[overcross] (7,3)--(7,6);
\end{scope}
\end{tikzpicture}
\caption{The braids $h_1,\, h_2,\,\dots,\, h_n$. The thicker line (pole) represents the cone point of degree $d$.}
 \label{fig:braids_hi}
\end{figure}
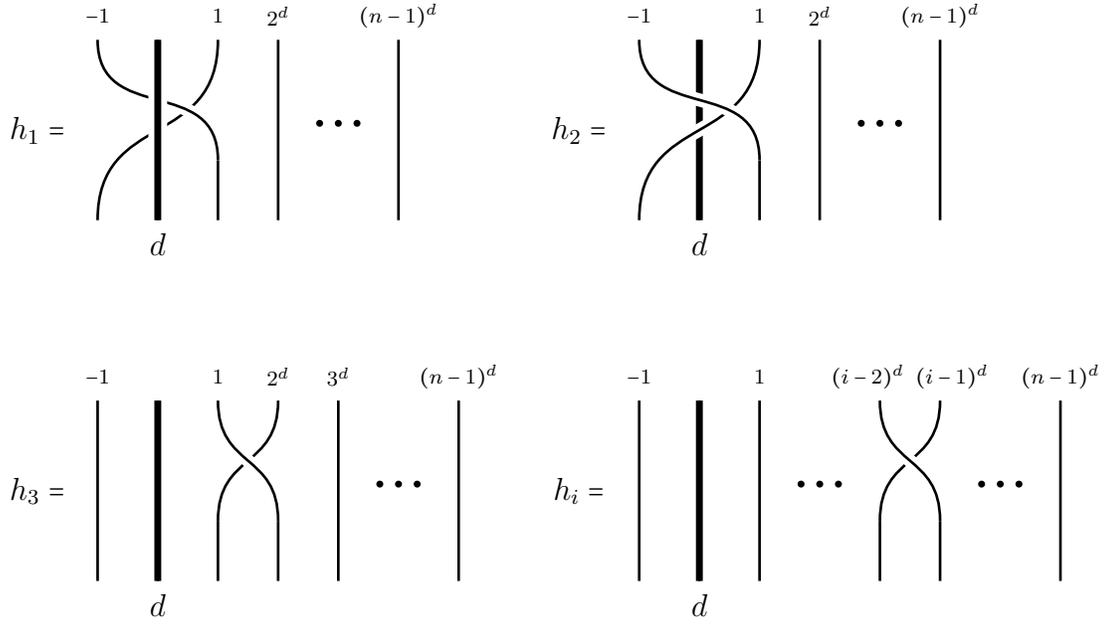

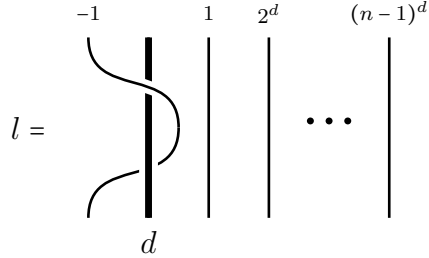
\begin{figure}[ht]
\begin{tikzpicture}[scale=0.8]
   \node at (0,6.4) {\tiny$-1$};
    \node at (2,6.4) {\tiny$1$};
    \node at (1,2.6) {$d$};
    \node at (3,6.4) {\tiny$2^d$};
    \node at (5,6.4) {\tiny$(n-1)^d$};
    \node at (-1,4.5) {$l=$};
    \node[font=\huge] at (4,4.5){$\cdots$};
    \draw[overcross, looseness=1.3](0,3)to[out=90,in=-90](1.5,4.5);
    \draw[line width=0.25cm, white] (1,3)--(1,4.5);
    \draw[overcross, thick, black] (1,3)--(1,6);
   \draw[overcross, looseness=1.3](1.5,4.5)to[out=90,in=-90](0,6);
    \draw[overcross] (3,3)--(3,6);
    \draw[overcross] (2,3)--(2,6);
    \draw[overcross] (5,3)--(5,6);
    \end{tikzpicture}
\caption{The loop $l$.}
 \label{fig:loop_l}
\end{figure}

\begin{lemma}
\label{lem:lorderd}
    The element $l\in Z_n(\MO_d)$ has order $d$.
\end{lemma}
\begin{proof}
By~\cite[\S4]{roushon21}, $P_n(\MO_d)$ is a subgroup of $Z_n(\MO_d)$, embedded as the subgroup of braids in $Z_n(\MO_d)$ in which each strand starts at ends at the same corresponding point in $\MO_d$ (i.e.\ pure braids). By~\cite[Rk.\ 2.15]{roushon21}, there is a homomorphism $\xi_n:P_n(\MO_d)\rightarrow P_{n-1}(\MO_d)$ obtained by removing the $n$th strand. Iterating this homomorphism gives a homomorphism $\xi$ from $P_n(\MO_d)$ to $P_1(\MO_d)=\pi_1(\MO_d)$ removing all strands except the first. It follows from~\cite[Rk.\ 2.2.2]{caramello} that $\xi(l^r)$ is not equal to the identity for $1\leq r\leq d-1$, from which the result follows.
\end{proof}

\begin{lemma}\label{lemma_Zn_description}
    Let $h_1,\dots,h_n,l$ be the elements of $Z_n(\MO_d)$ shown in Figures  \ref{fig:braids_hi} and \ref{fig:loop_l}.
    We have that 
    $$Z_n(\MO_d)=(\text{im}\beta)\cupdot(\text{im}\beta)l\cupdot(\text{im}\beta)l^2\cupdot\dots \cupdot(\text{im}\beta)l^{d-1}=\langle h_1,\dots, h_n,l\rangle.$$
\end{lemma}

\begin{proof}
Note that $B(d,d,n)$ is torsion free by~\cite[Thm. 0.4]{bessis15}
(see discussion after the theorem), noting that $G(d,d,n)$ is well-generated (see, for example,~\cite[\S 2.2.2]{lewismorales21}).
Hence, $\text{im}\beta=\beta(B(d,d,n))\subseteq Z_n(\MO_d)$ is torsion-free.
Hence, apart from the identity element, $\text{im}\beta$ has no element of finite order. Since the element $l$ has finite order $d$, we conclude that $l^q\not\in\text{im}\beta$ for any $q\not\in d\mathbb{Z}$.
Suppose now that $bl^q=b'l^p$ for some $b,\,b'\in\text{im}\beta$ and non-negative integers $q$, $p$. Then $l^{q-p}=(b')^{-1}b\in\text{im}\beta$ and so $l^{q-p}=e$. By Lemma~\ref{lem:lorderd}, $q\equiv p$ modulo $d$ and $l^p=l^q$.
In other words, if $l^p\neq l^q$, that is $p\neq q$ modulo $d$, then $(\text{im}\beta)l^p\cap (\text{im}\beta)l^q=\emptyset$ and, since $\text{im}\beta$ is a subgroup of index $d$ in $Z_n(\MO_d)$ by Propositions~\ref{prop_d-foldcovering} and \ref{prop_paths_drawn}, the cosets of $\text{im}\beta$ are $(\text{im}\beta)l^p$ for $0\leq p\leq d-1$. Since $\text{im}\beta=\langle h_1,\dots, h_n \rangle$ by Proposition~\ref{prop_paths_drawn}, we conclude that $Z_n(\MO_d)=\langle h_1,\dots, h_n,l\rangle$.
\end{proof}

\begin{theorem}
    The following is a group isomorphism
    \begin{align*}
        \gamma: {\frac{\AAA(B_n)}{\langle s^d =e\rangle}}\longrightarrow Z_n(\MO_d): s\mapsto l,\, t_2\mapsto h_2,\, t_i\mapsto h_i \textup{ for } 3\leq i\leq n,
    \end{align*}
    satisfying $\gamma(st_2s^{-1})=h_1$.
    See Figures \ref{fig:braids_hi} and \ref{fig:loop_l} for $h_1,\dots.h_n,l$. Furthermore, the diagram in Figure~\ref{fig:maindiagram} commutes.
\end{theorem}

\begin{proof}
We first show that the map $\gamma$ preserves the relations and it is hence a well-defined group homomorphism. By composing the braids, it is easy to see that $\gamma(st_2s^{-1})=lh_2l^{-1}=h_1$.

Composing the corresponding braids, it is immediate to see that the following relations are preserved by $\gamma$:
\begin{align*}
    &t_it_{i+1}t_i=t_{i+1}t_it_{i+1} \text{ for } 2\leq i\leq n-1, && t_it_j=t_jt_i \text{ for } \mid i-j\mid \geq 2, i,j\geq 3,\\
    &t_2t_i=t_it_2 \text{ for } 4\leq i\leq n, && st_i=t_i s \text{ for } 3\leq i\leq n. 
\end{align*}
Moreover, since the pole has degree $d$, it follows that $l^d$ is the identity and the relation $s^d=e$ is preserved. It only remains to show that the relation $st_2st_2=t_2st_2s$ is preserved. Note that $lh_2lh_2=h_2lh_2l$ as braids, see Figure~\ref{fig:gamma_well_defined} and note that this relation does not depend on the order of the pole.  Since only the first two strands and the pole are involved in this computation, we have omitted all the remaining strands from the pictures. Hence we conclude that $\gamma$ is a well-defined group homomorphism.

Using the description of $\varphi$ in Theorem~\ref{thm_varphi_iso} and the description of $\beta$ in Proposition~\ref{prop_paths_drawn}, we have $\gamma(\alpha(t_i))=\gamma(t_i)=h_i$, and $\beta\varphi(t_i)=\beta(\tau_i)=h_i$.
We also have $\gamma\alpha(st_2s^{-1})=\gamma(st_2s^{-1})=h_1$ (as noted above), while $\beta\varphi(st_2s^{-1})=\beta(\tau'_2)=h_1$.
Hence the diagram in Figure~\ref{fig:maindiagram} commutes.

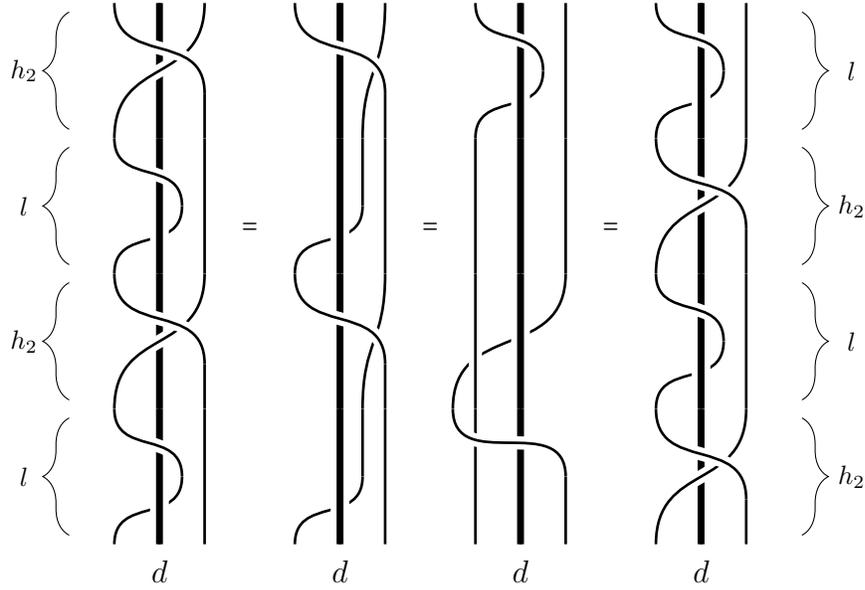
\begin{figure}
   \begin{tikzpicture}[scale=0.6]
       \begin{scope}
    \draw [decorate,decoration={brace,amplitude=10pt},xshift=-0.5cm,yshift=0pt]
(-0.5,3.2) -- (-0.5,5.8) node [black,midway,xshift=-0.6cm]
{\footnotesize $h_2$};
    \draw[overcross, thick, black] (1,3)--(1,6);
    \draw[overcross](2,4)--(2,3);
    \draw[overcross, looseness=1.3](2,6)to[out=-90,in=90](0,3)--(0,3);
    \draw[overcross, looseness=1.3](2,4)to[out=90,in=-90](0,6);
    \end{scope}
    \begin{scope}[yshift=-3cm]
    \draw [decorate,decoration={brace,amplitude=10pt},xshift=-0.5cm,yshift=0pt]
(-0.5,3.2) -- (-0.5,5.8) node [black,midway,xshift=-0.6cm]
{\footnotesize $l$};
    \draw[overcross, looseness=1.3](0,3)to[out=90,in=-90](1.5,4.5);
    \draw[line width=0.25cm, white] (1,3)--(1,4.5);
    \draw[overcross, thick, black] (1,3)--(1,6);
    \draw[overcross, looseness=1.3](1.5,4.5)to[out=90,in=-90](0,6);
    \draw[overcross] (2,3)--(2,6);
    \node at (3,4) {$=$};
   \end{scope}
   \begin{scope}[yshift=-6cm]
     \draw [decorate,decoration={brace,amplitude=10pt},xshift=-0.5cm,yshift=0pt]
(-0.5,3.2) -- (-0.5,5.8) node [black,midway,xshift=-0.6cm]
{\footnotesize $h_2$};
     \draw[overcross, thick, black] (1,3)--(1,6);
    \draw[overcross](2,4)--(2,3);
    \draw[overcross, looseness=1.3](2,6)to[out=-90,in=90](0,3)--(0,3);
    \draw[overcross, looseness=1.3](2,4)to[out=90,in=-90](0,6);
    \end{scope}
    \begin{scope}[yshift=-9cm]
     \draw [decorate,decoration={brace,amplitude=10pt},xshift=-0.5cm,yshift=0pt]
(-0.5,3.2) -- (-0.5,5.8) node [black,midway,xshift=-0.6cm]
{\footnotesize $l$};
    \draw[overcross, looseness=1.3](0,3)to[out=90,in=-90](1.5,4.5);
    \draw[line width=0.25cm, white] (1,3)--(1,4.5);
    \draw[overcross, thick, black] (1,3)--(1,6);
    \draw[overcross, looseness=1.3](1.5,4.5)to[out=90,in=-90](0,6);
    \draw[overcross] (2,3)--(2,6);
    \node at (1,2.4) {$d$};
   \end{scope}
   \begin{scope}[xshift=4cm]
       \begin{scope}
    \draw[overcross](2,4)--(2,3);
    \draw[overcross, looseness=1.3](2,6)to[out=-90,in=90](1.5,3);
    \draw[line width=0.25cm, white] (1,3)--(1,6);
    \draw[overcross, thick, black] (1,3)--(1,6);
    \draw[overcross, looseness=1.3](2,4)to[out=90,in=-90](0,6);
    \end{scope}
    \begin{scope}[yshift=-3cm]
    \draw[overcross, looseness=1.3](0,3)to[out=90,in=-90](1.5,4.5);
    \draw[overcross] (1.5,4.5)--(1.5,6);
    \draw[line width=0.25cm, white] (1,3)--(1,4.5);
    \draw[overcross, thick, black] (1,3)--(1,6);
    \draw[overcross] (2,3)--(2,6);
    \node at (3,4) {$=$};
   \end{scope}
   \begin{scope}[yshift=-6cm]
    \draw[overcross](2,4)--(2,3);
    \draw[overcross, looseness=1.3](2,6)to[out=-90,in=90](1.5,3);
    \draw[line width=0.25cm, white] (1,3)--(1,6);
    \draw[overcross, thick, black] (1,3)--(1,6);
    \draw[overcross, looseness=1.3](2,4)to[out=90,in=-90](0,6);
    \end{scope}
    \begin{scope}[yshift=-9cm]
    \draw[overcross, looseness=1.3](0,3)to[out=90,in=-90](1.5,4.5);
    \draw[overcross] (1.5,4.5)--(1.5,6);
    \draw[line width=0.25cm, white] (1,3)--(1,4.5);
    \draw[overcross, thick, black] (1,3)--(1,6);
    \draw[overcross] (2,3)--(2,6);
    \node at (1,2.4) {$d$};
   \end{scope}
   \end{scope}
   \begin{scope}[xshift=8cm]
       \begin{scope}
    \draw[overcross, looseness=1.3](0,3)to[out=90,in=-90](1.5,4.5);
    \draw[line width=0.25cm, white] (1,3)--(1,4.5);
    \draw[overcross, thick, black] (1,3)--(1,6);
    \draw[overcross, looseness=1.3](1.5,4.5)to[out=90,in=-90](0,6);
    \draw[overcross] (2,3)--(2,6);
    \end{scope}
    \begin{scope}[yshift=-3cm]
    \draw[overcross] (0,3)--(0,6);
    \draw[overcross] (2,3)--(2,6);
    \draw[overcross, thick, black] (1,3)--(1,6);
    \node at (3,4) {$=$};
   \end{scope}
   \begin{scope}[yshift=-6cm]
    \draw[overcross, looseness=1.3](2,6)to[out=-90,in=90](-0.5,3);
    \draw[overcross] (0,6)--(0,3);
    \draw[line width=0.25cm, white] (1,3)--(1,6);
    \draw[overcross, thick, black] (1,3)--(1,6);
    \end{scope}
    \begin{scope}[yshift=-9cm]
    \draw[overcross] (0,3)--(0,6);
    \draw[line width=0.25cm, white] (1,3)--(1,4.5);
    \draw[overcross, thick, black] (1,3)--(1,6);
    \draw[overcross, looseness=1.3](2,4.5)to[out=90,in=-90](-0.5,6);
    \draw[overcross] (2,3)--(2,4.5);
    \node at (1,2.4) {$d$};
   \end{scope}
   \end{scope}
   \begin{scope}[xshift=12cm]
        \begin{scope}[yshift=-3cm]
     \draw [decorate,decoration={brace,amplitude=10pt,mirror,raise=4pt},yshift=0pt]
(3,3.2) -- (3,5.8) node [black,midway,xshift=0.8cm] {\footnotesize
$h_2$};
     \draw[overcross, thick, black] (1,3)--(1,6);
    \draw[overcross](2,4)--(2,3);
    \draw[overcross, looseness=1.3](2,6)to[out=-90,in=90](0,3)--(0,3);
    \draw[overcross, looseness=1.3](2,4)to[out=90,in=-90](0,6);
    \end{scope}
    \begin{scope}[yshift=0cm]
     \draw [decorate,decoration={brace,amplitude=10pt,mirror,raise=4pt},yshift=0pt]
(3,3.2) -- (3,5.8) node [black,midway,xshift=0.8cm] {\footnotesize
$l$};
    \draw[overcross, looseness=1.3](0,3)to[out=90,in=-90](1.5,4.5);
    \draw[line width=0.25cm, white] (1,3)--(1,4.5);
    \draw[overcross, thick, black] (1,3)--(1,6);
    \draw[overcross, looseness=1.3](1.5,4.5)to[out=90,in=-90](0,6);
    \draw[overcross] (2,3)--(2,6);
   \end{scope}
   \begin{scope}[yshift=-9cm]
     \draw [decorate,decoration={brace,amplitude=10pt,mirror,raise=4pt},yshift=0pt]
(3,3.2) -- (3,5.8) node [black,midway,xshift=0.8cm] {\footnotesize
$h_2$};
     \draw[overcross, thick, black] (1,3)--(1,6);
    \draw[overcross](2,4)--(2,3);
    \draw[overcross, looseness=1.3](2,6)to[out=-90,in=90](0,3)--(0,3);
    \draw[overcross, looseness=1.3](2,4)to[out=90,in=-90](0,6);
    \node at (1,2.4) {$d$};
    \end{scope}
    \begin{scope}[yshift=-6cm]
    \draw [decorate,decoration={brace,amplitude=10pt,mirror,raise=4pt},yshift=0pt]
(3,3.2) -- (3,5.8) node [black,midway,xshift=0.8cm] {\footnotesize
$l$};
    \draw[overcross, looseness=1.3](0,3)to[out=90,in=-90](1.5,4.5);
    \draw[line width=0.25cm, white] (1,3)--(1,4.5);
    \draw[overcross, thick, black] (1,3)--(1,6);
    \draw[overcross, looseness=1.3](1.5,4.5)to[out=90,in=-90](0,6);
    \draw[overcross] (2,3)--(2,6);
   \end{scope}
   \end{scope}
   \end{tikzpicture}
   \caption{The braids $lh_2lh_2$ on the left and $h_2lh_2l$ on the right are equal.}
 \label{fig:gamma_well_defined}
\end{figure}

We now prove that $\gamma$ is a group isomorphism.
First note that $\gamma$ is surjective by Lemma~\ref{lemma_Zn_description}, noting that $h_1,h_2,\dots,h_n,l$ are all in the image of $\gamma$.

Recall that by Lemma~\ref{lemma_Nnormal_indexd},
\begin{align*}
        {\frac{\AAA(B_n)}{\langle s^d =e\rangle}}=N\cupdot Ns\cupdot Ns^2\cupdot \dots Ns^{d-1}.
\end{align*}

Then, if $\gamma(ns^p)=\gamma(n's^q)$ for two elements $ns^p$ and $n's^q$ in $\frac{\AAA(B_n)}{\langle s^d =e\rangle}$, where $n,n'\in N$ and $0\leq p,q\leq d-1$, we have that $\gamma(n)l^p=\gamma(n')l_q$.
Since the diagram in Figure~\ref{fig:maindiagram}
commutes, we have
$\beta(\varphi(n))l^p=\beta(\varphi(n'))l^q$.
By Lemma~\ref{lemma_Zn_description}, $\beta(\varphi(n))=\beta(\varphi(n'))$ and $p=q$. Since $\beta$ and $\varphi$ are injective, $n=n'$ and we see that $\gamma$ is injective and therefore an isomorphism as required.
\end{proof}

This completes the proof of Theorem~\ref{mainthmA}.

\section{Presentations of \texorpdfstring{$B(d,d,n)$}{Bddn} and \texorpdfstring{$G(d,d,n)$}{Gddn}}

In this section, we give new presentations of the complex braid groups $B(d,d,n)$ and their corresponding complex reflection groups $G(d,d,n)$.
In Definition~\ref{defn_quiver_on_disc}, we associate a (decorated) quiver $Q_T$ to an arbitrary tagged triangulation $T$ of a disk with $n$ marked points on the boundary and a cone point of degree $d$ in its interior. The quiver may have $2$-cycles, which we consider to be unoriented edges.
In Definition~\ref{defn_GQ_associated_group}, we associate a group $G_Q$ to such a quiver, given by generators and relations. A special case is Figure~\ref{fig:BMR_embedded}.
The corresponding presentation is the known presentation of $B(d,d,n)$ from~\cite[Thm. 2.27]{BMR98}.
The associated quiver in this case is an orientation of the diagram associated to the presentation in~\cite[Table 5]{BMR98} (drawn on the right in Figure~\ref{fig:BMR_embedded}); this formed part of the motivation for the approach we employ here.
In Definition~\ref{defn_mutationrules}, we introduce a mutation rule for such a quiver which is compatible with flipping triangulations.
We complete this section by showing that
the group $G_Q$ is invariant under mutation. Since the mutation graph of tagged triangulations of the disk is connected~\cite[Prop.\ 7.10]{FST}, it follows that all the quivers constructed as above give presentations of the group $B(d,d,n)$.
Adding the relations that the square of each generator is the identity gives a presentation of $G(d,d,n)$ by applying a result from~\cite{ariki}; see Theorem~\ref{thm:Gddnpresentation}.

\begin{figure}[ht]\scalebox{1}{
    \centering
 \begin{tikzpicture}[scale=1,
  quiverarrow/.style={black, -latex},
  iquiverarrow/.style={black, -latex, <-},
  mutationarc/.style={dashed, red, very thick},
  arc/.style={dashed, black},
  point/.style={gray},
  vertex/.style={black},
  conepoint/.style={gray, circle, draw=gray!100, fill=white!100, thick, inner sep=1.5pt},
  db/.style={thick, double, double distance=1.3pt, shorten <=-6pt}
  ]
 \begin{scope}
\draw[dashed](0,0) circle (2.5cm);
\draw[very thick, white] ([shift=(45:2.5cm)]0,0) arc (45:-30:2.5cm);
\draw[thick, loosely dotted] ([shift=(45:2.5cm)]0,0) arc (45:-30:2.5cm);
\node[point] (p0) at (170:2.5cm) {$\bullet$};
\node[point] (p1) at (140:2.5cm) {$\bullet$};
\node[point] (p2) at (110:2.5cm) {$\bullet$};
\node[point] (p3) at (80:2.5cm) {$\bullet$};
\node[point] (p4) at (50:2.5cm) {$\bullet$};
\node[point] (p5) at (-40:2.5cm) {$\bullet$};
\node[point] (p6) at (-70:2.5cm) {$\bullet$};
\node[point] (p7) at (-100:2.5cm) {$\bullet$};
\draw[arc](p0)--(p2);
\draw[arc](p0)--(p3);
\draw[arc](p0)--(p4);
 \draw[dashed] (p7) arc
	[start angle=-40,
		end angle=40,
		x radius=2.5cm,
		y radius =1.55cm
	] ;
 \draw[dashed] (-150:0.7cm) arc
	[start angle=140,
		end angle=220,
		x radius=2.5cm,
		y radius =1.55cm
	] ;
  \draw[thick, double, double distance=1.3pt] (-68:0.5cm)--(210:0.7cm);
  \node[conepoint] (c) at (-150:0.7cm) {\scriptsize $d$};
  \draw[dashed] (p0) arc
	[start angle=90,
		end angle=-45,
		x radius=3cm,
		y radius =1.7cm
	] ;
 \draw[dashed] (p0) arc
	[start angle=110,
		end angle=-30,
		x radius=2.8cm,
		y radius =1.9cm
	] ;
  \draw[dashed] (p0) arc
	[start angle=120,
		end angle=-10,
		x radius=3cm,
		y radius =1.8cm
	] ;
 \node[point] at (-100:2.5cm) {$\bullet$};
 \node[point] at (170:2.5cm) {$\bullet$};
 \node[vertex, label=
{[label distance=-5pt]90:{\tiny $n$}}] (w1) at  ($(p0)! 0.5!(p2)$) {$\bullet$};
 \node[vertex, label=
{[label distance=-5pt]90:{\tiny $n-1$}}] (w2) at  ($(p0)! 0.6!(p3)$) {$\bullet$};
 \node[vertex, label=
{[label distance=-5pt]90:{\tiny $n-2$}}] (w3) at  ($(p0)! 0.7!(p4)$) {$\bullet$};
 \node[vertex, label=
{[label distance=-5pt]90:{\tiny $5$}}] (w4) at  (15:1cm) {$\bullet$};
 \node[vertex, label=
{[label distance=-5pt]-90:{\tiny $4$}}] (w5) at  (-32:1.3cm) {$\bullet$};
 \node[vertex, label=
{[label distance=-5pt]90:{\tiny $3$}}] (w6) at  (-68:0.5cm) {$\bullet$};
 \node[vertex, label=
{[label distance=-8pt]0:{\tiny $1$}}] (w7) at  (-85:1.6cm) {$\bullet$};
 \node[vertex, label=
{[label distance=-8pt]180:{\tiny $2$}}] (w8) at  (-125:2cm) {$\bullet$};
 \draw[quiverarrow] (w1)--(w2);
 \draw[quiverarrow] (w2)--(w3);
 \draw[quiverarrow, dotted, thick] (w3)--(w4);
 \draw[quiverarrow] (w4)--(w5);
 \draw[quiverarrow] (w5)--(w6);
 \draw[quiverarrow] (w7)--(w6);
 \draw[quiverarrow] (w8)--(w6);
 \draw[-] (w8)--(w7);
 \node[below] at  ($(w8)! 0.5!(w7)$) {\tiny $d$};
  \node[rotate=-90] at (-1.2,-0.9) {\rotatebox{-60}{$\smalltriangleright$}};
  \node[rotate=-90] at (-1.06,-0.99) {\rotatebox{-60}{$\smalltriangleleft$}};
 \end{scope}
 \begin{scope}[xshift=5cm]
\node[vertex, label=
{[label distance=-5pt]90:{\tiny $2$}}] (w8) at  (0,1) {$\bullet$};
 \node[vertex, label=
{[label distance=-5pt]-90:{\tiny $1$}}] (w7) at  (0,-1) {$\bullet$};
 \node[vertex, label=
{[label distance=-5pt]90:{\tiny $3$}}] (w6) at  (1.5,0) {$\bullet$};
 \node[vertex, label=
{[label distance=-5pt]90:{\tiny $4$}}] (w5) at  (2.5,0) {$\bullet$};
 \node[vertex, label=
{[label distance=-5pt]90:{\tiny $5$}}] (w4) at  (3.5,0) {$\bullet$};
 \node[vertex, label=
{[label distance=-5pt]90:{\tiny $n-2$}}] (w3) at  (5,0) {$\bullet$};
 \node[vertex, label=
{[label distance=-5pt]90:{\tiny $n-1$}}] (w2) at  (6,0) {$\bullet$};
 \node[vertex, label=
{[label distance=-5pt]90:{\tiny $n$}}] (w1) at  (7,0) {$\bullet$};
\draw[quiverarrow] (w1)--(w2);
 \draw[quiverarrow] (w2)--(w3);
 \draw[quiverarrow, dotted, thick] (w3)--(w4);
 \draw[quiverarrow] (w4)--(w5);
 \draw[quiverarrow] (w5)--(w6);
 \draw[quiverarrow] (w7)--(w6);
 \draw[quiverarrow] (w8)--(w6);
 \draw[-] (w8)--(w7);
 \draw[db] (1.1,0)--(0.7,0);
 \node[left] at  ($(w8)! 0.5!(w7)$) {\tiny $d$};
 \end{scope}
 \end{tikzpicture}}
 \caption{Original presentation from \cite{BMR98} (shown on the right) embedded in a disk with $n$ marked points on the boundary and a cone point of degree $d$ in its interior.}
\label{fig:BMR_embedded}
\end{figure}
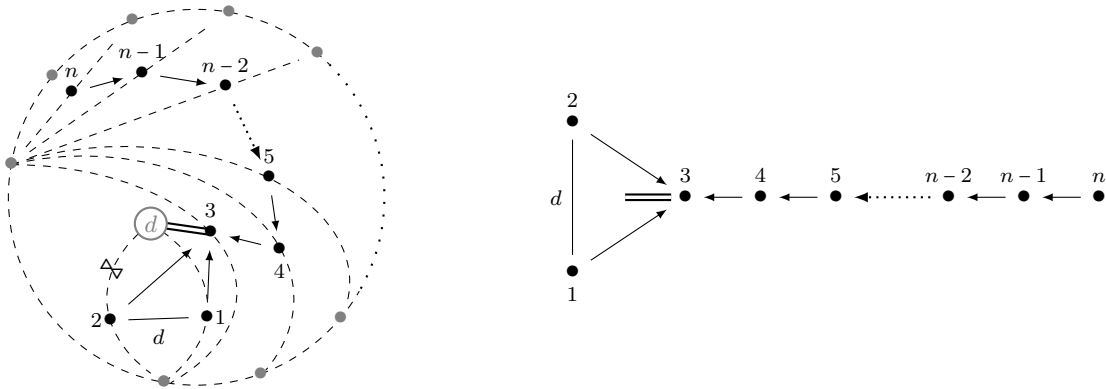

\subsection{Quivers from triangulated surfaces and groups from quivers.}\label{subsection_associating_quiver}
We now fix the surface $(X,M)$ we will be working with: $X$ is the disk $S$ with an interior marked point interpreted a cone point of degree $d\geq 2$ denoted by
$\begin{tikzpicture}[scale=1,
    conepoint/.style={gray, circle, draw=gray!100, fill=white!100, thick, inner sep=1.5pt}]
\node[conepoint]  at (0,-0.2) {\scriptsize $d$};
\end{tikzpicture}$
in figures, and $M$ is a set of $n\geq 2$ marked points on its boundary.

Next, we give a way of associating a decorated quiver $Q_T$ to each tagged triangulation of $(X,M)$. We then give a way of associating a group $G_{Q_T}$ to any such quiver.

Let $T$ be a tagged triangulation of $(X,M)$, regarding $C$ as a marked point, as in~\cite[\S7]{FST}. Note that $C$ is the unique marked point in the interior of $S$.

\begin{remark}
    \label{rem:puzzlepieces}
By~\cite[Rk.\ 4.2]{FST}, the tagged triangulation $T$ can be built up by gluing puzzle pieces of the kind shown in Figure~\ref{fig:puzzlepieces} by matching their boundary arcs (respecting the orientation).
\end{remark}

Let $B_T$ be the skew-symmetric matrix associated to $T$ in~\cite[Rk.\ 4.2]{FST}, and let $\widetilde{Q}_T$ be the corresponding quiver.

\begin{defn}\label{defn_quiver_on_disc}
We associate a quiver $Q_T$ to each tagged triangulation $T$ as considered above by modifying $\widetilde{Q}_T$ in the following way.

If the triangulation $T$ has precisely two arcs incident with the cone point then,
from each of the two vertices in $\widetilde{Q}_T$ that have arrows to or from the corresponding vertices in $\widetilde{Q}$, we draw a double edge pointing towards the conepoint. Thus, we draw such a double edge from each vertex labelling an arc in the triangulation $T$ bounding a region in the complement that has the cone point on its boundary. We also add an unoriented edge, labelled $d$, between these two vertices.

If the triangulation has at least $3$ arcs adjacent to the cone point, there will be an oriented cycle between the corresponding vertices in $\widetilde{Q}_T$.
We put the label
$\begin{tikzpicture}[scale=1,
conepoint/.style={gray, circle, draw=gray!100, fill=white!100, thick, inner sep=1.5pt}]
\node[conepoint]  at (0,-0.2) {\scriptsize $d$};
\end{tikzpicture}$
in the middle of such a cycle. Note that there is at most one such labelled cycle.
\end{defn}

See Figure~\ref{fig:BMR_embedded} for a complete example of Definition~\ref{defn_quiver_on_disc} and Figures~\ref{fig:typeAmutation}, \ref{fig:cornercase}, \ref{fig:mutationnearcone}, \ref{fig:mutation_tagged} and \ref{fig:mutation_untagged} for local portions of a triangulation and the associated quiver.

\begin{figure}[ht]\scalebox{1}{
    \centering
 \begin{tikzpicture}[scale=2,
  quiverarrow/.style={black, -latex},
  mutationarc/.style={dashed, red, very thick},
  arc/.style={dashed, black},
  point/.style={gray},
  vertex/.style={black},
  conepoint/.style={gray, circle, draw=gray!100, fill=white!100, thick, inner sep=1.5pt},
  db/.style={thick, double, double distance=1.3pt, shorten <=-6pt}
  ]
 \begin{scope}
  \node[point] (p0) at (0,0) {$\bullet$};
  \node[point] (p1) at (0.5,0.866) {$\bullet$};
  \node[point] (p2) at (1,0) {$\bullet$};
  \draw[arc](p0) -- (p1);
  \draw[arc](p1) -- (p2);
  \draw[arc](p2) -- (p0); 
\end{scope}
\begin{scope}[xshift=2cm]
  \node[point] (p0) at (0,0) {$\bullet$};
  \node[conepoint] (p1) at (0,0.7) {$d$};
  \node[point] (p2) at (0,1.4) {$\bullet$};
 \draw[arc,shorten <=-2pt, shorten >=-2pt] (p0) .. controls +(45:0.4) and +(315:0.4) .. (p1);
  \draw[arc, shorten <=-2pt, shorten >=-2pt] (p0) .. controls +(135:0.4) and +(225:0.4) .. (p1);
   \draw[arc,shorten <=-2pt, shorten >=-2pt] (p0) .. controls +(0:1) and +(360:1) .. (p2);
  \draw[arc, shorten <=-2pt, shorten >=-2pt] (p0) .. controls +(180:1) and +(180:1) .. (p2);
\end{scope}
 \end{tikzpicture}}
\caption{Puzzle pieces for a disk with a single interior marked point.}
\label{fig:puzzlepieces}
\end{figure}
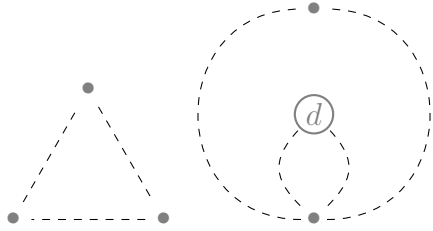

\begin{remark}
\label{rem:conepointportion}
It follows from Remark~\ref{rem:puzzlepieces} that the quiver $Q_T$ is built up from individual portions associated to triangles (as on the left of Figure~\ref{fig:puzzlepieces}) and a portion associated to the union of the puzzle pieces incident with the cone-point as shown in Figure~\ref{fig:puzzlepiecesquiver} (or with versions of the three right hand figures where the tags on all the arcs incident with the cone point have been flipped).
\end{remark}

\begin{figure}[ht]
\scalebox{1.2}{
    \centering
 \begin{tikzpicture}[scale=2,
  quiverarrow/.style={black,-latex,shorten <=-4pt, shorten >=-4pt},
  mutationarc/.style={dashed, red, very thick},
  arc/.style={dashed, black},
  point/.style={gray},
  vertex/.style={black},
  conepoint/.style={gray, circle, draw=gray!100, fill=white!100, thick, inner sep=1.5pt},
  db/.style={thick, double, double distance=1.3pt, shorten <=-6pt}
  ]
 \begin{scope}
  \node[point] (p0) at (0,0) {$\bullet$};
  \node[point] (p1) at (0.5,0.866) {$\bullet$};
  \node[point] (p2) at (1,0) {$\bullet$};
  \draw[arc](p0) -- (p1);
  \draw[arc](p1) -- (p2);
  \draw[arc](p2) -- (p0); 
  \node[vertex] (v0) at (0.25,0.433) {$\bullet$};
  \node[vertex] (v1) at (0.75,0.433) {$\bullet$};
  \node[vertex] (v2) at (0.5,0) {$\bullet$};
  \draw[quiverarrow](v0) -- (v2);
  \draw[quiverarrow](v2) -- (v1);
  \draw[quiverarrow](v1) -- (v0);
\end{scope}
\begin{scope}[xshift=2cm]
  \node[point] (p0) at (0,0) {$\bullet$};
  \node[conepoint] (p1) at (0,0.7) {\tiny $d$};
  \node[point] (p2) at (0,1.4) {$\bullet$};
  \draw[arc, shorten <=-3pt, shorten >=0pt] (p0) .. controls +(135:0.4) and +(225:0.4) ..
  coordinate[midway](m1) (p1);
  \draw[arc,shorten <=-3pt, shorten >=0pt] (p0) .. controls +(45:0.4) and +(315:0.4) .. coordinate[midway](m2) coordinate[pos=0.9](n1) (p1);
\draw[arc, shorten <=-2pt, shorten >=-2pt] (p0) .. controls +(180:1) and +(180:1) .. coordinate[midway](m3) (p2);
\draw[arc,shorten <=-2pt, shorten >=-2pt] (p0) .. controls +(0:1) and +(360:1) .. coordinate[midway](m4) (p2);
\node[vertex,rotate=45](t1) at (n1) {$\bowtie$};
\node[vertex](w1) at (m1) {$\bullet$};
\node[vertex](w2) at (m2) {$\bullet$};
\node[vertex](w3) at (m3) {$\bullet$};
\node[vertex](w4) at (m4) {$\bullet$};
\draw[quiverarrow](w3) -- (w1);
\draw[quiverarrow](w3) -- (w2);
\draw[quiverarrow](w1) -- (w4);
\draw[quiverarrow](w2) -- (w4);
\draw (w1) to node[midway,below=-2pt]{\tiny $d$} (w2);
\draw [db,shorten >=10pt](w3) -- (p1);
\draw [db,shorten >=10pt](w4) -- (p1);
\end{scope}
\begin{scope}[xshift=4cm]
  \node[point] (p0) at (0,0) {$\bullet$};
  \node[conepoint] (p1) at (0,0.7) {\tiny $d$};
  \node[point] (p2) at (0,1.4) {$\bullet$};
  \draw[arc, shorten <=-2pt, shorten >=0pt] (p0) --  coordinate[midway](m1) (p1);
 \draw[arc,shorten <=-2pt, shorten >=0pt] (p2) --coordinate[midway](m2) (p1);
\draw[arc, shorten <=-2pt, shorten >=-2pt] (p0) .. controls +(180:1) and +(180:1) .. coordinate[midway](m3) (p2);
\draw[arc,shorten <=-2pt, shorten >=-2pt] (p0) .. controls +(0:1) and +(360:1) .. coordinate[midway](m4) (p2);
\node[vertex](x1) at (m1) {$\bullet$};
\node[vertex](x2) at (m2) {$\bullet$};
\node[vertex](x3) at (m3) {$\bullet$};
\node[vertex](x4) at (m4) {$\bullet$};
\draw[quiverarrow](x3) -- (x1);
\draw[quiverarrow](x2) -- (x3);
\draw[quiverarrow](x1) -- (x4);
\draw[quiverarrow](x4) -- (x2);
\draw (x1) .. controls +(60:0.35) and +(-60:0.35) .. coordinate[midway,right=3pt](m12)(x2);
\node at (m12){\tiny $d$};
\draw [db,shorten >=10pt](x3) -- (p1);
\draw [db,shorten >=10pt](x4) -- (p1);
\end{scope}
\begin{scope}[xshift=6cm]
\draw[arc] (0,0.7) circle (0.7);
\node[point](q1) at (0,1.4) {$\bullet$};
\node[point](q2) at (0.606,0.35) {$\bullet$};
\node[point](q3) at (-0.606,0.35) {$\bullet$};
\node[conepoint](c) at (0,0.7) {\tiny $d$};
\draw[arc](q1) -- coordinate[midway](m1) (c);
\draw[arc](q2) -- coordinate[midway](m2) (c);
\draw[arc](q3) -- coordinate[midway](m3) (c);
\node[vertex](z1) at (m1) {$\bullet$};
\node[vertex](z2) at (m2) {$\bullet$};
\node[vertex](z3) at (m3) {$\bullet$};
\draw[quiverarrow](z1) ..
controls +(0:0.25) and +(60:0.25) .. (z2);
\draw[quiverarrow](z2) ..
controls +(240:0.25) and +(300:0.25) .. (z3);
\draw[quiverarrow](z3) ..
controls +(120:0.25) and +(180:0.25) .. (z1);
\end{scope}
\end{tikzpicture}}
\scalebox{1.2}{
    \centering
 \begin{tikzpicture}[scale=2,
  quiverarrow/.style={black,-latex,shorten <=-4pt, shorten >=-4pt},
  mutationarc/.style={dashed, red, very thick},
  arc/.style={dashed, black},
  point/.style={gray},
  vertex/.style={black},
  conepoint/.style={gray, circle, draw=gray!100, fill=white!100, thick, inner sep=1.5pt},
  db/.style={thick, double, double distance=1.3pt, shorten <=-6pt}
  ]
 \begin{scope}[yshift=0.5cm]
  \node[point] (p0) at (0,0) {};
  \node[point] (p1) at (0.5,0.866) {};
  \node[point] (p2) at (1,0) {};
  \draw[arc,draw=none](p0) -- (p1);
  \draw[arc,draw=none](p1) -- (p2);
  \draw[arc,draw=none](p2) -- (p0); 
  \node[vertex] (v0) at (0.25,0.433) {$\bullet$};
  \node[vertex] (v1) at (0.75,0.433) {$\bullet$};
  \node[vertex] (v2) at (0.5,0) {$\bullet$};
  \draw[quiverarrow](v0) -- (v2);
  \draw[quiverarrow](v2) -- (v1);
  \draw[quiverarrow](v1) -- (v0);\end{scope}
\begin{scope}[xshift=1.55cm,yshift=0.68cm]
\node[vertex](w1) at (0.5,0.5) {$\bullet$};
\node[vertex](w2) at (0.5,-0.5) {$\bullet$};
\node[vertex](w3) at (0,0) {$\bullet$};
\node[vertex](w4) at (1,0) {$\bullet$};
\draw[quiverarrow](w3) -- (w1);
\draw[quiverarrow](w3) -- (w2);
\draw[quiverarrow](w1) -- (w4);
\draw[quiverarrow](w2) -- (w4);
\draw (w1) to node[pos=0.3, right=-2pt]{\tiny $d$}(w2);
\draw [db,shorten >=10pt](w3) -- (0.5,0);
\draw [db,shorten >=10pt](w4) -- (0.5,0);
\end{scope}
\begin{scope}[xshift=4cm]
  \node[point] (p0) at (0,0) {};
  \node[conepoint, draw=none] (c) at (0,0.7) {};
  \node[point] (p2) at (0,1.4) {};
  \draw[draw=none,arc, shorten <=-2pt, shorten >=-2pt] (p0) --  coordinate[midway](m1) (c);
 \draw[draw=none,arc,shorten <=-2pt, shorten >=-2pt] (p2) --coordinate[midway](m2) (c);
\draw[draw=none,arc, shorten <=-2pt, shorten >=-2pt] (p0) .. controls +(180:1) and +(180:1) .. coordinate[midway](m3) (p2);
\draw[draw=none,arc,shorten <=-2pt, shorten >=-2pt] (p0) .. controls +(0:1) and +(360:1) .. coordinate[midway](m4) (p2);
\node[vertex](x1) at (m1) {$\bullet$};
\node[vertex](x2) at (m2) {$\bullet$};
\node[vertex](x3) at (m3) {$\bullet$};
\node[vertex](x4) at (m4) {$\bullet$};
\draw[quiverarrow](x3) -- (x1);
\draw[quiverarrow](x2) -- (x3);
\draw[quiverarrow](x1) -- (x4);
\draw[quiverarrow](x4) -- (x2);
\draw (x2) to node[pos=0.3, right=-2pt]{\tiny $d$}(x1); 
\draw [db,shorten >=10pt](x3) -- (c);
\draw [db,shorten >=10pt](x4) -- (c);
\end{scope}
\begin{scope}[xshift=6cm]
\node[point](q1) at (0,1.4) {};
\node[point](q2) at (0.606,0.35) {};
\node[point](q3) at (-0.606,0.35) {};
\node[conepoint](c) at (0,0.7) {\tiny $d$};
\draw[arc,draw=none](q1) -- coordinate[midway](m1) (c);
\draw[arc,draw=none](q2) -- coordinate[midway](m2) (c);
\draw[arc,draw=none](q3) -- coordinate[midway](m3) (c);
\node[vertex](z1) at (m1) {$\bullet$};
\node[vertex](z2) at (m2) {$\bullet$};
\node[vertex](z3) at (m3) {$\bullet$};
\draw[quiverarrow](z1) ..
controls +(0:0.25) and +(60:0.25) .. (z2);
\draw[quiverarrow](z2) ..
controls +(240:0.25) and +(300:0.25) .. (z3);
\draw[quiverarrow](z3) ..
controls +(120:0.25) and +(180:0.25) .. (z1);
\end{scope}
\end{tikzpicture}}
\caption{Building up the associated quiver for a disk with cone point. We show the corresponding quivers separately. Note that the cone point itself does not appear in the quiver in the second and third cases.}
\label{fig:puzzlepiecesquiver}
\end{figure}
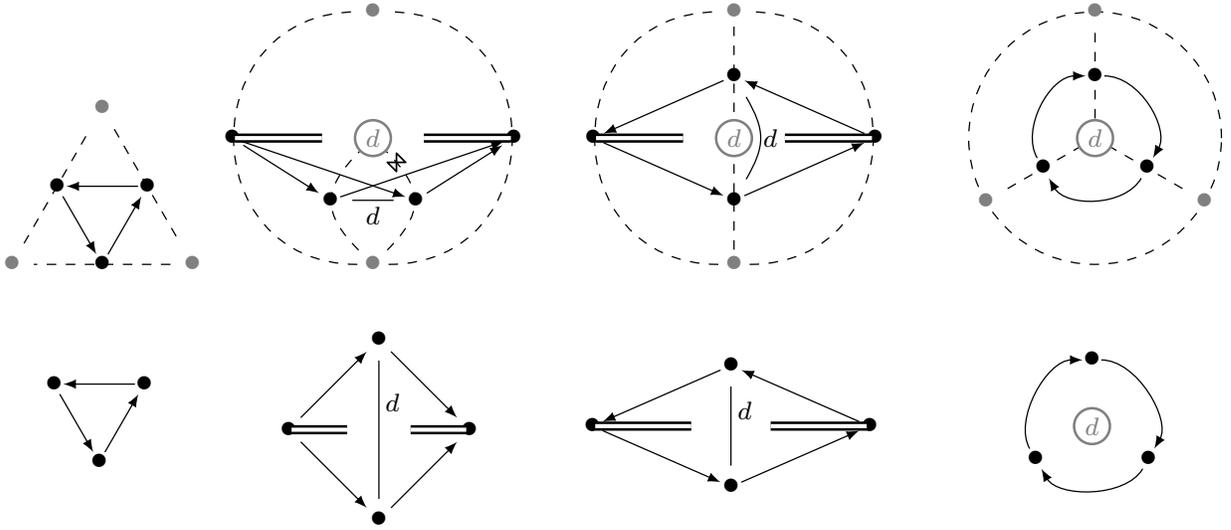

\begin{defn}\label{defn_GQ_associated_group}
  Let $Q$ be a quiver as in Definition~\ref{defn_quiver_on_disc} and $G_Q$ be the group with generators $S_Q=\{s_i\}_{i\in Q_0}$ subject to the following relations:
  \begin{enumerate}
      \item\label{item_commutation} $s_is_j=s_js_i$ if $i$ and $j$ are vertices with no arrows between them,
      \item\label{item_braid} $s_is_js_i=s_js_is_j$ if there is an arrow between $i$ and $j$ (in either direction).
      \item\label{item_cycle}  $s_is_js_ks_i=s_js_ks_is_j=s_ks_is_js_k$ if $Q$ contains an oriented $3$-cycle with no label $\begin{tikzpicture}[scale=1,
conepoint/.style={gray, circle, draw=gray!100, fill=white!100, thick, inner sep=1.5pt}]
\node[conepoint]  at (0,-0.2) {\scriptsize $d$};
\end{tikzpicture}$ in the middle of the form
$$
      
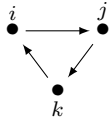
\begin{tikzpicture}[scale=1,
  quiverarrow/.style={black,-latex,shorten <=-4pt, shorten >=-4pt},
  mutationarc/.style={dashed, red, very thick, shorten <=-4pt, shorten >=-4pt},
  mutationarccone/.style={dashed, red, very thick, shorten >=-4pt},
  arc/.style={dashed, black, shorten <=-4pt, shorten >=-4pt},
  arccone/.style={dashed, black, shorten >=-4pt},
  point/.style={gray},
  vertex/.style={black},
  conepoint/.style={gray, circle, draw=gray!100, fill=white!100, thick, inner sep=1.5pt},
  db/.style={thick, double, double distance=1.3pt, shorten <=-6pt}
  ]
\node (j2) at (0,0) {$\bullet$};
\node [above] at (j2) {\tiny {$i$}}; 
\node (i2) at (1.2,0) {$\bullet$};
\node [above] at (i2) {\tiny {$j$}};
\node (k2) at (0.6,-0.8) {$\bullet$};
\node[below] at (k2) {\tiny {$k$}};
\draw [quiverarrow, shorten <=-2pt, shorten >=-2pt] (k2) -- (j2);
\draw [quiverarrow, shorten <=-2pt, shorten >=-2pt] (i2) -- (k2);
\draw [quiverarrow, shorten <=-2pt, shorten >=-2pt] (j2) -- (i2);
 \end{tikzpicture}
$$
\item\label{item_dreln} $\underbrace{s_is_j\cdots }_{\text{$d$ terms}} =\underbrace{s_js_i\cdots }_{\text{$d$ terms}}$ if there is an (unoriented) edge labelled $d$ between $i$ and $j$,
\item\label{item_cycled} $\underbrace{s_1s_2\cdots s_rs_1\dots}_{\text{$d(r-1)$ terms}} =\underbrace{s_2\cdots s_rs_1\dots}_{\text{$d(r-1)$ terms}}=\dots=\underbrace{s_rs_1\cdots s_rs_1\dots}_{\text{$d(r-1)$ terms}}$
if $Q$ contains an oriented labelled chordless $r$-cycle, for $r\geq 3$, of the form
$$
      
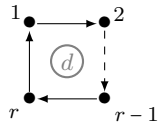
\begin{tikzpicture}[scale=1,
  quiverarrow/.style={black,-latex,shorten <=-4pt, shorten >=-4pt},
  mutationarc/.style={dashed, red, very thick, shorten <=-4pt, shorten >=-4pt},
  mutationarccone/.style={dashed, red, very thick, shorten >=-4pt},
  arc/.style={dashed, black, shorten <=-4pt, shorten >=-4pt},
  arccone/.style={dashed, black, shorten >=-4pt},
  point/.style={gray},
  vertex/.style={black},
  conepoint/.style={gray, circle, draw=gray!100, fill=white!100, thick, inner sep=1.5pt},
  db/.style={thick, double, double distance=1.3pt, shorten <=-6pt}
  ]
 \node[vertex, label=
{[label distance=-10pt]225:{\tiny $r$}}] (vr) at (0,0) {$\bullet$};
 \node[vertex, label={[label distance=-10pt]115:{\tiny $1$}}] (v1) at (0,1) {$\bullet$};
 \node[vertex, label=
{[label distance=-10pt]75:{\tiny $2$}}] (v2) at (1,1) {$\bullet$};
 \node[vertex, label=
{[label distance=-8pt]290:{\tiny $r-1$}}] (v3) at (1,0) {$\bullet$};
 \draw[quiverarrow] (v1) -- (v2);
 \draw[quiverarrow, dashed] (v2) -- (v3);
 \draw[quiverarrow] (v3) -- (vr);
 \draw[quiverarrow] (vr) -- (v1);
 \node[conepoint]  at (0.5,0.5) {\scriptsize $d$};
 \end{tikzpicture}
$$
\item\label{item_doubleedge} $s_ks_is_js_ks_is_j=s_is_js_ks_is_js_k$ if the vertices $i,j,k$ appear in either of the following configurations in $Q$
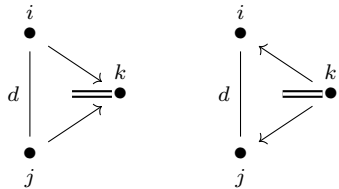
$$
 \begin{tikzpicture}[scale=0.8,
  quiverarrow/.style={black,-latex,shorten <=-4pt, shorten >=-4pt},
  mutationarc/.style={dashed, red, very thick, shorten <=-4pt, shorten >=-4pt},
  mutationarccone/.style={dashed, red, very thick, shorten >=-4pt},
  arc/.style={dashed, black, shorten <=-4pt, shorten >=-4pt},
  arccone/.style={dashed, black, shorten >=-4pt},
  point/.style={gray},
  vertex/.style={black},
  conepoint/.style={gray, circle, draw=gray!100, fill=white!100, thick, inner sep=1.5pt},
  db/.style={thick, double, double distance=1.3pt, shorten <=-6pt}
  ]
\begin{scope}
\node[vertex, label=
{[label distance=-5pt]90:{\tiny $i$}}] (w8) at  (0,1) {$\bullet$};
 \node[vertex, label=
{[label distance=-5pt]-90:{\tiny $j$}}] (w7) at  (0,-1) {$\bullet$};
 \node[vertex, label=
{[label distance=-5pt]90:{\tiny $k$}}] (w6) at  (1.5,0) {$\bullet$};
 \draw[->] (w7)--(w6);
 \draw[->] (w8)--(w6);
 \draw[-] (w8)--(w7);
 \draw[db] (1.1,0)--(0.7,0);
 \node[left] at  ($(w8)! 0.5!(w7)$) {\tiny $d$};
 \end{scope}
\begin{scope}[xshift=3.5cm]
\node[vertex, label=
{[label distance=-5pt]90:{\tiny $i$}}] (w8) at  (0,1) {$\bullet$};
 \node[vertex, label=
{[label distance=-5pt]-90:{\tiny $j$}}] (w7) at  (0,-1) {$\bullet$};
 \node[vertex, label=
{[label distance=-5pt]90:{\tiny $k$}}] (w6) at  (1.5,0) {$\bullet$};
 \draw[<-] (w7)--(w6);
 \draw[<-] (w8)--(w6);
 \draw[-] (w8)--(w7);
 \draw[db] (1.1,0)--(0.7,0);
 \node[left] at  ($(w8)! 0.5!(w7)$) {\tiny $d$};
 \end{scope}
\end{tikzpicture}
$$
Note that $s_i$ appears before $s_j$ in the relation because $i$ appears to the right of the double edge.
\item\label{item_broken4cycle} $s_is_js_ks_ls_is_j=s_ls_is_js_ks_ls_i$ and $s_js_ks_ls_is_js_k=s_ks_ls_is_js_ks_l$ if the vertices $i,j,k,l$ appear as follows in $Q$
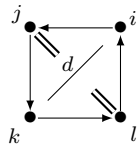
$$
      \begin{tikzpicture}[scale=1.2,
  quiverarrow/.style={black,-latex,shorten <=-4pt, shorten >=-4pt},
  mutationarc/.style={dashed, red, very thick, shorten <=-4pt, shorten >=-4pt},
  mutationarccone/.style={dashed, red, very thick, shorten >=-4pt},
  arc/.style={dashed, black, shorten <=-4pt, shorten >=-4pt},
  arccone/.style={dashed, black, shorten >=-4pt},
  point/.style={gray},
  vertex/.style={black},
  conepoint/.style={gray, circle, draw=gray!100, fill=white!100, thick, inner sep=1.5pt},
  db/.style={thick, double, double distance=1.3pt, shorten <=-6pt}
  ]
 \node[vertex, label=
{[label distance=-10pt]225:{\tiny $k$}}] (vk) at (0,0) {$\bullet$};
 \node[vertex, label={[label distance=-10pt]115:{\tiny $j$}}] (vj) at (0,1) {$\bullet$};
 \node[vertex, label=
{[label distance=-10pt]75:{\tiny $i$}}] (vi) at (1,1) {$\bullet$};
 \node[vertex, label=
{[label distance=-8pt]290:{\tiny $l$}}] (vl) at (1,0) {$\bullet$};
 \draw[quiverarrow] (vi) -- (vj);
 \draw[quiverarrow] (vj) -- (vk);
 \draw[quiverarrow] (vk) -- (vl);
 \draw[quiverarrow] (vl) -- (vi);
 \draw[-] (vi) -- (vk);
 \node[left] at  ($(vi)! 0.4!(vk)$) {\tiny $d$};
 \draw[db] (vj)--(0.3,0.7);
  \draw[db] (vl)--(0.7,0.3);
 \end{tikzpicture}
$$
  \end{enumerate}

Let $H_Q$ be the group defined as above, omitting the cycle relations (5).
\end{defn}

The construction of Grant and Marsh~\cite{GM} becomes a special case of the above construction as follows. Our aim is to generalise their construction for larger values of $d$.

\begin{remark}\label{remark_d2_muchsimpler}
    Note that when $d=2$, the relations from Definition~\ref{defn_GQ_associated_group} simplify. In fact, arrows labelled $2$ give commutation relations, so they can be omitted following rule (\ref{item_commutation}).
    Moreover, we can see that double edges can be omitted as well. In fact, in the situation of relation (\ref{item_doubleedge}), we have $s_is_j=s_js_i$, $s_is_ks_i=s_ks_is_k$ and $s_js_ks_j=s_ks_js_k$ and so
    \begin{align*}
        s_ks_is_js_ks_is_j&= s_ks_js_is_ks_is_j= s_ks_js_ks_is_ks_j = s_js_ks_js_is_ks_j = s_js_ks_is_js_ks_j\\
        &= s_js_ks_is_ks_js_k=s_js_is_ks_is_js_k= s_is_js_ks_is_js_k,
    \end{align*}
    that is relation (\ref{item_doubleedge}) becomes a consequence of the other relations.
    Note also that cycle relations as in (\ref{item_cycled}) and (\ref{item_broken4cycle}) reduce to cycle relations as in~\cite[Defn.\ 2.2]{GM}.
    Hence in this case the defining relations coincide with those in \cite[Defn.\ 2.2]{GM} and $G_Q$ is isomorphic to $\mathcal{A}(D_n)$, that is the Artin braid group of type $D_n$, by \cite[Remark 2.3 and Theorem 2.12]{GM}. Recall also that $\mathcal{A}(D_n)\cong B(2,2,n)$; see for example \cite[pp 188]{BMR98}.
\end{remark}

\begin{remark} \label{rem:tworelations}
    Note that there are only two relations in (\ref{item_broken4cycle}) corresponding to the $4$-cycle $i\rightarrow j\rightarrow k\rightarrow l\rightarrow i$, in contrast to the four relations appearing in~\cite[Defn.\ 2.2]{GM}:
$$s_is_js_ks_ls_is_j=s_js_ks_ls_is_js_k=s_ks_ls_is_js_ks_l=s_ls_is_js_ks_ls_i.$$
When $d=2$, these relations are equivalent to those in (\ref{item_broken4cycle}) (see Remark~\ref{remark_d2_muchsimpler}), but when $d>2$, we can see that the relations in~\cite[Defn.\ 2.2]{GM} would imply the unexpected relation $s_is_k=s_ks_i$.
In fact, using one of the two equalities from (\ref{item_broken4cycle}), together with the other relations, we have
\begin{align*}
s_is_js_ks_ls_is_j=s_ls_is_js_ks_ls_i&\iff s_js_ks_ls_is_j=s_i^{-1}s_ls_is_js_ks_ls_i
    \iff s_js_ks_ls_is_j=s_ls_is_l^{-1}s_js_ks_ls_i\\
    &\iff s_js_ks_ls_is_j=s_ls_is_js_l^{-1}s_ks_ls_i
    \iff s_js_ks_ls_is_j=s_ls_is_js_ks_ls_k^{-1}s_i\\
    &\iff s_js_ks_ls_is_j\underbrace{s_ks_i\cdots }_{\text{$d-2$ terms}}=s_ls_is_js_ks_ls_k^{-1}s_i\underbrace{s_ks_i\cdots }_{\text{$d-2$ terms}}\\
    &\iff s_js_ks_ls_is_j\underbrace{s_ks_i\cdots }_{\text{$d-2$ terms}}=s_ls_is_js_ks_ls_i\underbrace{s_ks_i\cdots s_x^{-1}}_{\text{$d-1$ terms}}\\
    &\iff s_js_ks_ls_is_js_k\underbrace{s_is_k\cdots s_x}_{\text{$d-2$ terms}}=s_ls_is_js_ks_ls_i\underbrace{s_ks_i\cdots }_{\text{$d-2$ terms}},
\end{align*}
where $x=i$ or $k$ depending on whether $d$ is odd or even respectively. If the four relations from \cite[Defn.\ 2.2]{GM} were true for $d>2$, this would imply that $\underbrace{s_is_k\cdots }_{\text{$d-2$ terms}} =\underbrace{s_ks_i\cdots }_{\text{$d-2$ terms}}$ and so by (\ref{item_dreln}), that $s_ks_i=s_is_k$.
However, this statement is false for $d>2$ (see Remark~\ref{rem:tworelationsB}).
\end{remark}

We can now use our construction to embed the presentation of the group $B(d,d,n)$ given by Brou\'{e}, Malle and Rouquier~\cite{BMR98} $B(d,d,n)$ into a triangulation of the surface $S$ as follows.

\begin{remark}
    Consider the triangulated surface in Figure~\ref{fig:BMR_embedded}, where we have drawn the corresponding quiver following the rules in Definition~\ref{defn_quiver_on_disc}. The associated group $G_Q$ with generators and relations as in 
    Definition~\ref{defn_GQ_associated_group} is exactly the presentation of the group $B(d,d,n)$ given by Brou\'{e}, Malle and Rouquier in~\cite[Table 5]{BMR98}.
\end{remark}

\subsection{Mutation of quivers and triangulations.}\label{subsection_mutation_q_triang}
By~\cite[\S7]{FST}, given a tagged triangulation $T$ of the disk $S$, and a choice of tagged or untagged arc, there is a unique tagged triangulation which coincides with $T$ except for this arc, i.e. the flip of $T$ at the given arc.
We give a way of mutating the quivers constructed as in the previous section that agrees with flipping the triangulation. This will coincide with Fomin-Zelevinsky mutation, see \cite[Lemma 8.5]{FZ}, for the portion of the quiver ``far from the cone point'', but we need different rules for the double edges, labelled (unoriented) edges and labelled cycles.

Note that by the construction of $Q$, all $3$-cycles in $Q$ where arrows have no labels are oriented cyclically by~\cite[Lemma 7.5]{FZ}; see also \cite[pp. 1948]{BM}. On the other hand, we sometimes have an unoriented edge labelled $d$ creating ``unoriented'' cycles, see for example the situation of (\ref{item_broken4cycle}) in Definition~\ref{defn_GQ_associated_group}. 
\begin{defn}\label{defn_mutationrules}
    Let $Q$ be a quiver as in Definition~\ref{defn_quiver_on_disc} and $k$ be a vertex in $Q$. We define \textit{the mutation of $Q$ at $k$}, denoted by $\mu_k(Q)$, as the following quiver on the same vertex set. See Figure~\ref{fig:mutations} for a pictorial representation of the following rules.
    \begin{enumerate}
        \item Reverse the orientations of all (oriented) arrows in $Q$ incident with $k$.
        \item\label{item_mutatepath_jki} For any path of the form $j\rightarrow k \rightarrow i$ in $Q$:
        \begin{itemize}
            \item if there is no arrow between $i$ and $j$ in $Q$, then there is an arrow $j\rightarrow i$ in $\mu_k(Q)$,
            \item if there is an arrow $i\rightarrow j$ and $j$, $k$, $i$ do not form a  $3$-cycle labelled $\begin{tikzpicture}[scale=1,
    conepoint/.style={gray, circle, draw=gray!100, fill=white!100, thick, inner sep=1.5pt}]
\node[conepoint]  at (0,-0.2) {\scriptsize $d$};
\end{tikzpicture}$  in $Q$, then there is no arrow between $i$ and $j$ in $\mu_k(Q)$,
            \item if $j$, $k$, $i$ form a $3$-cycle labelled $\begin{tikzpicture}[scale=1,
    conepoint/.style={gray, circle, draw=gray!100, fill=white!100, thick, inner sep=1.5pt}]
\node[conepoint]  at (0,-0.2) {\scriptsize $d$};
\end{tikzpicture}$ in $Q$, then there is an edge labelled $d$ between $i$ and $j$ in $\mu_k(Q)$, the $3$-cycle looses the label and the neighbours of $i$, $j$ in $\mu_k(Q)$ acquire a double edge towards the labelled edge in $\mu_k(Q)$,
            \item if there is an edge labelled $d$ between  $j$ and $i$ in $Q$, then there is an arrow $j\rightarrow i$ in $\mu_k(Q)$, the $3$-cycle $j$, $i$, $k$ acquires the label $\begin{tikzpicture}[scale=1,
    conepoint/.style={gray, circle, draw=gray!100, fill=white!100, thick, inner sep=1.5pt}]
\node[conepoint]  at (0,-0.2) {\scriptsize $d$};
\end{tikzpicture}$ and all double edges are removed in $\mu_k(Q)$,
        \end{itemize} 
        \item If $j\rightarrow k \rightarrow i$  is part of an $r$-cycle labelled $\begin{tikzpicture}[scale=1,
    conepoint/.style={gray, circle, draw=gray!100, fill=white!100, thick, inner sep=1.5pt}]
\node[conepoint]  at (0,-0.2) {\scriptsize $d$};
\end{tikzpicture}$ in $Q$ for $r\geq 4$, then the rules above apply and the label is kept in the $(r-1)$-cycle including $i$, $j$ but not $k$ in $\mu_k(Q)$.
        \item If there are arrows $j\rightarrow k \rightarrow i$ and $i$, $j$, but not $k$, are part of an $r$-cycle labelled $\begin{tikzpicture}[scale=1,
    conepoint/.style={gray, circle, draw=gray!100, fill=white!100, thick, inner sep=1.5pt}]
\node[conepoint]  at (0,-0.2) {\scriptsize $d$};
\end{tikzpicture}$ in $Q$ for $r\geq 3$, then the rules above apply and the label is kept in the $(r+1)$-cycle which includes $i\rightarrow k \rightarrow j$ in $\mu_k(Q)$.
\item
If in $Q$ there is one of the following configurations:
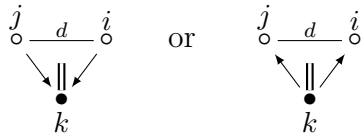
\begin{align*}
    \begin{tikzpicture}[scale=1,
  quiverarrow/.style={black, -latex},
  conepoint/.style={gray, circle, draw=gray!100, fill=white!100, thick, inner sep=1.5pt},
  db/.style={thick, double, double distance=1.3pt, shorten <=-6pt}]
  \begin{scope}[shift={(0,0)}]
\node (j2) at (0,0) {$\circ$};
\node [above] at (j2) {\small {$j$}}; 
\node (i2) at (1.2,0) {$\circ$};
\node [above] at (i2) {\small {$i$}};
\node (k2) at (0.6,-0.8) {$\bullet$};
\node[below] at (k2) {\small {$k$}};
\draw [quiverarrow, shorten <=-2pt, shorten >=-2pt] (j2) -- (k2);
\draw [quiverarrow, shorten <=-2pt, shorten >=-2pt] (i2) -- (k2);
\draw [-, shorten <=-2pt, shorten >=-2pt] (i2) -- (j2);
\draw [db] (0.6,-0.45)--(0.6,-0.3);
\node at (0.6,0.15) {\tiny {$d$}};
\node at (2.2,0) { or };
\begin{scope}[shift={(3.3,0)}]
\node (j2) at (0,0) {$\circ$};
\node [above] at (j2) {\small {$j$}}; 
\node (i2) at (1.2,0) {$\circ$};
\node [above] at (i2) {\small {$i$}};
\node (k2) at (0.6,-0.8) {$\bullet$};
\node[below] at (k2) {\small {$k$}};
\draw [quiverarrow, shorten <=-2pt, shorten >=-2pt] (k2) -- (j2);
\draw [quiverarrow, shorten <=-2pt, shorten >=-2pt] (k2) -- (i2);
\draw [-, shorten <=-2pt, shorten >=-2pt] (i2) -- (j2);
\draw [db] (0.6,-0.45)--(0.6,-0.3);
\node at (0.6,0.15) {\tiny {$d$}};
\end{scope}
\end{scope}
  \end{tikzpicture}
\end{align*}
then  follow the above rules (keeping also the double edge at $k$) and
\begin{itemize}
    \item if there is a vertex $l$ different from $k$ with a double edge  in $Q$, remove the double edge at $l$ in $\mu_k(Q)$,
    \item if there is a vertex $n$ different from $l$ and $k$ such that in $Q$  
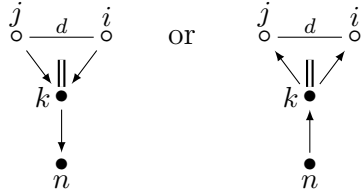
    \begin{align*}
    \begin{tikzpicture}[scale=1,
  quiverarrow/.style={black, -latex},
  conepoint/.style={gray, circle, draw=gray!100, fill=white!100, thick, inner sep=1.5pt},
  db/.style={thick, double, double distance=1.3pt, shorten <=-6pt}]
  \begin{scope}[shift={(0,0)}]
\node (j2) at (0,0) {$\circ$};
\node [above] at (j2) {\small {$j$}}; 
\node (i2) at (1.2,0) {$\circ$};
\node [above] at (i2) {\small {$i$}};
\node (k2) at (0.6,-0.8) {$\bullet$};
\node[left] at (k2) {\small {$k$}};
\node (n2) at (0.6,-1.7) {$\bullet$};
\node[below] at (n2) {\small {$n$}};
\draw [quiverarrow, shorten <=-2pt, shorten >=-2pt] (j2) -- (k2);
\draw [quiverarrow, shorten <=-2pt, shorten >=-2pt] (i2) -- (k2);
\draw [quiverarrow, shorten <=-2pt, shorten >=-2pt] (k2) -- (n2);
\draw [-, shorten <=-2pt, shorten >=-2pt] (i2) -- (j2);
\draw [db] (0.6,-0.45)--(0.6,-0.3);
\node at (0.6,0.15) {\tiny {$d$}};
\node at (2.2,0) { or };
\begin{scope}[shift={(3.3,0)}]
\node (j2) at (0,0) {$\circ$};
\node [above] at (j2) {\small {$j$}}; 
\node (i2) at (1.2,0) {$\circ$};
\node [above] at (i2) {\small {$i$}};
\node (k2) at (0.6,-0.8) {$\bullet$};
\node[left] at (k2) {\small {$k$}};
\node (n2) at (0.6,-1.7) {$\bullet$};
\node[below] at (n2) {\small {$n$}};
\draw [quiverarrow, shorten <=-2pt, shorten >=-2pt] (k2) -- (j2);
\draw [quiverarrow, shorten <=-2pt, shorten >=-2pt] (k2) -- (i2);
\draw [quiverarrow, shorten <=-2pt, shorten >=-2pt] (n2) -- (k2);
\draw [-, shorten <=-2pt, shorten >=-2pt] (i2) -- (j2);
\draw [db] (0.6,-0.45)--(0.6,-0.3);
\node at (0.6,0.15) {\tiny {$d$}};
\end{scope}
\end{scope}
  \end{tikzpicture}
\end{align*}
    then add a double edge at $n$.
\end{itemize}
    \end{enumerate}
\end{defn}

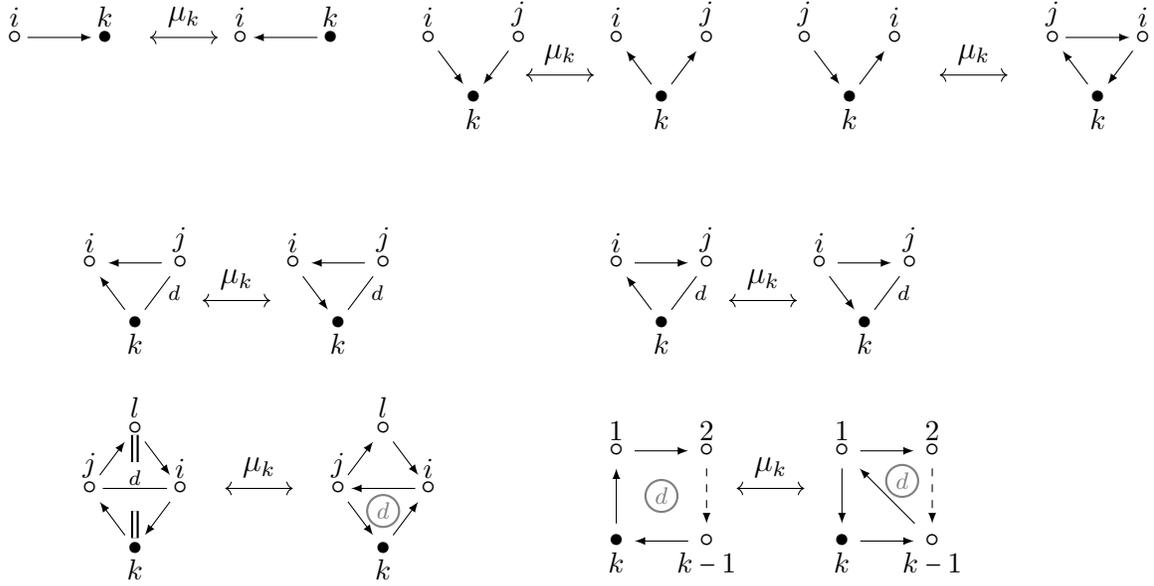
\begin{figure}
    \centering
\begin{tikzpicture}[scale=1,
  quiverarrow/.style={black, -latex},
  conepoint/.style={gray, circle, draw=gray!100, fill=white!100, thick, inner sep=1.5pt},
  db/.style={thick, double, double distance=1.3pt, shorten <=-6pt}] 
\begin{scope}[shift={(-1,0)}]
\node (i1) at (0,0) {$\circ$};
\node [above] at (i1) {\small {$i$}}; 
\node (k1) at (1.2,0) {$\bullet$};
\node[above] at (k1) {\small {$k$}};
\draw [quiverarrow, shorten <=-2pt, shorten >=-2pt] (i1) -- (k1);
\draw [<->] (1.8,0) to node[above] {$\mu_k$} (2.7,0);
\begin{scope}[shift={(3,0)}]
\node (i1) at (0,0) {$\circ$};
\node [above] at (i1) {\small {$i$}}; 
\node (k1) at (1.2,0) {$\bullet$};
\node[above] at (k1) {\small {$k$}};
\draw [quiverarrow, shorten <=-2pt, shorten >=-2pt] (k1) -- (i1);
\end{scope}
\end{scope}

\begin{scope}[shift={(4.5,0)}]
\node (i1) at (0,0) {$\circ$};
\node [above] at (i1) {\small {$i$}}; 
\node (j1) at (1.2,0) {$\circ$};
\node [above] at (j1) {\small {$j$}};
\node (k1) at (0.6,-0.8) {$\bullet$};
\node[below] at (k1) {\small {$k$}};
\draw [quiverarrow, shorten <=-2pt, shorten >=-2pt] (i1) -- (k1);
\draw [quiverarrow, shorten <=-2pt, shorten >=-2pt] (j1) -- (k1);
\draw [<->] (1.3,-0.5) to node[above] {$\mu_k$} (2.2,-0.5);
\begin{scope}[shift={(2.5,0)}]
\node (i2) at (0,0) {$\circ$};
\node [above] at (i2) {\small {$i$}}; 
\node (j2) at (1.2,0) {$\circ$};
\node [above] at (j2) {\small {$j$}};
\node (k2) at (0.6,-0.8) {$\bullet$};
\node[below] at (k2) {\small {$k$}};
\draw [quiverarrow, shorten <=-2pt, shorten >=-2pt] (k2) -- (i2);
\draw [quiverarrow, shorten <=-2pt, shorten >=-2pt] (k2) -- (j2);
\end{scope}
\end{scope}

\begin{scope}[shift={(9.5,0)}]
\node (j1) at (0,0) {$\circ$};
\node [above] at (j1) {\small {$j$}}; 
\node (i1) at (1.2,0) {$\circ$};
\node [above] at (i1) {\small {$i$}};
\node (k1) at (0.6,-0.8) {$\bullet$};
\node[below] at (k1) {\small {$k$}};
\draw [quiverarrow, shorten <=-2pt, shorten >=-2pt] (j1) -- (k1);
\draw [quiverarrow, shorten <=-2pt, shorten >=-2pt] (k1) -- (i1);
\draw [<->] (1.8,-0.5) to node[above] {$\mu_k$} (2.7,-0.5);
\begin{scope}[shift={(3.3,0)}]
\node (j2) at (0,0) {$\circ$};
\node [above] at (j2) {\small {$j$}}; 
\node (i2) at (1.2,0) {$\circ$};
\node [above] at (i2) {\small {$i$}};
\node (k2) at (0.6,-0.8) {$\bullet$};
\node[below] at (k2) {\small {$k$}};
\draw [quiverarrow, shorten <=-2pt, shorten >=-2pt] (k2) -- (j2);
\draw [quiverarrow, shorten <=-2pt, shorten >=-2pt] (i2) -- (k2);
\draw [quiverarrow, shorten <=-2pt, shorten >=-2pt] (j2) -- (i2);
\end{scope}
\end{scope}

\begin{scope}[shift={(0,-3)}]
\node (i1) at (0,0) {$\circ$};
\node [above] at (i1) {\small {$i$}}; 
\node (j1) at (1.2,0) {$\circ$};
\node [above] at (j1) {\small {$j$}};
\node (k1) at (0.6,-0.8) {$\bullet$};
\node[below] at (k1) {\small {$k$}};
\draw [quiverarrow, shorten <=-2pt, shorten >=-2pt] (k1) -- (i1);
\draw [quiverarrow] (j1) -- (i1);
\draw [-, shorten <=-2pt, shorten >=-2pt] (k1) to node[right] {\tiny $d$} (j1);
\draw [<->] (1.5,-0.5) to node[above] {$\mu_k$} (2.4,-0.5);
\begin{scope}[shift={(2.7,0)}]
\node (i2) at (0,0) {$\circ$};
\node [above] at (i2) {\small {$i$}}; 
\node (j2) at (1.2,0) {$\circ$};
\node [above] at (j2) {\small {$j$}};
\node (k2) at (0.6,-0.8) {$\bullet$};
\node[below] at (k2) {\small {$k$}};
\draw [quiverarrow, shorten <=-2pt, shorten >=-2pt] (i2) -- (k2);
\draw [-, shorten <=-2pt, shorten >=-2pt] (j2)  to node[right] {\tiny $d$} (k2);
\draw [quiverarrow] (j2) -- (i2);
\end{scope}
\begin{scope}[shift={(7,0)}]
\node (i1) at (0,0) {$\circ$};
\node [above] at (i1) {\small {$i$}}; 
\node (j1) at (1.2,0) {$\circ$};
\node [above] at (j1) {\small {$j$}};
\node (k1) at (0.6,-0.8) {$\bullet$};
\node[below] at (k1) {\small {$k$}};
\draw [quiverarrow, shorten <=-2pt, shorten >=-2pt] (k1) -- (i1);
\draw [quiverarrow] (i1) -- (j1);
\draw [-, shorten <=-2pt, shorten >=-2pt] (j1) to node[right] {\tiny $d$} (k1);
\draw [<->] (1.5,-0.5) to node[above] {$\mu_k$} (2.4,-0.5);
\begin{scope}[shift={(2.7,0)}]
\node (i2) at (0,0) {$\circ$};
\node [above] at (i2) {\small {$i$}}; 
\node (j2) at (1.2,0) {$\circ$};
\node [above] at (j2) {\small {$j$}};
\node (k2) at (0.6,-0.8) {$\bullet$};
\node[below] at (k2) {\small {$k$}};
\draw [quiverarrow, shorten <=-2pt, shorten >=-2pt] (i2) -- (k2);
\draw [-, shorten <=-2pt, shorten >=-2pt] (k2)  to node[right] {\tiny $d$} (j2);
\draw [quiverarrow] (i2) -- (j2);
\end{scope}
\end{scope}
\end{scope}

\begin{scope}[shift={(0,-6)}]
\node (j2) at (0,0) {$\circ$};
\node [above] at (j2) {\small {$j$}}; 
\node (i2) at (1.2,0) {$\circ$};
\node [above] at (i2) {\small {$i$}};
\node (k2) at (0.6,-0.8) {$\bullet$};
\node[below] at (k2) {\small {$k$}};
\node (l2) at (0.6,0.8) {$\circ$};
\node [above] at (l2) {\small {$l$}};
\draw [quiverarrow, shorten <=-2pt, shorten >=-2pt] (k2) -- (j2);
\draw [quiverarrow, shorten <=-2pt, shorten >=-2pt] (i2) -- (k2);
\draw [-, shorten <=-2pt, shorten >=-2pt] (i2) -- (j2);
\draw [quiverarrow, shorten <=-2pt, shorten >=-2pt] (j2) -- (l2);
\draw [quiverarrow, shorten <=-2pt, shorten >=-2pt] (l2) -- (i2);
\draw [db] (0.6,0.5)--(0.6,0.35);
\draw [db] (0.6,-0.45)--(0.6,-0.3);
\node at (0.6,0.15) {\tiny {$d$}};
\draw [<->] (1.8,0) to node[above] {$\mu_k$} (2.7,0);
\begin{scope}[shift={(3.3,0)}]
\node (j1) at (0,0) {$\circ$};
\node [above] at (j1) {\small {$j$}}; 
\node (i1) at (1.2,0) {$\circ$};
\node [above] at (i1) {\small {$i$}};
\node (k1) at (0.6,-0.8) {$\bullet$};
\node[below] at (k1) {\small {$k$}};
\node (l1) at (0.6,0.8) {$\circ$};
\node [above] at (l1) {\small {$l$}};
\draw [quiverarrow, shorten <=-2pt, shorten >=-2pt] (j1) -- (k1);
\draw [quiverarrow, shorten <=-2pt, shorten >=-2pt] (k1) -- (i1);
\draw [quiverarrow, shorten <=-2pt, shorten >=-2pt] (i1) -- (j1);
\draw [quiverarrow, shorten <=-2pt, shorten >=-2pt] (j1) -- (l1);
\draw [quiverarrow, shorten <=-2pt, shorten >=-2pt] (l1) -- (i1);
\node[conepoint] at (0.6,-0.3) {\scriptsize {$d$}};
\end{scope}
\end{scope}

\begin{scope}[shift={(7,-6.5)}]
\node (k) at (0,-0.2) {$\bullet$};
\node [below] at (k) {\small {$k$}};
\node (k-1) at (1.2,-0.2) {$\circ$};
\node [below] at (k-1) {\small {$k-1$}}; 
\node (v2) at (1.2,1) {$\circ$};
\node [above] at (v2) {\small {$2$}};
\node (v1) at (0,1) {$\circ$};
\node [above] at (v1) {\small {$1$}}; 
 \draw[quiverarrow] (v1) -- (v2);
 \draw[quiverarrow, dashed] (v2) -- (k-1);
 \draw[quiverarrow] (k-1) -- (k);
 \draw[quiverarrow] (k) -- (v1);
 \node[conepoint] at (0.6,0.4) {\scriptsize $d$};
 \draw [<->] (1.6,0.5) to node[above] {$\mu_k$} (2.5,0.5);
 \begin{scope}[shift={(3,0)}]
     \node (k) at (0,-0.2) {$\bullet$};
\node [below] at (k) {\small {$k$}};
\node (k-1) at (1.2,-0.2) {$\circ$};
\node [below] at (k-1) {\small {$k-1$}}; 
\node (v2) at (1.2,1) {$\circ$};
\node [above] at (v2) {\small {$2$}};
\node (v1) at (0,1) {$\circ$};
\node [above] at (v1) {\small {$1$}}; 
 \draw[quiverarrow] (v1) -- (v2);
 \draw[quiverarrow, dashed] (v2) -- (k-1);
 \draw[quiverarrow] (k) -- (k-1);
 \draw[quiverarrow] (k-1) -- (v1);
 \draw[quiverarrow] (v1) -- (k);
 \node[conepoint] at (0.8,0.6) {\scriptsize $d$};
 \end{scope}
\end{scope}
\end{tikzpicture}
 \caption{Local mutations following the rules from Definition~\ref{defn_mutationrules}.
 In the third line, in the first mutation vertex $l$ might not exist (if it corresponds to a boundary segment of the disc); in the second mutation we are assuming $k$ is at least $4$.
 }
    \label{fig:mutations}
\end{figure}

\begin{remark}
    Note that when $d=2$ the mutation rules simplify. In fact, as explained in Remark~\ref{remark_d2_muchsimpler}, arrows labelled $2$, double edges and labels on cycles can be omitted and we simply recover Fomin-Zelevinsky mutation.
\end{remark}

\begin{lemma}
\label{lem:flipmutation}
Let $T$ and $T'$ be tagged triangulations of $S$ such that $T'$ is obtained from $T$ by flipping the arc $\alpha$ as in~\cite[\S7]{FST}. Let $Q_T$ and $Q_{T'}$ be the corresponding quivers as in Definition~\ref{defn_quiver_on_disc}. Then 
\begin{itemize}
\item[(a)] The flip of $\alpha$ is given locally by one of the mutations in Figures~\ref{fig:typeAmutation},
\ref{fig:cornercase},~\ref{fig:mutationnearcone},~\ref{fig:mutation_tagged} or~\ref{fig:mutation_untagged} (from left to right or right to left), or by one of the mutations
from Figure~\ref{fig:cornercase}, \ref{fig:mutation_tagged} or \ref{fig:mutation_untagged} with all of the tags flipped.
\item[(b)] The quiver $Q_{T'}$ can be obtained from the quiver $Q_T$ by applying the mutation rule in Definition~\ref{defn_mutationrules}.
\end{itemize}
\end{lemma}
\begin{proof}
Part (a) follows from Remarks~\ref{rem:puzzlepieces} and~\ref{rem:conepointportion} on consideration of which vertex is the conepoint when gluing together puzzle pieces.
Part (b) follows from part (a) by computing the quiver in each case before and after mutation (see Figure~\ref{fig:puzzlepiecesquiver}).
\end{proof}

\subsection{Mutation of groups}
In this section we will show that the group associated in Definition~\ref{defn_GQ_associated_group} to a quiver as in Definition~\ref{defn_quiver_on_disc} is invariant under the mutation introduced in Definition~\ref{defn_mutationrules}:

\begin{theorem}
\label{thm:mutation_invariance}
    Let $Q$ be a quiver as in Definition~\ref{defn_quiver_on_disc}, $G_Q$ its associated group as in Definition~\ref{defn_GQ_associated_group} with generators $s_i$, and $k$ a vertex of $Q$. Let $\mu_k(Q)$ be the mutation of $Q$ at $k$ as in Definition~\ref{defn_mutationrules}, and let $t_i$ be the generators of $G_{\mu_k(Q)}$. Then there is a group isomorphism $\varphi^Q_k:G_{Q}\cong G_{\mu_k(Q)}$ given by
$\varphi^Q_k(s_i)=t_kt_it_k^{-1}$
if 
$i\rightarrow k$ in $Q$ or the vertices $i$ and $k$ correspond to the only two arcs incident with the conepoint in $T$
and the arc corresponding to $k$ is rotated anti-clockwise to the flipped arc; and $\varphi(s_i)=t_i$ otherwise.
\end{theorem}

Note that, in the above theorem, the situation when there are only two arcs incident with the conepoint in $T$ means that there is an unoriented edge between the corresponding vertices, labelled $d$.

We have already seen that the triangulation from Figure~\ref{fig:BMR_embedded} satisfies $G_Q\cong B(d,d,n)$, so this will allow us to conclude that every tagged triangulation gives a presentation of the group $B(d,d,n)$.

\begin{proposition}
\label{prop:GMcase}
\cite[Prop.\ 2.9]{GM}
Let $Q$ be one of the quivers on the left or right of (a)--(f) in Figure~\ref{fig:GMcases}. Let $k$ be a vertex of $Q$. Let $Q'=\mu_k(Q)$ be the quiver obtained from $Q$ by mutating at $k$. Suppose that the $t_i$ are elements of a group satisfying the defining relations (1), (2) and (3) from Definition~\ref{defn_GQ_associated_group} for the quiver $Q'$.
For $i\in Q_0$, let
$$S_i=\begin{cases}
t_kt_it_k^{-1}, & i\rightarrow k \text{ in $Q$}; \\
t_i, & else.
\end{cases}
$$
Then, for each of the cases in Figure~\ref{fig:GMcases}, the elements $S_i$ satisfy the defining relations (1), (2) and (3) of Definition~\ref{defn_GQ_associated_group} for the quiver $Q$.
\end{proposition}

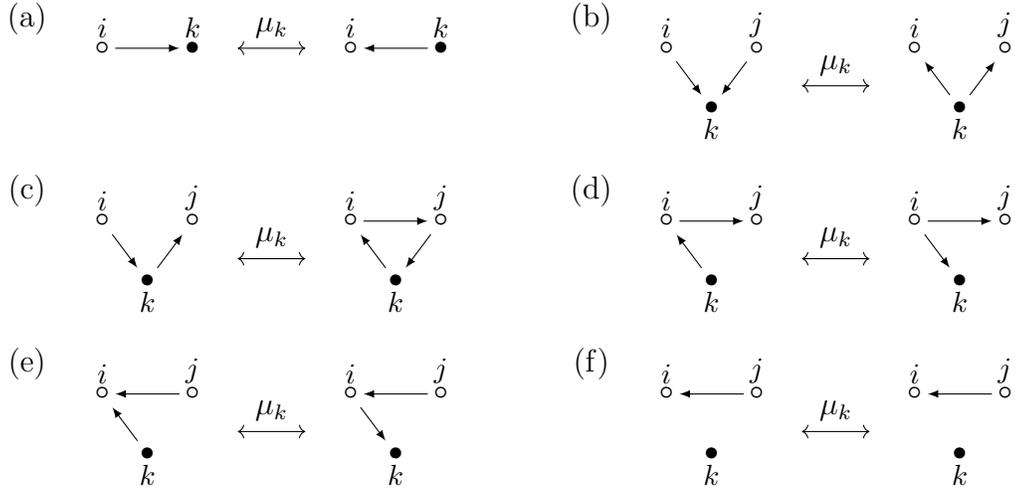
\begin{figure}
$$
\begin{tikzpicture}[scale=1,
  quiverarrow/.style={black, -latex}] 
\begin{scope}[shift={(0,0)}]
\node at (-1,0.4) {(a)};
\node (i1) at (0,0) {$\circ$};
\node [above] at (i1) {\small {$i$}}; 
\node (k1) at (1.2,0) {$\bullet$};
\node[above] at (k1) {\small {$k$}};
\draw [quiverarrow, shorten <=-2pt, shorten >=-2pt] (i1) -- (k1);
\draw [<->] (1.8,0) to node[above] {$\mu_k$} (2.7,0);
\begin{scope}[shift={(3.3,0)}]
\node (i1) at (0,0) {$\circ$};
\node [above] at (i1) {\small {$i$}}; 
\node (k1) at (1.2,0) {$\bullet$};
\node[above] at (k1) {\small {$k$}};
\draw [quiverarrow, shorten <=-2pt, shorten >=-2pt] (k1) -- (i1);
\end{scope}
\end{scope}

\begin{scope}[shift={(7.5,0)}]
\node at (-1,0.4) {(b)};
\node (i1) at (0,0) {$\circ$};
\node [above] at (i1) {\small {$i$}}; 
\node (j1) at (1.2,0) {$\circ$};
\node [above] at (j1) {\small {$j$}};
\node (k1) at (0.6,-0.8) {$\bullet$};
\node[below] at (k1) {\small {$k$}};
\draw [quiverarrow, shorten <=-2pt, shorten >=-2pt] (i1) -- (k1);
\draw [quiverarrow, shorten <=-2pt, shorten >=-2pt] (j1) -- (k1);
\draw [<->] (1.8,-0.5) to node[above] {$\mu_k$} (2.7,-0.5);
\begin{scope}[shift={(3.3,0)}]
\node (i2) at (0,0) {$\circ$};
\node [above] at (i2) {\small {$i$}}; 
\node (j2) at (1.2,0) {$\circ$};
\node [above] at (j2) {\small {$j$}};
\node (k2) at (0.6,-0.8) {$\bullet$};
\node[below] at (k2) {\small {$k$}};
\draw [quiverarrow, shorten <=-2pt, shorten >=-2pt] (k2) -- (i2);
\draw [quiverarrow, shorten <=-2pt, shorten >=-2pt] (k2) -- (j2);
\end{scope}
\end{scope}

\begin{scope}[shift={(0,-2.3)}]
\node at (-1,0.4) {(c)};
\node (i1) at (0,0) {$\circ$};
\node [above] at (i1) {\small {$i$}}; 
\node (j1) at (1.2,0) {$\circ$};
\node [above] at (j1) {\small {$j$}};
\node (k1) at (0.6,-0.8) {$\bullet$};
\node[below] at (k1) {\small {$k$}};
\draw [quiverarrow, shorten <=-2pt, shorten >=-2pt] (i1) -- (k1);
\draw [quiverarrow, shorten <=-2pt, shorten >=-2pt] (k1) -- (j1);
\draw [<->] (1.8,-0.5) to node[above] {$\mu_k$} (2.7,-0.5);
\begin{scope}[shift={(3.3,0)}]
\node (i2) at (0,0) {$\circ$};
\node [above] at (i2) {\small {$i$}}; 
\node (j2) at (1.2,0) {$\circ$};
\node [above] at (j2) {\small {$j$}};
\node (k2) at (0.6,-0.8) {$\bullet$};
\node[below] at (k2) {\small {$k$}};
\draw [quiverarrow, shorten <=-2pt, shorten >=-2pt] (k2) -- (i2);
\draw [quiverarrow, shorten <=-2pt, shorten >=-2pt] (j2) -- (k2);
\draw [quiverarrow, shorten <=-2pt, shorten >=-2pt] (i2) -- (j2);
\end{scope}
\end{scope}

\begin{scope}[shift={(7.5,-2.3)}]
\node at (-1,0.4) {(d)};
\node (i1) at (0,0) {$\circ$};
\node [above] at (i1) {\small {$i$}}; 
\node (j1) at (1.2,0) {$\circ$};
\node [above] at (j1) {\small {$j$}};
\node (k1) at (0.6,-0.8) {$\bullet$};
\node[below] at (k1) {\small {$k$}};
\draw [quiverarrow, shorten <=-2pt, shorten >=-2pt] (k1) -- (i1);
\draw [quiverarrow, shorten <=-2pt, shorten >=-2pt] (i1) -- (j1);
\draw [<->] (1.8,-0.5) to node[above] {$\mu_k$} (2.7,-0.5);
\begin{scope}[shift={(3.3,0)}]
\node (i2) at (0,0) {$\circ$};
\node [above] at (i2) {\small {$i$}}; 
\node (j2) at (1.2,0) {$\circ$};
\node [above] at (j2) {\small {$j$}};
\node (k2) at (0.6,-0.8) {$\bullet$};
\node[below] at (k2) {\small {$k$}};
\draw [quiverarrow, shorten <=-2pt, shorten >=-2pt] (i2) -- (k2);
\draw [quiverarrow, shorten <=-2pt, shorten >=-2pt] (i2) -- (j2);
\end{scope}
\end{scope}

\begin{scope}[shift={(0,-4.6)}]
\node at (-1,0.4) {(e)};
\node (i1) at (0,0) {$\circ$};
\node [above] at (i1) {\small {$i$}}; 
\node (j1) at (1.2,0) {$\circ$};
\node [above] at (j1) {\small {$j$}};
\node (k1) at (0.6,-0.8) {$\bullet$};
\node[below] at (k1) {\small {$k$}};
\draw [quiverarrow, shorten <=-2pt, shorten >=-2pt] (k1) -- (i1);
\draw [quiverarrow, shorten <=-2pt, shorten >=-2pt] (j1) -- (i1);
\draw [<->] (1.8,-0.5) to node[above] {$\mu_k$} (2.7,-0.5);
\begin{scope}[shift={(3.3,0)}]
\node (i2) at (0,0) {$\circ$};
\node [above] at (i2) {\small {$i$}}; 
\node (j2) at (1.2,0) {$\circ$};
\node [above] at (j2) {\small {$j$}};
\node (k2) at (0.6,-0.8) {$\bullet$};
\node[below] at (k2) {\small {$k$}};
\draw [quiverarrow, shorten <=-2pt, shorten >=-2pt] (i2) -- (k2);
\draw [quiverarrow, shorten <=-2pt, shorten >=-2pt] (j2) -- (i2);
\end{scope}
\end{scope}

\begin{scope}[shift={(7.5,-4.6)}]
\node at (-1,0.4) {(f)};
\node (i1) at (0,0) {$\circ$};
\node [above] at (i1) {\small {$i$}}; 
\node (j1) at (1.2,0) {$\circ$};
\node [above] at (j1) {\small {$j$}};
\node (k1) at (0.6,-0.8) {$\bullet$};
\node[below] at (k1) {\small {$k$}};
\draw [quiverarrow, shorten <=-2pt, shorten >=-2pt] (j1) -- (i1);
\draw [<->] (1.8,-0.5) to node[above] {$\mu_k$} (2.7,-0.5);
\begin{scope}[shift={(3.3,0)}]
\node (i2) at (0,0) {$\circ$};
\node [above] at (i2) {\small {$i$}}; 
\node (j2) at (1.2,0) {$\circ$};
\node [above] at (j2) {\small {$j$}};
\node (k2) at (0.6,-0.8) {$\bullet$};
\node[below] at (k2) {\small {$k$}};
\draw [quiverarrow, shorten <=-2pt, shorten >=-2pt] (j2) -- (i2);
\end{scope}
\end{scope}
\end{tikzpicture}
$$
\caption{Cases for Proposition~\ref{prop:GMcase}}
\label{fig:GMcases}
\end{figure}

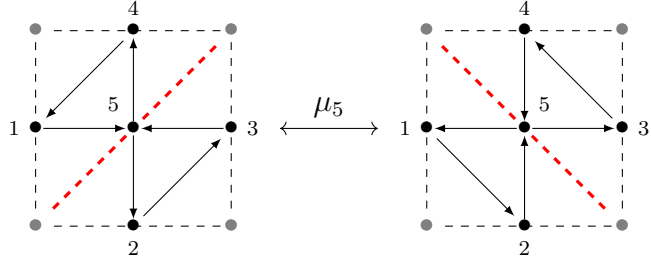
\begin{figure}
\begin{tikzpicture}[scale=1.3,
  quiverarrow/.style={black,-latex,shorten <=-4pt, shorten >=-4pt},
  mutationarc/.style={dashed, red, very thick, shorten <=-4pt, shorten >=-4pt},
  mutationarccone/.style={dashed, red, very thick, shorten >=-4pt},
  arc/.style={dashed, black, shorten <=-4pt, shorten >=-4pt},
  arccone/.style={dashed, black, shorten >=-4pt},
  point/.style={gray},
  vertex/.style={black},
  conepoint/.style={gray, circle, draw=gray!100, fill=white!100, thick, inner sep=1.5pt},
  db/.style={thick, double, double distance=1.3pt, shorten <=-6pt}
  ]
   \begin{scope}
   \node[point] (p1) at (0,0) {$\bullet$};
   \node[point] (p2) at (2,0) {$\bullet$};
   \node[point] (p3) at (0,2) {$\bullet$};
   \node[point] (p4) at (2,2) {$\bullet$};
  \draw[arccone] (p1)--(p2);
  \draw[arccone] (p1)--(p3);
  \draw[arccone] (p2)--(p4);
  \draw[arccone] (p3)--(p4);
  \draw[mutationarccone] (p1)--(p4);
  \draw[gray] (0,0) node{$\bullet$};
  \draw[gray] (2,0) node{$\bullet$};
  \draw[gray] (0,2) node{$\bullet$};
  \draw[gray] (2,2) node{$\bullet$};
  \node[vertex, label=
{[label distance=-5pt]180:{\tiny $1$}}] (v1) at (0,1) {$\bullet$};
\node[vertex, label=
{[label distance=-5pt]-90:{\tiny $2$}}] (v2) at (1,0) {$\bullet$};
\node[vertex, label=
{[label distance=-5pt]0:{\tiny $3$}}] (v3) at (2,1) {$\bullet$};
\node[vertex, label=
{[label distance=-5pt]90:{\tiny $4$}}] (v4) at (1,2) {$\bullet$};
\node[vertex, label=
{[label distance=-5pt]120:{\tiny $5$}}] (v5) at (1,1) {$\bullet$};
  \draw[quiverarrow] (v1)--(v5);
  \draw[quiverarrow] (v5)--(v4);
  \draw[quiverarrow] (v4)--(v1);
  \draw[quiverarrow] (v3)--(v5);
  \draw[quiverarrow] (v5)--(v2);
  \draw[quiverarrow] (v2)--(v3);
\draw [<->] (2.5,1) to node[above] {$\mu_5$} (3.5,1);
 \end{scope}
  \begin{scope}[xshift=4cm]
  \node[point] (p1) at (0,0) {$\bullet$};
   \node[point] (p2) at (2,0) {$\bullet$};
   \node[point] (p3) at (0,2) {$\bullet$};
   \node[point] (p4) at (2,2) {$\bullet$};
  \draw[arccone] (p1)--(p2);
  \draw[arccone] (p1)--(p3);
  \draw[arccone] (p2)--(p4);
  \draw[arccone] (p3)--(p4);
  \draw[mutationarccone] (p2)--(p3);
  \draw[gray] (0,0) node{$\bullet$};
  \draw[gray] (2,0) node{$\bullet$};
  \draw[gray] (0,2) node{$\bullet$};
  \draw[gray] (2,2) node{$\bullet$};
  \node[vertex, label=
{[label distance=-5pt]180:{\tiny $1$}}] (v1) at (0,1) {$\bullet$};
\node[vertex, label=
{[label distance=-5pt]-90:{\tiny $2$}}] (v2) at (1,0) {$\bullet$};
\node[vertex, label=
{[label distance=-5pt]0:{\tiny $3$}}] (v3) at (2,1) {$\bullet$};
\node[vertex, label=
{[label distance=-5pt]90:{\tiny $4$}}] (v4) at (1,2) {$\bullet$};
\node[vertex, label=
{[label distance=-5pt]60:{\tiny $5$}}] (v5) at (1,1) {$\bullet$};
  \draw[quiverarrow] (v5)--(v1);
  \draw[quiverarrow] (v4)--(v5);
  \draw[quiverarrow] (v1)--(v2);
  \draw[quiverarrow] (v5)--(v3);
  \draw[quiverarrow] (v2)--(v5);
  \draw[quiverarrow] (v3)--(v4);
 \end{scope}
  \end{tikzpicture}
  \caption{A mutation far from the conepoint.}
    \label{fig:typeAmutation}
  \end{figure}

\begin{lemma}
\label{lem:mutationtypeA}
Let $Q$ be the quiver on the left of Figure~\ref{fig:typeAmutation} and $Q'$ the quiver on the right. Let $G_Q$ (respectively $G_{Q'}$) be the group with generators $s_i$ (respectively $t_i$) with $1\leq i\leq 5$ satisfying the relations associated with $Q$ (respectively $Q'$). Then, there are group homomorphisms:
\begin{itemize}\setlength\itemsep{0.8em}
    \item $\varphi^Q_5: G_{Q}\rightarrow G_{Q'}$ with
    $\varphi^Q_5(s_1)=S_1=t_5t_1t_5^{-1}$, $\varphi^Q_5(s_3)=S_3=t_5t_3t_5^{-1}$, $\varphi^Q_5(s_i)=S_i=t_i$ for $i=2,4,5$;
    \item $\varphi^{Q'}_5: G_{Q'}\rightarrow G_Q$ with $\varphi_5(t_4)=T_4=s_5s_4s_5^{-1}$, $\varphi_5(t_2)=T_2=s_5s_2s_5^{-1}$, $\varphi_5(t_i)=T_i=s_i$ for $i=1,3,5$.
\end{itemize}
\end{lemma}
\begin{proof}
For the first statement, it is enough to check that the elements $S_i$ satisfy the defining relations of $G_Q$. This follows directly from Proposition~\ref{prop:GMcase}.
The proof of the second statement is similar.
\end{proof}

\begin{setup}
\label{setup1}
Let $Q$ be the quiver on the left of Figure~\ref{fig:cornercase}(a) (respectively, the quiver on the left of Figure~\ref{fig:cornercase}(b))
for $n\geq 3$ (respectively, for $n=2$) and let $Q'$ be the quiver on the right in each case.
Let $H_Q$ be the group defined in Definition~\ref{defn_GQ_associated_group}, with generators $s_i$, and let $H_{Q'}$ be the group defined in Definition~\ref{defn_GQ_associated_group}, with generators $t_i$.
Let $S_1=t_0t_1t_0^{-1}$, $S_c=t_0t_ct_0^{-1}$,and $S_i=t_i$ for $i\not=n,c$.
We regard the subscripts of the $s_i$ and $S_i$ for $1\leq i\leq n$ to be taken modulo $n$ (with representatives $\{1,2\ldots ,n\}$), and the subscripts of the $t_i$ for $0\leq i\leq n$ to be taken modulo $n+1$ (with representatives $\{0,1,2,\ldots ,n\}$).
\end{setup} 

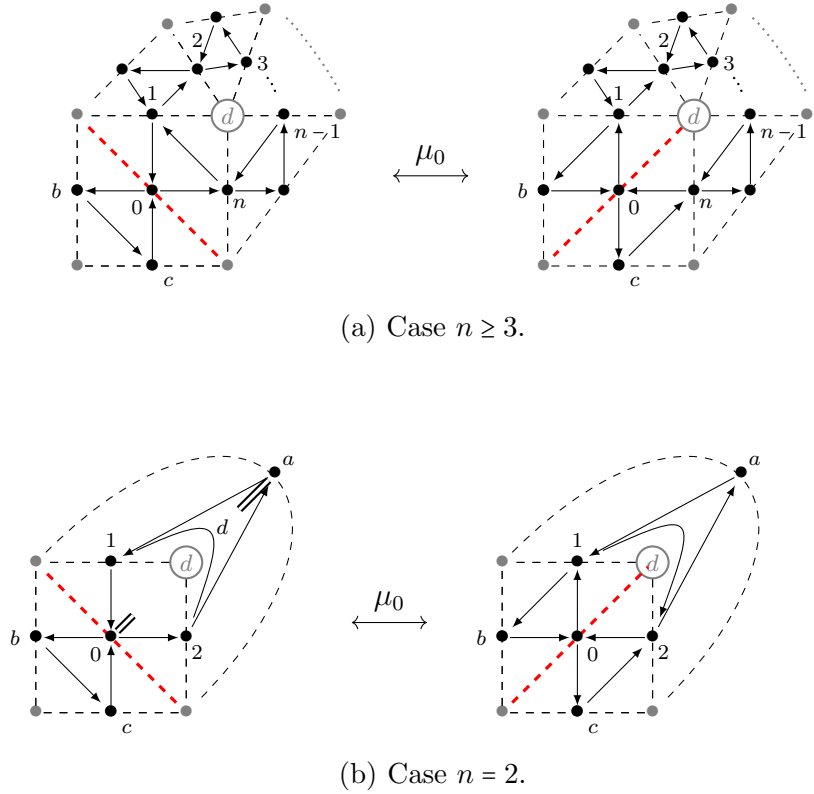
\begin{figure}[ht]
    \centering
\begin{subfigure}{\textwidth}
\renewcommand\captionlabelfont{}
\centering
\begin{tikzpicture}[scale=1,
  quiverarrow/.style={black,-latex,shorten <=-4pt, shorten >=-4pt},
  mutationarc/.style={dashed, red, very thick, shorten <=-4pt, shorten >=-4pt},
  mutationarccone/.style={dashed, red, very thick, shorten >=-4pt},
  arc/.style={dashed, black, shorten <=-4pt, shorten >=-4pt},
  arccone/.style={dashed, black, shorten >=-4pt},
  point/.style={gray},
  vertex/.style={black},
  conepoint/.style={gray, circle, draw=gray!100, fill=white!100, thick, inner sep=1.5pt},
  db/.style={thick, double, double distance=1.3pt, shorten <=-6pt}
  ]
\begin{scope}
 \node[point] (p1) at (1.2,3.2) {$\bullet$};
 \node[point] (p2) at (2.5,3.4) {$\bullet$};
 \node[point] (p3) at (3.5,2) {$\bullet$};
 \node[point] (p4) at (2,0) {$\bullet$};
 \node[point] (p5) at (0,0) {$\bullet$};
 \node[point] (p6) at (0,2) {$\bullet$};
 \node[conepoint] (p7) at (2,2) {\scriptsize $d$};

 \draw[arc] (p4) -- (p5);
 \draw[arc] (p5) -- (p6);
 \draw[arccone] (p7) -- (p4);
 \draw[arccone] (p7) -- (p6);
 \draw[arccone] (p7) -- (p2);
 \draw[arccone] (p7) -- (p1);
 \draw[arccone] (p7) -- (p3);

\draw[arccone] (p1) -- coordinate[midway](m12) (p2);
\draw[arccone] (p3) -- coordinate[midway](m34) (p4);
\draw[arccone] (p6) -- coordinate[midway](m61) (p1);

\node[vertex] (v12) at (m12) {$\bullet$};
\node[vertex] (v34) at (m34) {$\bullet$};
\node[vertex] (v61) at (m61) {$\bullet$};

 \draw[dotted, gray, thick, shorten <=1pt, shorten >=1pt] (p2) .. controls +(-20:0.4) and +(90:0.4) .. (p3);
 \draw[mutationarc] (p4) --(p6);

 \node[vertex] (w1) at  ($(p5)! 0.5!( p6)$) {$\bullet$};
 \node[vertex] (w2) at ($(p4)! 0.5!(p5)$) {$\bullet$};
 \node[vertex, label=
{[label distance=-10pt]335:{\tiny $n$}}] (vn) at ($(p4)! 0.5!(p7)$) {$\bullet$};
 \node[vertex, label={[label distance=-5pt]90:{\tiny $1$}}] (v1) at ($(p6)! 0.5!(p7)$) {$\bullet$};
 \node[vertex, label=
{[label distance=-10pt]225:{\tiny $0$}}] (v0) at ($(p4)! 0.5!(p6)$) {$\bullet$};
 \node[vertex, label=
{[label distance=-8pt]290:{\tiny $n-1$}}] (vn1) at ($(p3)! 0.5!(p7)$) {$\bullet$};
 \node[vertex, label=
{[label distance=-7pt]0:{\tiny $3$}}] (v3) at ($(p2)! 0.5!(p7)$) {$\bullet$};
 \node[vertex, label=
{[label distance=-2pt]90:{\tiny $2$}}] (v2) at ($(p1)! 0.5!(p7)$) {$\bullet$};
\node[vertex, label=
{[label distance=-7pt]-70:{\tiny $c$}}] (vc) at ($(p4)! 0.5!(p5)$) {$\bullet$};
\node[vertex, label=
{[label distance=-5pt]180:{\tiny $b$}}] (vb) at ($(p6)! 0.5!(p5)$) {$\bullet$};

 \draw[arc] (p4) -- (p5);
 \draw[arc] (p5) -- (p6);
 \draw[arccone] (p7) -- (p4);
 \draw[arccone] (p7) -- (p6);
 \draw[arccone] (p7) -- (p2);
 \draw[arccone] (p7) -- (p1);
 \draw[arccone] (p7) -- (p3);

 \draw[quiverarrow] (v0) -- (w1);
 \draw[quiverarrow] (v1) -- (v0);
 \draw[quiverarrow] (w1) -- (w2);
 \draw[quiverarrow] (v0) -- (vn);
 \draw[quiverarrow] (w2) -- (v0);
 \draw[quiverarrow] (vn) -- (v1);
 \draw[quiverarrow] (v2)--(v3);
 \draw[quiverarrow] (vn1)--(vn);
 \draw[quiverarrow] (v1) --(v2);

 \draw[quiverarrow] (v2) -- (v61);
 \draw[quiverarrow] (v61) -- (v1);
 \draw[quiverarrow] (v3) -- (v12);
 \draw[quiverarrow] (v12) -- (v2);
 \draw[quiverarrow] (v34) -- (vn1);
\draw[quiverarrow] (vn) -- (v34);

 \draw[dotted, thick, shorten <=1pt, shorten >=1pt] (v3) .. controls +(-35:0.4) and +(110:0.4) .. (vn1);
 \draw [<->] (4.2,1.2) to node[above] {$\mu_0$} (5.2,1.2);
\end{scope}
\begin{scope}[shift={(6.2,0)}]
 \node[point] (p1) at (1.2,3.2) {$\bullet$};
 \node[point] (p2) at (2.5,3.4) {$\bullet$};
 \node[point] (p3) at (3.5,2) {$\bullet$};
 \node[point] (p4) at (2,0) {$\bullet$};
 \node[point] (p5) at (0,0) {$\bullet$};
 \node[point] (p6) at (0,2) {$\bullet$};
 \node[conepoint] (p7) at (2,2) {\scriptsize $d$};
 
 \draw[arc] (p4) -- (p5);
 \draw[arc] (p5) -- (p6);
 \draw[arccone] (p7) -- (p4);
 \draw[arccone] (p7) -- (p6);
 \draw[arccone] (p7) -- (p2);
 \draw[arccone] (p7) -- (p1);
 \draw[arccone] (p7) -- (p3);

\draw[arccone] (p1) -- coordinate[midway](m12) (p2);
\draw[arccone] (p3) -- coordinate[midway](m34) (p4);
\draw[arccone] (p6) -- coordinate[midway](m61) (p1);

 \draw[dotted, gray, thick, shorten <=1pt, shorten >=1pt] (p2) .. controls +(-20:0.4) and +(90:0.4) .. (p3);
 \draw[mutationarccone] (p7)--(p5);

\node[vertex] (w1) at  ($(p5)! 0.5!( p6)$) {$\bullet$};
 \node[vertex] (w2) at ($(p4)! 0.5!(p5)$) {$\bullet$};
 \node[vertex, label=
{[label distance=-10pt]335:{\tiny $n$}}] (vn) at ($(p4)! 0.5!(p7)$) {$\bullet$};
 \node[vertex, label={[label distance=-5pt]90:{\tiny $1$}}] (v1) at ($(p6)! 0.5!(p7)$) {$\bullet$};
 \node[vertex, label=
{[label distance=-10pt]315:{\tiny $0$}}] (v0) at ($(p7)! 0.5!(p5)$) {$\bullet$};
 \node[vertex, label=
{[label distance=-8pt]290:{\tiny $n-1$}}] (vn1) at ($(p3)! 0.5!(p7)$) {$\bullet$};
 \node[vertex, label=
{[label distance=-7pt]0:{\tiny $3$}}] (v3) at ($(p2)! 0.5!(p7)$) {$\bullet$};
 \node[vertex, label=
{[label distance=-2pt]90:{\tiny $2$}}] (v2) at ($(p1)! 0.5!(p7)$) {$\bullet$};
\node[vertex, label=
{[label distance=-7pt]-70:{\tiny $c$}}] (vc) at ($(p4)! 0.5!(p5)$) {$\bullet$};
\node[vertex, label=
{[label distance=-5pt]180:{\tiny $b$}}] (vb) at ($(p6)! 0.5!(p5)$) {$\bullet$};

\node[vertex] (v12) at (m12) {$\bullet$};
\node[vertex] (v34) at (m34) {$\bullet$};
\node[vertex] (v61) at (m61) {$\bullet$};

 \draw[quiverarrow] (v0) -- (v1);
 \draw[quiverarrow] (v1) -- (w1);
 \draw[quiverarrow] (w1) -- (v0);
 \draw[quiverarrow] (v0) -- (w2);
 \draw[quiverarrow] (w2) -- (vn);
 \draw[quiverarrow] (vn) -- (v0);
 \draw[quiverarrow] (v1) -- (v2);
 \draw[quiverarrow] (v2) -- (v3);
 \draw[quiverarrow] (vn1) -- (vn);

 \draw[quiverarrow] (v2) -- (v61);
 \draw[quiverarrow] (v61) -- (v1);
 \draw[quiverarrow] (v3) -- (v12);
 \draw[quiverarrow] (v12) -- (v2);
 \draw[quiverarrow] (v34) -- (vn1);
\draw[quiverarrow] (vn) -- (v34);

 \draw[dotted, thick, shorten <=1pt, shorten >=1pt] (v3) .. controls +(-35:0.4) and +(110:0.4) .. (vn1);
 
\end{scope}
\end{tikzpicture}
\caption{Case $n\geq 3$.}
\label{fig:cornercasen3}
\end{subfigure}
\begin{subfigure}{\textwidth}
\centering
\renewcommand\captionlabelfont{}
\begin{tikzpicture}[scale=1,
  quiverarrow/.style={black,-latex,shorten <=-4pt, shorten >=-4pt},
  mutationarc/.style={dashed, red, very thick, shorten <=-4pt, shorten >=-4pt},
  mutationarccone/.style={dashed, red, very thick, shorten >=-4pt},
  arc/.style={dashed, black, shorten <=-4pt, shorten >=-4pt},
  arccone/.style={dashed, black, shorten >=-4pt},
  point/.style={gray},
  vertex/.style={black},
  conepoint/.style={gray, circle, draw=gray!100, fill=white!100, thick, inner sep=1.5pt},
  db/.style={thick, double, double distance=1.3pt, shorten <=-6pt}
  ]
\begin{scope}
 \node[point] (p4) at (2,0) {$\bullet$};
 \node[point] (p5) at (0,0) {$\bullet$};
 \node[point] (p6) at (0,2) {$\bullet$};
 \node[conepoint] (p7) at (2,2) {\scriptsize $d$};

 \draw[arc] (p4) -- (p5);
 \draw[arc] (p5) -- (p6);
 \draw[arccone] (p7) -- (p4);
 \draw[arccone] (p7) -- (p6);
 
 \draw[arc] (p6) .. controls +(45:4) and +(45:4) .. coordinate[midway](m46) (p4);

\node[vertex, label={[label distance=-10pt]35:{\tiny $a$}}] (v46) at (m46) {$\bullet$};

\draw[mutationarc] (p4) --(p6);

\node[vertex] (w1) at  ($(p5)! 0.5!( p6)$) {$\bullet$};
\node[vertex] (w2) at ($(p4)! 0.5!(p5)$) {$\bullet$};
\node[vertex, label=
{[label distance=-10pt]335:{\tiny $2$}}] (vn) at ($(p4)! 0.5!(p7)$) {$\bullet$};
\node[vertex, label={[label distance=-5pt]90:{\tiny $1$}}] (v1) at ($(p6)! 0.5!(p7)$) {$\bullet$};
\node[vertex, label=
{[label distance=-10pt]225:{\tiny $0$}}] (v0) at ($(p4)! 0.5!(p6)$) {$\bullet$};
\node[vertex, label=
{[label distance=-7pt]-70:{\tiny $c$}}] (vc) at ($(p4)! 0.5!(p5)$) {$\bullet$};
\node[vertex, label=
{[label distance=-5pt]180:{\tiny $b$}}] (vb) at ($(p6)! 0.5!(p5)$) {$\bullet$};

\draw[arc] (p4) -- (p5);
\draw[arc] (p5) -- (p6);
\draw[arccone] (p7) -- (p4);
\draw[arccone] (p7) -- (p6);

\draw[quiverarrow] (v0) -- (w1);
\draw[quiverarrow] (v1) -- (v0);
\draw[quiverarrow] (w1) -- (w2);
\draw[quiverarrow] (w2) -- (v0);
\draw[quiverarrow] (v0) -- (vn);

\draw[db,shorten >=22pt] (v0) -- (p7);
\draw[db,shorten >=22pt] (v46) -- (p7);

 \draw[-, shorten <=3pt, shorten >=1pt] (v1) .. controls +(25:1.65) and +(70:1.65) .. (vn);

\node at (2.48,2.48) {$\scriptstyle d$};

\draw[quiverarrow, shorten <=-3pt, shorten >=5pt] (vn) -- (m46);
\draw[quiverarrow, shorten <=3pt, shorten >=-3pt] (m46) -- (v1);
\end{scope}
 \draw [<->] (4.2,1.2) to node[above] {$\mu_0$} (5.2,1.2);
\begin{scope}[shift={(6.2,0)}]
 \node[point] (p4) at (2,0) {$\bullet$};
 \node[point] (p5) at (0,0) {$\bullet$};
 \node[point] (p6) at (0,2) {$\bullet$};
 \node[conepoint] (p7) at (2,2) {\scriptsize $d$};

 \draw[arc] (p4) -- (p5);
 \draw[arc] (p5) -- (p6);
 \draw[arccone] (p7) -- (p4);
 \draw[arccone] (p7) -- (p6);
 
 \draw[arc] (p6) .. controls +(45:4) and +(45:4) .. coordinate[midway](m46) (p4);
 
\node[vertex, label={[label distance=-10pt]35:{\tiny $a$}}] (v46) at (m46) {$\bullet$};

\draw[mutationarc] (p5) --(p7);

\node[vertex] (w1) at  ($(p5)! 0.5!( p6)$) {$\bullet$};
\node[vertex] (w2) at ($(p4)! 0.5!(p5)$) {$\bullet$};
\node[vertex, label=
{[label distance=-10pt]335:{\tiny $2$}}] (vn) at ($(p4)! 0.5!(p7)$) {$\bullet$};
\node[vertex, label={[label distance=-5pt]90:{\tiny $1$}}] (v1) at ($(p6)! 0.5!(p7)$) {$\bullet$};
\node[vertex, label=
{[label distance=-10pt]315:{\tiny $0$}}] (v0) at ($(p4)! 0.5!(p6)$) {$\bullet$};
\node[vertex, label=
{[label distance=-7pt]-70:{\tiny $c$}}] (vc) at ($(p4)! 0.5!(p5)$) {$\bullet$};
\node[vertex, label=
{[label distance=-5pt]180:{\tiny $b$}}] (vb) at ($(p6)! 0.5!(p5)$) {$\bullet$};

\draw[arc] (p4) -- (p5);
\draw[arc] (p5) -- (p6);
\draw[arccone] (p7) -- (p4);
\draw[arccone] (p7) -- (p6);

\draw[quiverarrow] (v0) -- (v1);
\draw[quiverarrow] (v1) -- (w1);
\draw[quiverarrow] (w1) -- (v0);
\draw[quiverarrow] (v0) -- (w2);
\draw[quiverarrow] (w2) -- (vn);
\draw[quiverarrow] (vn) -- (v0);

\draw[quiverarrow, shorten <=-3pt, shorten >=5pt] (vn) -- (m46);
\draw[quiverarrow, shorten <=3pt, shorten >=-3pt] (m46) -- (v1);s
 \draw[quiverarrow, shorten <=3pt, shorten >=1pt] (v1) .. controls +(25:1.75) and +(70:1.75) .. (vn);

\end{scope}
\end{tikzpicture}
\caption{Case $n=2$.}
\label{fig:cornercasen2}
\end{subfigure}
\caption{A mutation involving a cycle around the conepoint. In case (b), one of the vertices is labelled $a$ to avoid confusion in arguments involving indices taken modulo $n$.}
\label{fig:cornercase}
\end{figure}

\begin{lemma}
    \label{lem:braidS}
Let $n\geq 2$ be an integer.
In Setup~\ref{setup1}, the elements $S_i$ satisfy the defining relations of $H_Q$.
\end{lemma}
\begin{proof}
    This follows from Proposition~\ref{prop:GMcase}.
\end{proof}

\begin{lemma}
\label{lem:othercasehelp}
Let $n\geq 2$ be an integer.
Suppose we are in Setup~\ref{setup1}, and that $0\leq i,r\leq n$ and $r\not=i-1,i\mod n+1$.
Then we have the following:
\begin{itemize}
    \item[(a)]
    $$t_{r-n+1}t_{r-n+2}\cdots t_{r}t_i^{-1}=t_{i+1}^{-1}t_{r-n+1}t_{r-n+2}\cdots t_r.$$
    \item[(b)]
$$t_{r-dn+1}\cdots t_{r-1}t_rt_i^{-1}=
t_{d+i}^{-1}t_{r-dn+1}\cdots t_{r-1}t_r.$$
\end{itemize}
\end{lemma}
\begin{proof}
    For (a), we have
\begin{align*}
t_{r-n+1}t_{r-n+2}\cdots t_rt_i^{-1} &=
t_{r-n+1}\cdots t_it_{i+1}\cdots t_rt_i^{-1} \\
&= t_{r-n+1}\cdots t_it_{i+1}t_i^{-1}t_{i+2}\cdots t_r \\
&= t_{r-n+1}\cdots t_{i+1}^{-1}t_it_{i+1}t_{i+2}\cdots t_r \\
&= t_{i+1}^{-1}t_{r-n+1}t_{r-n+2}\cdots t_r.
\end{align*}
For (b), we argue by induction on $d$.
If $d=1$, the result follows from part (a).
Suppose the result holds for $d$. Then
\begin{align*}
t_{r-(d+1)n+1}\cdots t_{r-1}t_rt_i^{-1} &=
t_{r-(d+1)n+1}\cdots t_{r-dn}t_{r-dn+1}\cdots t_rt_i^{-1} \\
&=
t_{r-(d+1)n+1}\cdots t_{r-dn}t_{i+d}^{-1}t_{r-dn+1}\cdots t_r \\
&=
t_{i+d+1}^{-1}
t_{r-(d+1)n+1}\cdots t_{r-1}t_r,
\end{align*}
using the induction hypothesis and then noting
that
$r-dn\not\equiv i+d,i+d-1\mod n+1,$ since $r-dn\equiv r+d\mod n+1$.
\end{proof}

\begin{lemma}
    \label{lem:cyclehelp}
    Let $n\geq 2$ and $d\geq 1$ be integers.
    Then, in Setup~\ref{setup1}, for any $1\leq r\leq n$, we have:
    $$t_dt_{d-1}\cdots t_1\underbrace{S_{r-d(n-1)+1}\cdots S_{r-1}S_r}_{d(n-1)\text{ terms}}=\underbrace{t_{r-dn+1}\cdots t_{r-1}t_r}_{dn\text{ terms}}.$$
\end{lemma}
\begin{proof}
    We prove the result by induction on $d$.
For $d=1$, we have (for the case $r=n$),
$$t_1S_2\cdots S_n=t_1t_2\cdots t_n,$$
as required.
For $r\not=n$, we have
\begin{align*}
    t_1S_{r-(n-1)+1}\cdots S_r &=
    t_1S_{r+2}\cdots S_nS_1\cdots S_r \\
    &= t_1t_{r+2}\cdots t_nt_0t_1t_0^{-1}t_2\cdots t_r \\
    &= t_1t_{r+2}\cdots t_nt_1^{-1}t_0t_1t_2\cdots t_r \\
    &= t_1t_1^{-1} t_{r+2}\cdots t_nt_0t_1\cdots t_r \\
    &= t_{r+2}\cdots t_nt_0t_1\cdots t_r \\
    &= t_{r-(n+1)+2}\cdots t_{r-1}t_r \\
    &= t_{r-n+1}\cdots t_{r-1}t_r, \\
\end{align*}
as required, giving the result for $d=1$.
Assume the result holds for an integer $d\geq 1$.
Then, using the induction hypothesis, we have
(for $r=n$):
\begin{align*}
t_{d+1}\cdots t_1
\underbrace{S_{n-(d+1)(n-1)+1}\cdots S_n}_{(d+1)(n-1) \text{terms}}
&=
t_{d+1}(t_d\cdots t_1)(S_{n-(d+1)(n-1)+1}\cdots S_nS_1)S_2\cdots S_n \\
&=
t_{d+1}(t_d\cdots t_1)(S_{1-d(n-1)+1}\cdots S_nS_1)S_2\cdots S_n \\
&= t_{d+1}(t_{1-dn+1}\cdots t_1)t_2\cdots t_n, \\
\end{align*}
as required, noting that $1-dn+1\equiv d+2\mod n+1$.
For $r\not=n$, we have,
\begin{align*}
    t_{d+1}\cdots t_1\underbrace{S_{r-(d+1)(n-1)+1}\cdots S_r}_{(d+1)(n-1) \text{terms}} &=
        t_{d+1}(t_d\cdots t_1)
    \underbrace{(S_{r-(d+1)(n-1)+1}\cdots S_{r-(n-1)}}_{d(n-1) \text{terms}} \underbrace{S_{r-(n-2)}\cdots S_r}_{n-1\text{ terms}} \\
    &= t_{d+1}(t_d\cdots t_1)
    \underbrace{(S_{r-(d+1)(n-1)+1}\cdots S_{r-(n-1)}}_{d(n-1) \text{terms}} \underbrace{S_{r+2}\cdots S_nS_1\cdots S_r}_{n-1\text{ terms}} \\
    &=
    t_{d+1}(t_d\cdots t_1)\underbrace{(S_{r+1-d(n-1)+1}\cdots S_{r+1})}_{d(n-1) \text{terms}} \underbrace{S_{r+2}\cdots S_nS_1\cdots S_r}_{n-1\text{ terms}} \\
    &= t_{d+1}\underbrace{t_{r+1-dn+1}\cdots t_{r+1}}_{dn \text{ terms} }t_{r+2}\cdots t_nt_0t_1t_0^{-1}t_2\cdots t_r \\
    &=t_{d+1}\underbrace{t_{r+1-dn+1}\cdots t_{r+1}}_{dn \text{ terms} }t_{r+2}\cdots t_nt_1^{-1}t_0t_1t_2\cdots t_r \\
    &=t_{d+1}\underbrace{t_{r+1-dn+1}\cdots t_{r+1}}_{dn \text{ terms} }t_1^{-1}t_{r+2}\cdots t_nt_0t_1t_2\cdots t_r, \\    
\end{align*}
using the induction hypothesis.
Note that the subscripts of the $T$s are reduced mod $n$ first before being applied to the $T$s (and then reduced mod $n+1$!).
Since $1\leq r\leq n-1$, $r+1\not\equiv 0,1\mod n+1$, so by
Lemma~\ref{lem:othercasehelp},
$$\underbrace{t_{r+1-dn+1}\cdots t_{r+1}}_{dn \text{ terms} }t_1^{-1}=
t_{d+1}^{-1}\underbrace{t_{r+1-dn+1}\cdots t_{r+1}}_{dn \text{ terms} }
$$
It follows that
\begin{align*}
    t_{d+1}\cdots t_1\underbrace{S_{r-(d+1)(n+1)+1}\cdots S_r}_{(d+1)(n-1) \text{terms}} &=
    t_{d+1}t_{d+1}^{-1}\underbrace{t_{r+1-dn+1}\cdots t_{r+1}}_{dn \text{ terms} }t_{r+2}\cdots t_nt_0t_1t_2\cdots t_r, \\ 
    &=
    \underbrace{t_{r+1-dn+1}\cdots t_{r+1}}_{dn \text{ terms} }t_{r+2}\cdots t_nt_0t_1t_2\cdots t_r,
\end{align*}
with $(d+1)n$ terms in total, as required, noting that $r+1-dn+1\equiv r-(d+1)n+1\mod n+1$.
\end{proof}

We have the following generalization of~\cite[Lem.\ 2.4]{GM}.

\begin{lemma}
\label{lem:oneall}
Let $n\geq 3$ be an integer. Let $g_0,g_1,\ldots ,g_{n-1}$ be elements of a group $G$, with subscripts taken modulo $n$, satisfying the relations:
\begin{align*}
g_ig_{i+1}g_i=g_{i+1}g_ig_{i+1}, \quad 0\leq i\leq n-1 \\
g_ig_j=g_jg_i, 0\leq i,j\leq n-1,j\not=i,i+1\mod n.
\end{align*}
Then 
\begin{itemize}
    \item[(a)] $g_{i+1}^{-1}g_i\cdots g_{i+n-2}=g_i\cdots g_{i+n-2}g_i^{-1}$, for $1\leq i\leq n$.
    \item[(b)] $g_1^{-1}(g_0g_1\cdots g_{(n-1)d-1})=g_0g_1\cdots g_{(n-1)d-1}g_{(n-1)d+1}^{-1}$
    \item[(c)] If the relation
\begin{equation}
    g_rg_{r+1}\cdots g_{(n-1)d+r-1} = g_{r+1}g_{r+2}\cdots g_{((n-1)d+r}
    \label{eq:cycler}
\end{equation}
holds for some $r\in \{0,1,\ldots ,n-1\}$, then it holds for all $r$.
\end{itemize}
\end{lemma}
\begin{proof}
For (a), we have:
\begin{align*}
    g_{i+1}^{-1}g_ig_{i+1}g_{i+2}\cdots g_{i+n-2} &=
    g_ig_{i+1}g_i^{-1}g_{i+2}\cdots g_{i+n-2} \\
    &= g_ig_{i+1}\cdots g_{i+n-2}g_i^{-1},
\end{align*}
as required, using commutations in the second step.
For part (b), we have:
\begin{align*}
    g_1^{-1}(g_0g_1\cdots g_{(n-1)d-1}) &=
    g_1^{-1}(g_0\cdots g_{n-2})(g_{n-1}\cdots g_{2n-3})\cdots g_{(d-1)(n-1)}\cdots g_{(n-1)d-1}) \\
    &= (g_0\cdots g_{n-2})(g_{n-1}\cdots g_{2n-3})\cdots (g_{(d-1)(n-1)\cdots g_{(n-1)d-1})})g_{1-d}, 
\end{align*}
using part (a) $d$ times. Finally, note that
$1-d\equiv (n-1)d+1\mod n$.
For part (c), assume first that equation~\eqref{eq:cycler} holds for $r=0$, so that
$$g_0g_1\cdots g_{(n-1)d+r-1} = g_1g_2\cdots g_{((n-1)d}.$$
Multiplying this on the left by $g_1^{-1}$ and on the right by $g_{(n-1)d+1}$ gives
$$g_0g_1\cdots g_{(n-1)d-1} =
g_2g_3\cdots g_{(n-1)d+1},$$
by part (b). Repeated application of this argument gives the desired result.
\end{proof}

\begin{lemma}
\label{lem:cornercasen3}
Let $n\geq 3$ be an integer. Let $Q$ be the quiver on the left of Figure~\ref{fig:cornercase}(a) and $Q'$ the quiver on the right. Let $G_Q$ (respectively $G_{Q'}$) be the group with generators $s_{i}$ (respectively $t_i$), with $i$ ranging through the vertices of the quivers, satisfying the relations associated with $Q$ (respectively $Q'$). Then, there is a group homomorphism
\begin{align*}
    \varphi^Q_0: G_Q\rightarrow G_{Q'} \text{ given by } \varphi^Q_0(s_1)=S_1=t_0t_1t_0^{-1},\,  \varphi^Q_0(s_c)=S_c=t_0t_ct_0^{-1},\,\varphi^Q_0(s_i)=S_i=t_i \text{ for } i\neq 1,c.
\end{align*}
\end{lemma}
\begin{proof}
The defining relations for $G_Q$, apart from those corresponding to the $n$-cycle $1\rightarrow 2\rightarrow \cdots \rightarrow n\rightarrow 1$, hold by Lemma~\ref{lem:braidS}.

The defining relations for $G_{Q'}$ corresponding to the $n+1$-cycle $0\rightarrow 1\rightarrow \cdots \rightarrow n+1\rightarrow 0$ can be written in the form:
$$\underbrace{t_{r-nd+1}\cdots
t_{r-1}t_r}_{\text{$dn$ terms}}=
\underbrace{t_{r+1-nd+1}\cdots t_rt_{r+1}}_{\text{$dn$ terms}},$$
for $0\leq r\leq n$.

The defining relations for $G_Q$ corresponding to the $n$-cycle
    $1\rightarrow 2\rightarrow \cdots \rightarrow n\rightarrow 1$ can be written in the form:
$$\underbrace{(S_{r-d(n-1)+1} \cdots
S_{r-1}S_r) }_{\text{$d(n-1)$ terms}}=
\underbrace{(S_{r+1-d(n-1)+1} \cdots S_rS_{r+1})}_{\text{$d(n-1)$ terms}},$$
for $1\leq r\leq n$.

The cycle relations for $G_Q$ for $1\leq r\leq n-1$ now follow from Lemma~\ref{lem:cyclehelp}.
The cycle relation for $r=n$ follows from Lemma~\ref{lem:oneall}(c), taking $g_0=S_1$, $g_2=S_2, \ldots,$ $g_{n-1}=S_n$. This gives the required result.
\end{proof}

\begin{lemma}
\label{lem:cornercasen2}
Let $n=2$.
Let $Q$ be the quiver on the left of Figure~\ref{fig:cornercase}(b) and $Q'$ the quiver on the right. Let $G_Q$ (respectively $G_{Q'}$) be the group with generators $s_{i}$ (respectively $t_i$), with $i$ ranging through the vertices of the quivers, satisfying the relations associated with $Q$ (respectively $Q'$). Then, there is a group homomorphism
\begin{align*}
    \varphi^Q_0: G_Q\rightarrow G_{Q'} \text{ given by } \varphi^Q_0(s_1)=S_1=t_0t_1t_0^{-1}, \, 
    \varphi^Q_0(s_c)=S_c=t_0t_ct_0^{-1}, \,
    \varphi^Q_0(s_i)=S_i=t_i \text{ for } i\neq 1,c.
\end{align*}
\end{lemma}
\begin{proof}
The defining relations for $G_Q$, apart from those corresponding to the $3$-cycle $0\rightarrow 2\rightarrow 1\rightarrow 0$, the $4$-cycle $0\rightarrow 2\rightarrow a\rightarrow 1\rightarrow 0$ and the double edges in $Q$, all hold by Lemma~\ref{lem:braidS}.

We have $t_1t_2t_1=t_2t_1t_2$, which gives:
$$S_0^{-1}S_1S_0S_2S_0^{-1}S_1S_0=S_2S_0^{-1}S_1S_0S_2.$$
Multiplying on the left by $S_0$ gives
$$S_1S_0S_2\underline{S_0^{-1}S_1S_0}=\underline{S_0S_2S_0^{-1}}S_1S_0S_2.$$
Applying the braid relations corresponding to the arrows $1\rightarrow 0$ and $0\rightarrow 2$ in $Q$ we obtain:
$$S_1S_0S_2S_1S_0S_1^{-1}=S_2^{-1}S_0S_2S_1S_0S_2,$$
and hence, multiplying on the left by $S_2$ and on the right by $S_1$:
$$S_2S_1S_0S_2S_1S_0=S_0S_2S_1S_0S_2S_1,$$
which is the relation associated to the double edge incident with $0$ in $Q$.

The relation associated to the double edge incident with the vertex $a$ in $Q$ is:
\begin{equation}
    \label{eq:3double}
    S_aS_1S_2S_aS_1S_2=S_1S_2S_aS_1S_2S_a.
\end{equation}
We apply transformations to~\eqref{eq:3double} to give a series of equivalent versions.
Firstly,~\eqref{eq:3double} is equivalent to
$$\underline{t_at_0}t_1t_0^{-1}t_2t_a\underline{t_0t_1t_0^{-1}}t_2=
t_0t_1t_0^{-1}t_2t_a\underline{t_0t_1t_0^{-1}}t_2t_a$$
Applying the commutation $t_at_0=t_0t_a$ and the braid relation corresponding to the arrow $0\rightarrow 1$ in $Q'$ on both sides gives
$$t_0t_at_1t_0^{-1}t_2t_at_1^{-1}t_0t_1t_2=
t_0t_1t_0^{-1}t_2t_at_1^{-1}t_0t_1t_2t_a.$$
We multiply on the left by $t_1^{-1}t_0$ to obtain:
$$\underline{t_1^{-1}t_at_1}t_0^{-1}t_2t_at_1^{-1}t_0t_1t_2=
t_0^{-1}t_2t_at_1^{-1}t_0\underline{t_1t_2t_a}.$$
Since $t_2t_at_1t_2=t_at_1t_2t_a$, we may substitute in
    $$t_1t_2t_a=t_a^{-1}t_2t_at_1t_2$$
on the right-hand side, while on the left-hand side we apply the braid relation for the arrow $3\rightarrow 1$ in $Q'$ to get:
$$t_at_1t_a^{-1}t_0^{-1}t_2t_at_1^{-1}t_0t_1t_2=
t_0^{-1}t_2t_at_1^{-1}t_0t_a^{-1}t_2t_at_1t_2.$$
Multiplying on the right by $t_2^{-1}t_1^{-1}$ and on the left by $t_0$ gives:
$$\underline{t_0t_a}t_1\underline{t_a^{-1}t_0^{-1}}t_2t_at_1^{-1}t_0=
t_2t_at_1^{-1}t_0t_a^{-1}t_2t_a.$$
Applying the commutation $t_0t_a=t_at_0$ twice on the left hand side, and multiplying on the right by $t_2$ gives:
$$t_at_0t_1t_0^{-1}t_a^{-1}t_2t_at_1^{-1}t_0=t_2t_at_1^{-1}t_0\underline{t_a^{-1}t_2t_a}.$$
We apply the braid relation corresponding to the arrow $2\rightarrow 3$ in $Q'$ to obtain:
$$t_at_0t_1t_0^{-1}t_a^{-1}t_2t_at_1^{-1}t_0=t_2t_at_1^{-1}t_0t_2t_at_2^{-1},$$
and, multiplying on the right by $t_2$ we obtain:
$$t_a\underline{t_0t_1t_0^{-1}}\underline{t_a^{-1}t_2t_a}t_1^{-1}t_0t_2=t_2t_at_1^{-1}t_0t_2t_a.$$
Applying the braid relations corresponding to the arrows $0\rightarrow 1$ and $2\rightarrow a$ in $Q'$ gives:
$$t_at_1^{-1}t_0\underline{t_1t_2t_at_2^{-1}t_1^{-1}}t_0t_2=t_2t_at_1^{-1}t_0t_2t_a.$$
Since $t_2t_at_1t_2=t_at_1t_2t_a$, we may substitute in
    $$t_1t_2t_at_2^{-1}t_1^{-1}=t_a^{-1}t_2t_a$$
    on the left-hand side to obtain:
$$t_at_1^{-1}t_0t_a^{-1}t_2t_at_0t_2=t_2t_at_1^{-1}t_0t_2t_a.$$
Applying the commutation $t_0t_a=t_at_0$ on the left hand side twice, we obtain:
$$\underline{t_at_1^{-1}t_a^{-1}}t_0t_2t_0t_at_2=t_2t_at_1^{-1}t_0t_2t_a.$$
Applying the braid relation corresponding to the arrow $a\rightarrow 1$ in $Q'$ gives:
$$t_1^{-1}t_a^{-1}t_1t_0t_2t_0t_at_2=t_2t_at_1^{-1}t_0t_2t_a.$$
Multiplying on the left by $t_at_1$ gives:
$$t_1t_0t_2t_0t_at_2=\underline{t_at_1t_2t_at_1^{-1}}t_0t_2t_a.$$
Since $t_at_1t_2t_a=t_1t_2t_at_1$, we may substitute in
$t_at_1t_2t_at_1^{-1}=t_1t_2t_a$ on the right to get:
$$t_1t_0t_2t_0t_at_2=t_1t_2t_at_0t_2t_a.$$
Applying the braid relation corresponding to the arrow $2\rightarrow 0$ in $Q'$ on the left-hand side, we obtain:
$$t_1t_2t_0t_2t_at_2=t_1t_2t_at_0t_2t_a.$$
Multiplying both sides on the left by $t_2^{-1}t_1^{-1}$ gives
$$t_0t_2t_at_2=t_at_0t_2t_a.$$
Since this final equivalent version holds, we see that~\eqref{eq:3double} holds, as desired.

Next we check the two relations associated to the $4$-cycle $0\rightarrow 2\rightarrow a\rightarrow 1\rightarrow 0$ in $Q$.
Since $t_at_1t_2t_a=t_1t_2t_at_1$, we have:
$$S_aS_0^{-1}S_1S_0S_2S_a=S_0^{-1}S_1S_0S_2S_aS_0^{-1}S_1S_0.$$
Multiplying on the left by $S_0$ and applying the commutation $S_0S_a=S_aS_0$ gives:
$$S_aS_1S_0S_2S_a=S_1S_0S_2S_a\underline{S_0^{-1}S_1S_0}.$$
Applying the braid relation associated to the arrow $1\rightarrow 0$ in $Q$ gives:
$$S_aS_1S_0S_2S_a=S_1S_0S_2S_aS_1S_0S_1^{-1}.$$
Then, multiplying on the right by $S_1$ gives:
$$S_aS_1S_0S_2S_aS_1=S_1S_0S_2S_aS_1S_0,$$
which is one of the relations associated to the $4$-cycle in $Q$.

Since $t_2t_at_1t_2=t_at_1t_2t_a$ we have, multiplying on the left by $t_2t_0$, that:
$$\underline{t_2t_0t_2}t_at_1t_2=t_2t_0t_at_1t_2t_a.$$
Applying the braid relation corresponding to the arrow $2\rightarrow 0$ in $Q'$ gives:
$$t_0t_2t_0t_at_1t_2=t_2\underline{t_0t_a}t_1t_2t_a.$$
Applying the commutation $t_0t_a=t_at_0$ on the right hand side gives:
$$t_0t_2t_at_0t_1t_2=t_2t_at_0t_1t_2t_a.$$
We insert the product $t_0^{-1}t_0$ in two places to obtain:
$$t_0t_2t_at_0t_1t_0^{-1}t_0t_2=t_2t_at_0t_1t_0^{-1}t_0t_2t_a,$$
which can be rewritten as:
$$S_0S_2S_aS_1S_0S_2=S_2S_aS_1S_0S_2S_a,$$
which is the other relation associated to the $4$-cycle in $Q$.

One of the relations corresponding to the $3$-cycle $1\rightarrow 2\rightarrow 0\rightarrow 1$ in $Q'$ is
$$t_1t_2\cdots t_{2d}=t_2t_3\cdots t_{2d+1}.$$
By Lemma~\ref{lem:cyclehelp}, taking $r=1,2$, we have
$$t_dt_{d-1}\cdots t_1S_{1-d+1}\cdots S_0S_1=
t_{1-2d+1}\cdots t_0t_1$$
and
$$t_dt_{d-1}\cdots t_1S_{2-d+1}\cdots S_1S_2=
t_{2-2d+1}\cdots t_1t_2.$$
so
$$S_{1-d+1}\cdots S_0S_1=S_{2-d+1}\cdots S_1S_2,$$
which can be rewritten as
$$S_1S_2\cdots S_d=S_2S_3\cdots S_{d+1}$$
(switching the two sides of the equality if $d$ is even).
This is the remaining required defining relation of $G_Q$. Hence all the defining relations of $G_Q$ hold, and the result is shown.
\end{proof}

We consider the following setup:
\begin{setup}
\label{setup2new}
Let $Q$ be the quiver on the right of Figure~\ref{fig:cornercase}(a) (respectively, the quiver on the right of Figure~\ref{fig:cornercase}(b))
for $n\geq 3$ (respectively, for $n=2$) and let $Q$ be the quiver on the left in each case.
Let $H_Q$ be the group defined in Definition~\ref{defn_GQ_associated_group}, with generators $s_i$, and let $H_{Q'}$ be the group defined in Definition~\ref{defn_GQ_associated_group}, with generators $t_i$.
Let $S_n=t_0t_nt_0^{-1}$, $S_b=t_0t_bt_0^{-1}$, and $S_i=t_i$ for $i\neq n,b$.
We regard the subscripts of the $s_i$ and $S_i$ for $0\leq i\leq n$ to be taken modulo $n+1$ (with representatives $\{0,1,2\ldots ,n\}$), and the subscripts of the $t_i$ for $1\leq i\leq n$ to be taken modulo $n$ (with representatives $\{1,2,\ldots ,n\}$).
\end{setup}

We note the following:

\begin{lemma} 
    \label{lem:braidS2}
Let $n\geq 2$ be an integer.
In Setup~\ref{setup2new}, the elements $S_i$ satisfy the defining relations of $H_Q$.
\end{lemma}
\begin{proof}
    For $n\geq 3$, this follows from Proposition~\ref{prop:GMcase}.

For $n=2$, the braid and commutation relations and the relations for the $3$-cycles all hold by Proposition~\ref{prop:GMcase}, except for the relation for arrow from $1$ to $2$ and the relations for the $3$-cycle $2\rightarrow a\rightarrow 1\rightarrow 2$.

The relation for the double edge incident with $0$ in $Q'$ is:
$$t_0t_2t_1t_0t_2t_1=t_2t_1t_0t_2t_1t_0.$$
Substituting, this gives
$$S_0S_0^{-1}S_2S_0S_1S_0S_0^{-1}S_2S_0S_1=
S_0^{-1}S_2S_0S_1S_0S_0^{-1}S_2S_0S_1S_0,$$
which simplifies to
$$S_0S_2S_0S_1S_2S_0S_1=S_2S_0S_1S_2S_0S_1S_0$$
Using the braid relations for the arrows $0\rightarrow 2$ and $1\rightarrow 0$ in $Q$:
$$\underline{S_0S_2S_0}S_1S_2S_0S_1=S_2S_0S_1S_2\underline{S_0S_1S_0}$$
gives:
$$S_2S_0S_2S_1S_2S_0S_1=S_2S_0S_1S_2S_1S_0S_1,$$
which, after cancelling elements on the left and right, gives
$$S_2S_1S_2=S_1S_2S_1.$$

There are two relations for the $4$-cycle on vertices $2$, $a$, $1$ and $0$ in $Q'$. One of these is:
$$t_0t_2t_at_1t_0t_2=t_2t_at_1t_0t_2t_a,$$
which gives
$$S_0S_0^{-1}S_2S_0S_aS_1S_0S_0^{-1}S_2S_0=
S_0^{-1}S_2S_0S_aS_1S_0S_0^{-1}S_2S_0S_a,$$
and hence
$$\underline{S_0S_2S_0}S_aS_1S_2S_0=S_2S_0S_aS_1S_2\underline{S_0S_a}.$$
Applying the braid relation for the arrow $2\rightarrow 0$ in $Q$ and the commutation $S_0S_a=S_aS_0$ gives
$$S_2S_0S_2S_aS_1S_2S_0=S_2S_0S_aS_1S_2S_aS_0.$$
Applying cancellations on the left and right gives
$$S_2S_aS_1S_2=S_aS_1S_2S_a.$$
By~\cite[Lem.\ 2.4]{GM}, we conclude that
$$S_2S_aS_1S_2=S_aS_1S_2S_a=S_1S_2S_aS_1,$$
giving the result for $n=2$.
\end{proof}

\begin{lemma}
\label{lem:conjugatethrough}
In Setup~\ref{setup2new}, we have:
\begin{itemize}
    \item[(a)] $$S_i^{-1}S_{i+2}\cdots S_{i+n+1}=S_{i+2}\cdots S_{i+n+1}S_{i-1}^{-1};$$
\item[(b)] $$S_n^{-1}(S_1S_2\cdots S_{nd})=(S_1\cdots S_{nd})S_{n-d}^{-1}.$$
\end{itemize}
\end{lemma}
\begin{proof}
We use Lemma~\ref{lem:braidS2} throughout.
For (a), we have:
\begin{align*}
    S_i^{-1}S_{i+2}S_{i+3}\cdots S_{i+n+1}
    &=
    S_{i+2}S_{i+3}\cdots S_{i+n-1}S_i^{-1}S_{i+n}S_{i+n+1} \\
    &=
    S_{i+2}S_{i+3}\cdots S_{i+n-1}S_{i+n}S_{i+n+1}S_{i+n}^{-1},
\end{align*}
as required, using commutations in the first step and
noting that $i+n+1\equiv i$ and $i+n\equiv i-1 \mod n+1$.
For part (b), note that
$$S_1\cdots S_{nd}=(S_1\cdots S_n)(S_{n+1}\cdots S_{2n})\cdots (S_d\cdots S_{d+n-1}),$$
and use part (a) $d$ times.
\end{proof}

\begin{lemma}
\label{lem:cyclehelp2}
In Setup~\ref{setup2new}, let $d\geq 1$ be an integer.
Then, for any $1\leq r\leq n$, we have:
$$
        \underbrace{(t_rt_{r+1} \cdots t_{d(n-1)+r-1}) }_{\text{$d(n-1)$ terms}}\underbrace{(S_nS_{n-1}\cdots S_{n-d+1})}_{\text{$d$ terms}} =
        \underbrace{S_rS_{r+1}\cdots S_{r+nd-1}}_{\text{$dn$ terms}}
$$
\end{lemma}
\begin{proof}
We prove the result by induction on $d$,
using Lemma~\ref{lem:braidS2} throughout.
For $d=1$, we have $t_1t_2\cdots t_{n-1}S_n=S_1S_2\cdots S_n$. For $r\geq 2$, we have (recalling that subscripts of the $t_i$ are written modulo $n$):
\begin{align*}
    t_rt_{r+1}\cdots t_{r+n-2}S_n
    &=
    t_r\cdots t_{n-1}t_nt_1t_2\cdots t_{r+n-2-n}S_n
    \\
    &=
    (S_r\cdots S_{n-1})(S_0^{-1}S_nS_0)(S_1S_2\cdots S_{r-2})S_n \\
    &= (S_r\cdots S_{n-1})(S_nS_0S_n^{-1})(S_1S_2\cdots S_{r-2})S_n \\
    &= (S_r\cdots S_{n-1})S_nS_0(S_1\cdots S_{r-2})S_n^{-1}S_n,
\end{align*}
and we see that the result holds for $d=1$.

Assume the result holds for an integer $d\geq 1$. Then, using the induction hypothesis,
\begin{align*}
    \underbrace{(t_1t_2\cdots t_{(d+1)(n-1)})}_{\text{$(d+1)(n-1)$ terms}}(\underbrace{S_nS_{n-1}\cdots S_{n-(d+1)-1}}_{\text{$d+1$ terms}}) &=
    (t_1t_2\ldots t_{n-1})
    \underbrace{(t_nt_{n+1}\cdots t_{(d+1)(n-1)})}_{\text{$d(n-1)$ terms}}(\underbrace{(S_nS_{n-1}\cdots S_{n-d+1}}_{\text{$d$ terms}})
    S_{n-d} \\
    &=
    (t_1t_2\cdots t_{n-1}) \underbrace{(S_nS_{n+1}\cdots S_{n+nd-1})}_{\text{$nd$ terms}} S_{n-d} \\
    &=
    S_1S_2\cdots S_{n+nd-1}S_{n+dn},
\end{align*}
noting that $n-d\equiv n+dn\mod n+1$.
We also have, for $2\leq r\leq n$, using the inductive hypothesis in the second step:
\begin{align*}
    &\underbrace{(t_rt_{r+1}\cdots t_{(d+1)(n-1)+r-1})}_{\text{$(d+1)(n-1)$ terms}}(\underbrace{S_nS_{n-1}\cdots S_{n-(d+1)+1}}_{\text{$d+1$ terms}}) \\
    &=
    \underbrace{(t_rt_{r+1}\cdots t_{r+n-2})}_{\text{$n-1$ terms}}
    \underbrace{(t_{r+n-1}t_{r+n}\cdots t_{(d+1)(n-1)+r-1})}_{\text{$d(n-1)$ terms}}(\underbrace{S_nS_{n-1}\cdots S_{n-d+1}}_{\text{$d$ terms}})
    S_{n-d} \\
    &=
    \underbrace{(t_rt_{r+1}\cdots t_{r+n-2})}_{\text{$n-1$ terms}}
    \underbrace{(t_{r-1}t_{r}\cdots t_{r-1+d(n-1)-1})}_{\text{$d(n-1)$ terms}}(\underbrace{S_nS_{n-1}\cdots S_{n-d+1}}_{\text{$d$ terms}})
    S_{n-d} \\
    &=
    \underbrace{(t_rt_{r+1}\cdots t_{r+n-2})}_{\text{$n-1$ terms}}
    \underbrace{(S_{r-1}S_{r}\cdots S_{r-1+nd-1})}_{\text{$dn$ terms}} S_{n-d} \\
    &=
    \underbrace{(S_rS_{r+1}\cdots S_{n-1})}_{\text{$n-r$ terms}}S_0^{-1}S_nS_0
    \underbrace{(S_1S_2\cdots S_{r-2})}_{\text{$r-2$ terms}}
    \underbrace{(S_{r-1}S_{r}
    \cdots S_{nd+r-2})}_{\text{$dn$ terms}}S_{n-d} \\
    &=
    \underbrace{(S_rS_{r+1}\cdots S_{n-1})}_{\text{$n-r$ terms}}S_nS_0S_n^{-1}
    \underbrace{(S_1S_2
    \cdots S_{dn})}_{\text{$dn$ terms}}
    \underbrace{(S_{dn+1}S_{dn+2}\cdots S_{dn+r-2})}_{\text{$r-2$ terms}}
    S_{n-d} \\
    &=
    \underbrace{(S_rS_{r+1}\cdots S_{n-1})}_{\text{$n-r$ terms}}S_nS_0S_n^{-1}
    \underbrace{(S_1S_2
    \cdots S_{dn})}_{\text{$dn$ terms}}
    \underbrace{(S_{dn+1}S_{dn+2}\cdots S_{dn+r-2})}_{\text{$r-2$ terms}}
    S_{n-d} \\
    &=
    \underbrace{(S_rS_{r+1}\cdots S_{n-1})}_{\text{$n-r$ terms}}S_nS_0
    \underbrace{(S_1S_2
    \cdots S_{dn})}_{\text{$dn$ terms}}
    S_{n-d}^{-1}
    \underbrace{(S_{dn+1}S_{dn+2}\cdots S_{dn+r-2})}_{\text{$r-2$ terms}}
    S_{n-d} \\
\end{align*}
as required, using Lemma~\ref{lem:conjugatethrough} in the last but one step.
Note that $dn+i\equiv n-d+i+1\mod n+1$ for $1\leq i\leq r-2$, so $S_{n-d}^{-1}$ commutes with
$S_{nd+1}=S_{n-d+2},S_{nd+2}=S_{n-d+3},\ldots ,S_{nd+r-2}=S_{n-d+r-1}$, giving
the required result for $d+1$, since $2\leq r\leq n$.
The result follows by induction on $d$.
\end{proof}

\begin{lemma}
\label{lem:cornercasen3back}

Let $n\geq 3$ be an integer. Let $Q$ be the quiver on the right of Figure~\ref{fig:cornercase}(a) and $Q'$ the quiver on the left. Let $G_Q$ (respectively $G_{Q'}$) be the group with generators $s_{i}$ (respectively $t_i$)with $i$ ranging through the vertices of the quivers, satisfying the relations associated with $Q$ (respectively $Q'$). Then, there is a group homomorphism, $\varphi^Q_0: G_Q\rightarrow G_{Q'}$ given by:
\begin{align*}
\varphi^Q_0(s_n)=S_n=t_0t_nt_0^{-1}, \, \varphi^Q_0(s_b)=S_b=t_0t_bt_0^{-1}, \,\varphi^Q_0(s_i)=S_i=t_i \text{ for } i\neq n,b.
\end{align*}
\end{lemma}
\begin{proof}
    The defining relations for $G_Q$, apart from the relations corresponding to the $(n+1)$-cycle $0\rightarrow 1\rightarrow \cdots \rightarrow n\rightarrow 0$, hold by
Lemma~\ref{lem:braidS2}.
    The relations for $G_{Q'}$ corresponding to the $n$-cycle
    $1\rightarrow 2\rightarrow \cdots \rightarrow n\rightarrow 1$ are:
$$\underbrace{(t_rt_{r+1} \cdots t_{d(n-1)+r-1}) }_{\text{$d(n-1)$ terms}}=
\underbrace{(t_{r+1}t_{r+2} \cdots t_{d(n-1)+r}) }_{\text{$d(n-1)$ terms}}
,$$
for $1\leq r\leq n$.
The relations for $G_Q$ corresponding to the $(n+1)$-cycle $0\rightarrow 1\rightarrow \cdots \rightarrow n+1\rightarrow 0$ are:
$$\underbrace{S_rS_{r+1}\cdots S_{r+nd-1}}_{\text{$dn$ terms}}=
\underbrace{S_{r+1}S_{r+2}\cdots S_{dn+r}}_{\text{$dn$ terms}},$$
for $0\leq r\leq n$.
For $1\leq r\leq n-1$, these relations now follow from Lemma~\ref{lem:cyclehelp2}. The relations for $r=0$ and $r=n$ follow from Lemma~\ref{lem:oneall} applied to $g_0=S_0$, $g_1=S_1, \ldots,$ $g_n=S_n$.
\end{proof}

\begin{lemma}
\label{lem:cornercasen2back}
Let $n=2$.
Let $Q$ be the quiver on the right of Figure~\ref{fig:cornercase}(b) and $Q'$ the quiver on the left. Let $G_Q$ (respectively $G_{Q'}$) be the group with generators $s_{i}$ (respectively $t_i$), with $i$ ranging through the vertices of the quivers, satisfying the relations associated with $Q$ (respectively $Q'$). Then, there is a group homomorphism, $\varphi^Q_0: G_Q\rightarrow G_{Q'}$ given by:
\begin{align*}
    \varphi^Q_0(s_n)=S_n=t_0t_nt_0^{-1}, \,
    \varphi^Q_0(s_b)=S_b=t_0t_bt_0^{-1}, \,
    \varphi^Q_0(s_i)=S_i=t_i \text{ for } i\neq n,b.
\end{align*}
\end{lemma}
\begin{proof}
The defining relations for $G_Q$, apart from the relations corresponding to the $3$-cycle $0\rightarrow 1\rightarrow  2\rightarrow 0$, hold by Lemma~\ref{lem:braidS2}.

The relations for the $G_{Q'}$ corresponding to the unoriented edge in $Q'$ between vertices $1$ and $2$ labelled $d$ are
$$t_1t_2\cdots t_d=t_2t_3\cdots t_{d+1}$$
By Lemma~\ref{lem:cyclehelp2}, taking $r=1,2$, we have
$$(t_1t_2\cdots t_d)(S_2S_1\cdots S_{2-d+1})=S_1S_2\cdots S_{2d}$$
and
$$(t_2t_3\cdots t_{d+1})(S_2S_1\cdots S_{2-d+1})=
S_2S_3\cdots S_{2d+1},$$
giving
$$S_1S_2\cdots S_{2d}=S_2S_3\cdots S_{2d+1},$$
which is one of the cycle relations for the
$3$-cycle $0\rightarrow 1\rightarrow 2\rightarrow 0$ in $Q$.
The other cycle relations for this cycle follow from
Lemma~\ref{lem:oneall}, taking $g_0=S_0$, $g_1=S_1$ and $g_2=S_2$. Hence all the defining relations for $G_Q$ hold.
\end{proof}

\begin{figure}[ht]\scalebox{1}{
    \centering
 \begin{tikzpicture}[scale=2,
  quiverarrow/.style={black, -latex},
  mutationarc/.style={dashed, red, very thick},
  arc/.style={dashed, black},
  point/.style={gray},
  vertex/.style={black},
  conepoint/.style={gray, circle, draw=gray!100, fill=white!100, thick, inner sep=1.5pt},
  db/.style={thick, double, double distance=1.3pt, shorten <=-6pt}
  ]
 \begin{scope}
 \draw[dashed, red, thick] (2,2) arc
	[start angle=90,
		end angle=180,
		x radius=2cm,
		y radius =2cm
	] ;
    \draw[dashed] (0,0) arc
	[start angle=-90,
		end angle=0,
		x radius=2cm,
		y radius =2cm
	] ;
    \draw[dashed] (2,2) arc
	[start angle=100,
		end angle=150,
		x radius=1.45cm,
		y radius =2cm
	] ;
  \node[point] (p0) at (0,0) {$\bullet$};
  \node[point] (p1) at (-1,3) {$\bullet$};
  \node[point] (p2) at (2,2) {$\bullet$};
   \node[conepoint] (p3) at (1,1) {\scriptsize $d$};
  \draw[arc] (p0)--(p1);
  \draw[arc] (p1)--(p2);
  \draw[dashed,shorten <=3pt, shorten >=-1pt] (p3) .. controls +(10:0.5) and +(250:0.5) .. coordinate[pos=0.15](t1) coordinate[midway](w2) (p2);
  \node[vertex,rotate=-63](tt1) at (t1) {$\bowtie$};
  
  \node[vertex, label=
{[label distance=-20pt]0:{\tiny $5$}}] (v5) at (-0.42,1.22) {$\bullet$};
\node[vertex, label=
{[label distance=-5pt]90:{\tiny $3$}}] (v3) at (0.9,2.36) {$\bullet$};
\node[vertex, label=
{[label distance=-10pt]95:{\tiny $4$}}] (v4) at (0.55,1.35) {$\bullet$};
\node[vertex, label=
{[label distance=-7pt]90:{\tiny $0$}}] (v0) at (1.38,1.63) {$\bullet$};
\node[vertex, label=
{[label distance=-8pt]-90:{\tiny $1$}}] (v1) at (1.25,0.45) {$\bullet$};
\node[vertex, label=
{[label distance=-12pt]-10:{\tiny $2$}}] (v2) at (1.64,1.32) {$\bullet$};
  \draw[quiverarrow] (v3)--(v5);
  \draw[quiverarrow] (v4)--(v3);
  \draw[quiverarrow] (v5)--(v4);
  \draw[quiverarrow] (v2)--(v4);
  \draw[quiverarrow] (v0)--(v4);
  \draw [-, shorten <=-2pt, shorten >=-2pt] (v2) to node[above right,xshift=-0.1cm] {\tiny $d$} (v0);
  \draw[quiverarrow] (v1)--(v0);
  \draw[quiverarrow] (v1)--(v2);
  \draw[quiverarrow] (0.5,1.25) arc
	[start angle=170,
		end angle=230,
		x radius=1.7cm,
		y radius =0.88cm
	] ;
 \draw[db] (v4)--(0.8,1.1);
 \draw[db] (v1)--(1.1,0.8);
 \node[vertex] at (v1) {$\bullet$};
\draw [<->] (2.7,1.2) to node[above] {$\mu_4$} (3.7,1.2);
 \end{scope}
  \begin{scope}[xshift=5cm]
  \draw[dashed,red,thick](-0.98,3)--(0.018,1.1);
  \draw[dashed, red, thick] (2,2) arc
	[start angle=-15,
		end angle=-140,
		x radius=1.2cm,
		y radius =1.9cm
	] ;
    \draw[dashed] (0,0) arc
	[start angle=-90,
		end angle=0,
		x radius=2cm,
		y radius =2cm
	] ;
    \draw[dashed] (2,2) arc
	[start angle=100,
		end angle=150,
		x radius=1.45cm,
		y radius =2cm
	] ;
  \node[point] (p0) at (0,0) {$\bullet$};
  \node[point] (p1) at (-1,3) {$\bullet$};
  \node[point] (p2) at (2,2) {$\bullet$};
   \node[conepoint] (p3) at (1,1.08) {\scriptsize $d$};
  \draw[dashed,shorten <=3pt, shorten >=-1pt] (p3) .. controls +(10:0.6) and +(250:0.6) .. coordinate[pos=0.13](t1) coordinate[midway](w2) (p2);
  \node[vertex,rotate=-63](tt1) at (t1) {$\bowtie$};

  \draw[arc] (p0)--(p1);
  \draw[arc] (p1)--(p2);
  \node[vertex, label=
{[label distance=-20pt]0:{\tiny $5$}}] (v5) at (-0.42,1.22) {$\bullet$};
\node[vertex, label=
{[label distance=-5pt]90:{\tiny $3$}}] (v3) at (0.3,2.56) {$\bullet$};
\node[vertex, label=
{[label distance=-5pt]340:{\tiny $2$}}] (v2) at (1.64,1.32) {$\bullet$};
\node[vertex, label=
{[label distance=-6pt]90:{\tiny $0$}}] (v0) at (1.38,1.63) {$\bullet$};
\node[vertex, label=
{[label distance=-7pt]-90:{\tiny $1$}}] (v1) at (0.85,0.17) {$\bullet$};
\node[vertex, label=
{[label distance=-7pt]-80:{\tiny $4$}}] (v4) at (1.27,0.68) {$\bullet$};
  \draw[quiverarrow] (v0)--(v3);
  \draw[-, shorten <=-2pt, shorten >=-2pt] (v2) to node[above right,xshift=-0.1cm] {\tiny $d$} (v0);
  \draw[quiverarrow] (v4)--(1.64,1.22);
  \draw[quiverarrow] (v4)--(v0);
  \draw[quiverarrow] (0.9,0.27)--(1.2,0.58);
   \draw[quiverarrow] (1.55,1.35) arc
	[start angle=-120,
		end angle=-166,
		x radius=2.3cm,
		y radius =1.6cm
	] ;
 \draw[quiverarrow] (-0.35,1.12) arc
	[start angle=200,
		end angle=265,
		x radius=1.2cm,
		y radius =1.4cm
	] ;
 \draw[quiverarrow] (0.23,2.48) arc
	[start angle=174,
		end angle=280,
		x radius=0.73cm,
		y radius =1.6cm
	] ;
 \draw[quiverarrow] (1.1,0.57) arc
	[start angle=-70,
		end angle=-136,
		x radius=1.32cm,
		y radius =2.6cm
	] ;
 \draw[thick, double, double distance=1.3pt] (1.22,0.78)--(1.12,0.89);
 \draw[thick, double, double distance=1.3pt] (0.3,2.4)--(0.85,1.25);
 \end{scope}
 \end{tikzpicture}}
\caption{A mutation near the cone point.}
\label{fig:mutationnearcone}
\end{figure}
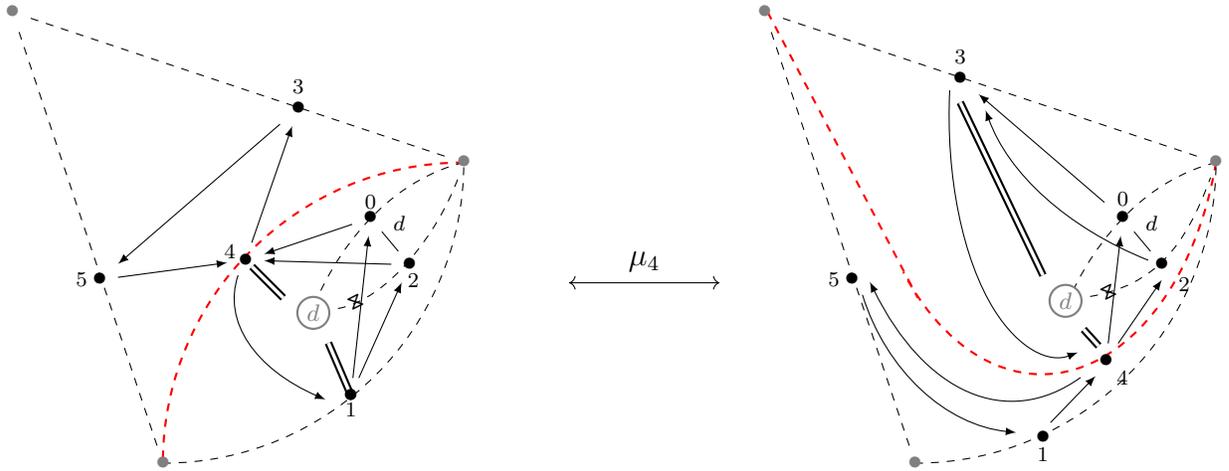

\begin{lemma}
\label{lem:mutationnearcone}
Let $Q$ be the quiver on the left of Figure~\ref{fig:mutationnearcone} and $Q'$ the quiver on the right. Let $G_Q$ (respectively $G_{Q'}$) be the group with generators $s_{i}$ (respectively $t_i$) with $0\leq i\leq 5$ satisfying the relations associated with $Q$ (respectively $Q'$). Then, there is a group homomorphism:
\begin{align*}
    \varphi^Q_4: G_{Q}\rightarrow G_{Q'} \text{ given by } &\varphi^Q_4(s_0)=S_0=t_4t_0t_4^{-1}, \, \varphi^Q_4(s_2)=S_2=t_4t_2t_4^{-1}, \, \\&\varphi^Q_4(s_5)=S_5=t_4t_5t_4^{-1},\, \varphi^Q_4(s_i)=S_i=t_i \text{ for } i= 1,3,4.
\end{align*}
\end{lemma}
\begin{proof}
In order to check that $\varphi^Q_4$ is well-defined, it is enough to prove that the elements $S_i$, for $0\leq i\leq 5$, satisfy the defining relations of $G_{Q}$.

All the braid and commutation relations for $G_Q$, as well as the cycle relations:
\begin{align*}
    S_4 S_1 S_0 S_4 &= S_1 S_0 S_4 S_1= S_0 S_4 S_1 S_0, \\
    S_4 S_1 S_2 S_4 &= S_1 S_2 S_4 S_1= S_2 S_4 S_1 S_2, \\
    S_4 S_3 S_5 S_4 &= S_3 S_5 S_4 S_3= S_5 S_4 S_3 S_5,
\end{align*}
hold by Proposition~\ref{prop:GMcase}. Moreover, the relation 
\begin{align*}
\underbrace{t_2t_0t_2t_0\cdots }_{\text{$d$ terms}} &=\underbrace{t_0t_2t_0t_2\cdots }_{\text{$d$ terms}}
\end{align*}
implies that
\begin{align*}
\underbrace{S_4^{-1}S_2S_4 S_4^{-1}S_0S_4 S_4^{-1}S_2S_4 S_4^{-1}S_0S_4\cdots S_4}_{\text{$3d$ terms}} &=\underbrace{S_4^{-1}S_0S_4 S_4^{-1}S_2S_4 S_4^{-1}S_0S_4 S_4^{-1}S_2S_4\cdots S_4}_{\text{$3d$ terms}}
\end{align*}
Since $S_4S_4^{-1}$ is the identity, we can cancel the $d-1$ occurrences of this product on both sides, reducing the number of terms to $3d-(2d-2)=d+2$ on both sides. Finally, since the two sides start and end with the same element, these can be cancelled to obtain 
\begin{align*}
\underbrace{S_2S_0S_2S_0\cdots }_{\text{$d$ terms}} &=\underbrace{S_0S_2S_0S_2\cdots }_{\text{$d$ terms}}.
\end{align*}
The double edge relation at vertex $4$ in $Q'$ is:
$$
t_2t_0t_4t_2t_0t_4=t_4t_2t_0t_4t_2t_0.
$$
This, together with the definition of the $S_i$'s, implies that
\begin{align*}
    S_4^{-1}S_2S_4 S_4^{-1}S_0S_4 S_4 S_4^{-1}S_2S_4 S_4^{-1}S_0S_4 S_4=
    S_4 S_4^{-1}S_2S_4 S_4^{-1}S_0S_4 S_4 S_4^{-1}S_2S_4 S_4^{-1}S_0S_4.
\end{align*}
Multiplying both sides by $S_4$ on the left and by $S_4^{-1}$ on the right and cancelling the occurrences of $S_4S_4^{-1}$, we obtain $S_2S_0S_4S_2S_0S_4=S_4S_2S_0S_4S_2S_0$. It only remains to show the  double edge relation at vertex $1$ holds. Using the definition of the $S_i$'s and the defining relations of $G_{Q'}$, we have 
\begin{align*}
S_2S_0S_1S_2S_0S_1&=t_4t_2\cancel{t_4^{-1}t_4}t_0t_4^{-1}t_1t_4t_2\cancel{t_4^{-1}t_4}t_0t_4^{-1}t_1
=t_4t_2t_0t_1t_4t_1^{-1}t_2t_0t_4^{-1}t_1\\
&=t_4t_1t_2t_0t_4t_2t_0t_1^{-1}t_4^{-1}t_1
=t_4t_1\underline{t_2t_0t_4t_2t_0t_4}t_1^{-1}t_4^{-1}\\
&=t_4t_1\underline{t_4t_2t_0t_4t_2t_0}t_1^{-1}t_4^{-1}
=t_1t_4t_1t_2t_0t_4t_2t_0t_1^{-1}t_4^{-1}\\
&=t_1t_4t_2t_0t_1t_4t_1^{-1}t_2t_0t_4^{-1}
=t_1t_4t_2t_0t_4^{-1}t_1t_4t_2t_0t_4^{-1}\\
&=t_1t_4t_2t_4^{-1}t_4t_0t_4^{-1}t_1t_4t_2t_4^{-1}t_4t_0t_4^{-1}
=S_1S_2S_0S_1S_2S_0,
\end{align*}
where the underlined relation is the relation for the double edge at vertex $4$ in $Q'$, while all the other equalities follow from braid relations, commutations or multiplying by, or simplifying, the identity $t_i^{-1}t_i$. Hence all the defining relations of $G_Q$ hold.
\end{proof}

\begin{lemma}
\label{lem:mutationnearcone_otherway}
Let $Q$ be the quiver on the right of Figure~\ref{fig:mutationnearcone} and $Q'$ the quiver on the left. Let $G_Q$ (respectively $G_{Q'}$) be the group with generators $s_{i}$ (respectively $t_i$) with $0\leq i\leq 5$ satisfying the relations associated with $Q$ (respectively $Q'$). Then, there is a group homomorphism
\begin{align*}
    \varphi^Q_4: G_Q\rightarrow G_{Q'} \text{ given by } \varphi^Q_4(s_1)=S_1=t_4t_1t_4^{-1}, \, \varphi^Q_4(s_3)=S_3=t_4t_3t_4^{-1}, \,\varphi^Q_4(s_i)=S_i=t_i \text{ for } i\neq 1,3.
\end{align*}
\end{lemma}

\begin{proof}
 In order to check that $\varphi_4$ is well-defined, it is enough to prove that the elements $S_i$, for $0\leq i\leq 5$, satisfy the defining relations of $G_Q$.

First note that the relations
\begin{align*}
S_2S_0S_4S_2S_0S_4=S_4S_2S_0S_4S_2S_0, \,\,
    \underbrace{S_0S_2\cdots }_{\text{$d$ terms}} =\underbrace{S_2S_0\cdots }_{\text{$d$ terms}}
\end{align*}
trivially follow from the corresponding relations for $Q'$. Moreover, all the braid and commutation relations for $Q$, as well as the cycle relations
\begin{align*}
    S_1 S_4 S_5 S_1= S_4 S_5 S_1 S_4= S_5 S_1 S_4 S_5, \\
    S_2 S_3 S_4 S_2= S_3 S_4 S_2 S_3= S_4 S_2 S_3 S_4,\\
    S_0 S_3 S_4 S_0= S_3 S_4 S_0 S_1= S_4 S_0 S_1 S_4
\end{align*}
hold by Proposition~\ref{prop:GMcase}.
It only remains to show the second double edge relation holds. Using the definition of the elements $S_i$ and the defining relations of $G_{Q'}$, we have 
\begin{align*}
    S_2 S_0 S_3 S_2 S_0 S_3&= t_2t_0t_4t_3t_4^{-1}t_2t_0t_4t_3t_4^{-1}
    =t_2t_0t_3^{-1}t_4t_3t_2t_0t_4t_3t_4^{-1}
    =t_3^{-1}t_2t_0t_4t_2t_0t_3t_4t_3t_4^{-1}\\
    &=t_3^{-1}\underline{t_2t_0t_4t_2t_0t_4}t_3t_4t_4^{-1}
    =t_3^{-1}\underline{t_4t_2t_0t_4t_2t_0}t_3
    =t_4t_4^{-1}t_3^{-1}t_4t_2t_0t_4t_3t_2t_0\\
    &=t_4t_3t_4^{-1}t_3^{-1}t_2t_0t_4t_3t_2t_0
    =t_4t_3t_4^{-1}t_2t_0t_3^{-1}t_4t_3t_2t_0
    =t_4t_3t_4^{-1}t_2t_0t_4t_3t_4^{-1}t_2t_0\\
    &= S_3 S_2 S_0 S_3 S_2 S_0,
\end{align*}
where the underlined relation is the relation for the double edge at vertex $4$ in $Q'$, while all the other equalities follow from braid relations, commutations or multiplying by, or simplifying, the identity $t_it_i^{-1}$. Hence all the defining relations of $G_Q$ hold as required.
\end{proof}

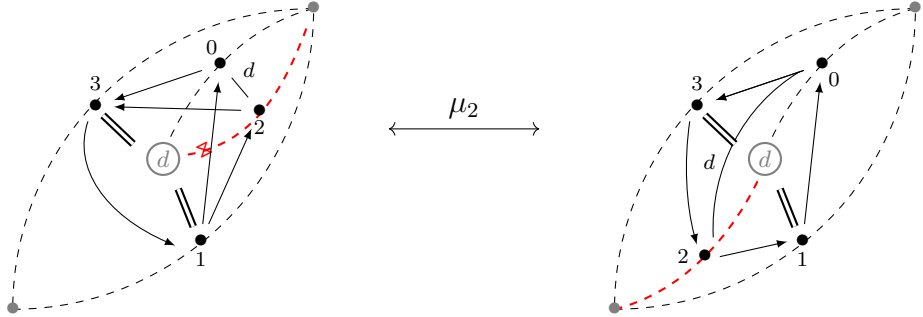
\begin{figure}[ht]\scalebox{1}{
    \centering
 \begin{tikzpicture}[scale=2,
  quiverarrow/.style={black, -latex},
  mutationarc/.style={dashed, red, very thick},
  arc/.style={dashed, black},
  point/.style={gray},
  vertex/.style={black},
  conepoint/.style={gray, circle, draw=gray!100, fill=white!100, thick, inner sep=1.5pt},
  db/.style={thick, double, double distance=1.3pt, shorten <=-6pt}
  ]
 \begin{scope}
 \draw[dashed] (2,2) arc
	[start angle=90,
		end angle=180,
		x radius=2cm,
		y radius =2cm
	] ;
    \draw[dashed] (0,0) arc
	[start angle=-90,
		end angle=0,
		x radius=2cm,
		y radius =2cm
	] ;
    \draw[dashed] (2,2) arc
	[start angle=100,
		end angle=150,
		x radius=1.45cm,
		y radius =2cm
	] ;
 \node[point] (p0) at (0,0) {$\bullet$};
 \node[point] (p1) at (2,2) {$\bullet$};
 \node[conepoint] (p3) at (1,1) {\scriptsize $d$};

\draw[dashed,red, thick, shorten <=3pt, shorten >=-1pt] (p3) .. controls +(10:0.6) and +(250:0.6) .. coordinate[pos=0.12](t1) coordinate[midway](w2) (p1);
\node[vertex,rotate=-63, red](tt1) at (t1) {$\bowtie$};
  
 \node[vertex, label=
{[label distance=-22pt]-90:{\tiny $3$}}] (v3) at (0.55,1.35) {$\bullet$};
 \node[vertex, label=
{[label distance=-20pt]90:{\tiny $1$}}] (v1) at (1.25,0.45) {$\bullet$};
 \node[vertex, label=
{[label distance=-10pt]175:{\tiny $0$}}] (v0) at (1.38,1.63) {$\bullet$};
 \node[vertex, label=
{[label distance=-7pt]270:{\tiny $2$}}] (v2) at (1.64,1.32) {$\bullet$};
  \draw[quiverarrow] (v0)--(v3);
  \draw[quiverarrow] (v2)--(v3);
  \draw[quiverarrow] (v1)--(v0);
  \draw[quiverarrow] (v1)--(v2);
  \draw [-, shorten <=-2pt, shorten >=-2pt] (v2) to node[above right,xshift=-0.1cm] {\tiny $d$} (v0);
  \draw[quiverarrow] (0.5,1.25) arc
	[start angle=170,
		end angle=230,
		x radius=1.7cm,
		y radius =0.88cm
	] ;
 \draw[thick, double, double distance=1.3pt] (0.6,1.29)--(0.8,1.1);
 \draw[thick, double, double distance=1.3pt] (1.2,0.55)--(1.1,0.8);
  \draw [<->] (2.5,1.2) to node[above] {$\mu_2$} (3.5,1.2);
 \end{scope}
  \begin{scope}[xshift=4cm]
  \draw[dashed] (2,2) arc
	[start angle=90,
		end angle=180,
		x radius=2cm,
		y radius =2cm
	] ;
    \draw[dashed] (0,0) arc
	[start angle=-90,
		end angle=0,
		x radius=2cm,
		y radius =2cm
	] ;
    \draw[dashed] (2,2) arc
	[start angle=100,
		end angle=150,
		x radius=1.45cm,
		y radius =2cm
	] ;
 \draw[dashed, red, thick] (0,0) arc
	[start angle=-80,
		end angle=-30,
		x radius=1.45cm,
		y radius =2cm
	] ;
  \node[point] (p0) at (0,0) {$\bullet$};
 \node[point] (p1) at (2,2) {$\bullet$};
 \node[conepoint] (p3) at (1,1) {\scriptsize $d$};
 \node[vertex, label=
{[label distance=-22pt]-90:{\tiny $3$}}] (v3) at (0.55,1.35) {$\bullet$};
 \node[vertex, label=
{[label distance=-20pt]90:{\tiny $1$}}] (v1) at (1.25,0.45) {$\bullet$};
 \node[vertex, label=
{[label distance=-10pt]330:{\tiny $0$}}] (v0) at (1.38,1.63) {$\bullet$};
 \node[vertex, label=
{[label distance=-5pt]180:{\tiny $2$}}] (v2) at (0.6,0.35) {$\bullet$};
  \draw[quiverarrow] (v0)--(v3);
  \draw[quiverarrow] (v1)--(v0);
  \draw[quiverarrow] (0.7,0.35)--(1.15,0.45);
  \draw[quiverarrow] (v0)--(v3);
  \draw (0.63,0.98) node{$\scriptstyle d$};
  \draw[-] (1.23,1.58) arc
	[start angle=120,
		end angle=183,
		x radius=1.15cm,
		y radius =1.2cm
	] ;
 \draw[quiverarrow] (0.5,1.25) arc
	[start angle=170,
		end angle=198,
		x radius=1.4cm,
		y radius =1.7cm
	] ;
 \draw[thick, double, double distance=1.3pt] (0.6,1.29)--(0.8,1.1);
 \draw[thick, double, double distance=1.3pt] (1.2,0.55)--(1.1,0.8);
 \end{scope}
 \end{tikzpicture}}
 \caption{Mutation of a tagged to an untagged arc and viceversa, case $1$.}
\label{fig:mutation_tagged}
\end{figure}

\begin{lemma}
\label{lem:GM_with_d}
Let $d\geq 1$ be an integer and let $G$ be a group containing elements $a,b$ satisfying
\begin{equation}
\underbrace{ab\cdots }_{d \text{ terms}}=
\underbrace{ab\cdots }_{d \text { terms}}.
\label{eq:braidab}
\end{equation}
Then the elements $A=a$ and $B=aba^{-1}$ satisfy the same relation:
$$\underbrace{AB\cdots }_{d \text{ terms}}=
\underbrace{AB\cdots }_{d \text { terms}}.$$
\end{lemma}
\begin{proof}
By relation (\ref{eq:braidab}), we have
\begin{align*}    \underbrace{A^{-1}BA\cancel{AA^{-1}}BA\cancel{A}\cdots A}_{\text{$3\cdot \lceil d/2\rceil+\lfloor d/2 \rfloor$ terms}} &=\underbrace{\cancel{AA^{-1}}BA\cancel{AA^{-1}}AB\cdots A,}_{\text{$3\cdot \lfloor d/2\rfloor+\lceil d/2 \rceil$ terms}}
    \end{align*}
    where for a real number $r$, $\lfloor r\rfloor$ (respectively $\lceil r\rceil$) is the largest (respectively smallest) integer at most (respectively at least) equal to $r$.
    Cancelling the occurrences of $AA^{-1}$, on the left hand side we cancel $2\cdot(\lceil d/2 \rceil -1)$ terms, while on the right hand side we cancel $2\cdot \lfloor d/2\rfloor $ terms. Hence we have $d+2$ terms on the left hand side and $d$ terms on the right hand side. Since the last term on both sides is $A$, we can cancel it. Moreover, multiplying both sides by $A$ on the left and simplifying the left hand side, we obtain
    \begin{align*}
        \underbrace{BA\cdots }_{\text{$d$ terms}} &=\underbrace{AB\cdots }_{\text{$d$ terms}},
    \end{align*}
as required.
\end{proof}

\begin{lemma}
\label{lem:mutation_tagged}
Let $Q$ be the quiver on the left of Figure~\ref{fig:mutation_tagged} and $Q'$ the quiver on the right. Let $G_Q$ (respectively $G_{Q'}$) be the group with generators $s_i$ (respectively $t_i$) with $0\leq i\leq 3$ satisfying the relations associated with $Q$ (respectively $Q'$). Then there is a group homomorphism as follows:
\begin{align*}
    \varphi^Q_2: G_{Q}\rightarrow G_{Q'} \text{ defined as } \varphi^Q_2(s_1)=S_1=t_2t_1t_2^{-1}, \, \varphi^Q_2(t_i)=S_i=T_i \text{ for } i=0,2,3.
\end{align*}
\end{lemma}

\begin{proof}
In order to check that $\varphi^Q_2$ is well-defined, it is enough to prove that the elements $S_i$, for $0\leq i\leq 3$, satisfy the defining relations of $G_{Q}$.

The relations $S_0S_3S_0=S_3S_0S_3$, $S_2S_0S_3S_2S_0S_3= S_3S_2S_0S_3S_2S_0$ and
   \begin{align*}
        \underbrace{S_0S_2\cdots }_{\text{$d$ terms}} &=\underbrace{S_2 S_0\cdots }_{\text{$d$ terms}}
    \end{align*}
trivially follow from the corresponding defining relations of $G_{Q'}$. Moreover, the relations $S_2S_3S_2=S_3S_2S_3$, $S_1S_3S_1=S_3S_1S_3$, $S_2S_1S_2=S_1S_2S_1$ and $S_1S_2S_3S_1=S_2S_3S_1S_2=S_3S_1S_2S_3$ hold by Proposition~\ref{prop:GMcase}.

The relation for the double edge at $1$ in $Q'$:
$$t_0t_2t_1t_0t_2t_1=t_1t_0t_2t_1t_0t_2$$
implies that
\begin{align*}
S_0\cancel{S_2S_2^{-1}}S_1S_2S_0\cancel{S_2S_2^{-1}}S_1S_2=S_2^{-1}S_1S_2S_0\cancel{S_2S_2^{-1}}S_1S_2S_0S_2\iff
    S_2S_0S_1S_2S_0S_1=S_1S_2S_0S_1S_2S_0,
\end{align*}
that is the double edge relation corresponding to vertex $1$ in $Q$.

The braid relation $t_1t_0t_1=t_0t_1t_0$ together with the relations proved above imply that
\begin{align*}
S_2^{-1}S_1S_2S_0S_2^{-1}S_1S_2=S_0S_2^{-1}S_1S_2S_0&\iff
    S_1S_2S_0S_1S_2S_1^{-1}=S_2S_0S_1S_2S_1^{-1}S_0\\
    &\iff \hcancel{S_2S_0S_1S_2}S_0S_1S_0^{-1}S_1^{-1}=\hcancel{S_2S_0S_1S_2}S_1^{-1}S_0\\
    &\iff S_1S_0S_1=S_0S_1S_0.
\end{align*}

The relation $t_2t_1t_0t_3t_2t_1=t_3t_2t_1t_0t_3t_2$ together with the relations proved above imply
\begin{align*}    \cancel{S_2S_2^{-1}}S_1S_2S_0S_3\cancel{S_2S_2^{-1}}S_1S_2=S_3\cancel{S_2S_2^{-1}}S_1S_2S_0S_3S_2
    &\iff S_1S_2S_0S_3S_1=S_3S_1S_2S_0S_3
    \\&\iff  \hcancel{S_1S_2S_3}S_0S_3S_0^{-1}S_1=\hcancel{S_1S_2S_3}S_1S_3^{-1}S_0S_3
    \\&\iff  S_3^{-1}S_0S_3S_1=S_1S_0S_3S_0^{-1}
    \\&\iff  S_0S_3S_1S_0=S_3S_1S_0S_3.
\end{align*}

By \cite[Lem. 2.4]{GM}, we conclude that $S_0S_3S_1S_0=S_3S_1S_0S_3=S_1S_0S_3S_1$. Hence all the defining relations of $G_Q$ hold, as required.
\end{proof}

\begin{lemma}
\label{lem:mutation_tagged_reverse}
Let $Q$ be the quiver on the right of Figure~\ref{fig:mutation_tagged} and $Q'$ the quiver on the left. Let $G_Q$ (respectively $G_{Q'}$) be the group with generators $s_i$ (respectively $t_i$) with $0\leq i\leq 3$ satisfying the relations associated with $Q$ (respectively $Q'$). Then there is a group homomorphism as follows:
\begin{align*}
    \varphi^Q_2: G_Q\rightarrow G_{Q'} \text{ defined as } \varphi^Q_2(s_3)=S_3=t_2t_3t_2^{-1}, \, \varphi^Q_2(s_0)=S_0=t_2t_0t_2^{-1}, \,\varphi^Q_2(s_i)=t_i \text{ for } i=1,2.
\end{align*}
\end{lemma}

\begin{proof}
    In order to check that $\varphi^Q_2$ is well-defined, it is enough to prove that the elements $S_i$, for $0\leq i\leq 3$, satisfy the defining relations of $G_Q$.
    
    The relations $S_2S_3S_2=S_3S_2S_3$, $S_2S_1S_2=S_1S_2S_1$ and $S_1S_3=S_3S_1$ hold by Proposition~\ref{prop:GMcase}, while
    \begin{align*}
        \underbrace{S_0S_2\cdots }_{\text{$d$ terms}} &=\underbrace{S_2 S_0\cdots }_{\text{$d$ terms}}.
    \end{align*}
holds by Lemma~\ref{lem:GM_with_d}. The relation $t_3t_0t_3=t_0t_3t_0$ implies that
    \begin{align*}
        S_2^{-1}S_3\cancel{ S_2 S_2^{-1}}S_0 \cancel{S_2 S_2^{-1}}S_3S_2=S_2^{-1}S_0\cancel{S_2 S_2^{-1}}S_3\cancel{S_2S_2^{-1}}S_0S_2.
    \end{align*}
    Since the first and last term are equal on the two sides, we can cancel them and obtain $S_3S_0S_3=S_0S_3S_0$.
    The relation $t_2t_0t_1t_2t_0t_1= t_1t_2t_0t_1t_2t_0$ for the double edge at $1$ in $Q'$ implies that
    \begin{align*}
        \cancel{S_2 S_2^{-1}}S_0S_2 S_1 \cancel{S_2 S_2^{-1}}S_0S_2 S_1
        =S_1 \cancel{S_2 S_2^{-1}}S_0S_2 S_1 \cancel{S_2 S_2^{-1}}S_0S_2.
    \end{align*}
    So the relation for the double edge at $1$ in $Q$
    $$S_0S_2S_1S_0S_2S_1=S_1S_0S_2S_1S_0S_2$$
    follows after cancellations.
     The relation $t_2t_0t_3t_2t_0t_3= t_3t_2t_0t_3t_2t_0$ for the double edge at $3$ in $Q'$ implies that
     \begin{align*}
         \cancel{S_2 S_2^{-1}}S_0 \cancel{S_2 S_2^{-1}}S_3 S_2 \cancel{S_2 S_2^{-1}}S_0 \cancel{S_2 S_2^{-1}}S_3S_2
         =S_2^{-1}S_3S_2 \cancel{S_2 S_2^{-1}}S_0 \cancel{S_2 S_2^{-1}}S_3 S_2\cancel{S_2 S_2^{-1}}S_0S_2.
     \end{align*}
    Multiplying both sides by $S_2$ on the left and by $S_2^{-1}$ on the right and simplifying, the relation for the double edge at $3$ in $Q$:   $$S_2S_0S_3S_2S_0S_3=S_3S_2S_0S_3S_2S_0$$
    follows.

    Using one of the braid relations from the left diagram and the relations found above, we have that
    \begin{align*}
        t_1t_0t_1=&t_0t_1t_0\iff S_1S_2^{-1}S_0S_2S_1= S_2^{-1}S_0S_2 S_1 S_2^{-1}S_0S_2
        \iff S_1^{-1}S_2S_1S_0S_2S_1= S_0S_1^{-1}S_2 S_1S_0S_2\\ 
        &\iff \underline{S_2S_1S_0S_2S_1}= S_1S_0S_1^{-1}S_2 S_1S_0S_2
        \iff  \underline{S_0^{-1}S_1S_0\hcancel{S_2S_1S_0S_2}}= S_1S_0S_1^{-1}\hcancel{S_2 S_1S_0S_2}
        \\&\iff S_1S_0S_1=S_0S_1S_0,
    \end{align*}
    where the underlined relation follows from the double edge relation at vertex $1$ in $Q$.

    The last two relations left to prove correspond to the $4$-cycle around the cone point in $Q$.
    The relation $t_1t_0t_3t_1=t_0t_3t_1t_0=t_3t_1t_0t_3$ implies that
    \begin{align*}
    S_1S_2^{-1}S_0\cancel{S_2S_2^{-1}}S_3S_2S_1=S_2^{-1}S_0\cancel{S_2S_2^{-1}}S_3S_2S_1S_2^{-1}S_0S_2=S_2^{-1}S_3S_2S_1S_2^{-1}S_0\cancel{S_2S_2^{-1}}S_3S_2
    \end{align*}
    Using the relations already proved, we see that the equality of the first and third expressions above is true if and only if
    \begin{align*}
    S_2S_1S_2^{-1}S_0S_3S_2S_1=S_3S_1^{-1}S_2S_1S_0S_3S_2 
        &\iff  S_1^{-1}S_2S_1S_0S_3S_2S_1=S_1^{-1}S_3S_2S_1S_0S_3S_2 
        \\&\iff S_2S_1S_0S_3S_2S_1=S_3S_2S_1S_0S_3S_2. 
    \end{align*}

    Similarly, the equality of the second and third expressions above is true if and only if
    \begin{align*}
        S_0S_3S_1^{-1}S_2S_1S_0=S_3S_1^{-1}S_2S_1S_0S_3
        &\iff S_1S_0S_1^{-1}S_3S_2S_1S_0=S_3S_2S_1S_0S_3\\
        &\iff S_0^{-1}S_1S_0S_3S_2S_1S_0=S_3S_2S_1S_0S_3\\
        &\iff S_1S_0S_3S_2S_1S_0=S_0S_3S_2S_1S_0S_3.
    \end{align*}
    Hence all the relations corresponding to the quiver $Q$ hold, as required.
\end{proof}

\begin{figure}[ht]\scalebox{1}{
    \centering
 \begin{tikzpicture}[scale=2,
  quiverarrow/.style={black, -latex},
  iquiverarrow/.style={black, -latex, <-},
  mutationarc/.style={dashed, red, very thick},
  arc/.style={dashed, black},
  point/.style={gray},
  vertex/.style={black},
  conepoint/.style={gray, circle, draw=gray!100, fill=white!100, thick, inner sep=1.5pt},
  db/.style={thick, double, double distance=1.3pt, shorten <=-6pt}
  ]
 \begin{scope}
 \draw[dashed] (2,2) arc
	[start angle=90,
		end angle=180,
		x radius=2cm,
		y radius =2cm
	] ;
    \draw[dashed] (0,0) arc
	[start angle=-90,
		end angle=0,
		x radius=2cm,
		y radius =2cm
	] ;
 \draw[dashed, thick] (2,2) arc
	[start angle=-10,
		end angle=-60,
		x radius=2cm,
		y radius =1.45cm
	] ;
 \node[point] (p0) at (0,0) {$\bullet$};
 \node[point] (p1) at (2,2) {$\bullet$};
 \node[conepoint] (p3) at (1,1) {\scriptsize $d$};

\draw[dashed,red, thick, shorten <=3pt, shorten >=-2pt] (p3) .. controls +(80:0.45) and +(200:0.45) .. coordinate[pos=0.11](t1) coordinate[midway](w2) (p1);
\node[vertex,rotate=-23, red](tt1) at (t1) {$\bowtie$};
 
 \node[vertex, label=
{[label distance=-22pt]-90:{\tiny $3$}}] (v3) at (0.55,1.35) {$\bullet$};
 \node[vertex, label=
{[label distance=-20pt]90:{\tiny $1$}}] (v1) at (1.25,0.45) {$\bullet$};
 \node[vertex, label=
{[label distance=-21pt]-90:{\tiny $2$}}] (v0) at (1.38,1.63) {$\bullet$};
 \node[vertex, label=
{[label distance=-7pt]270:{\tiny $0$}}] (v2) at (1.64,1.32) {$\bullet$};
  \draw[quiverarrow] (v0)--(v3);
  \draw[quiverarrow] (v2)--(v3);
  \draw[quiverarrow] (v1)--(v0);
  \draw[quiverarrow] (v1)--(v2);
  \draw [-, shorten <=-2pt, shorten >=-2pt] (v2) to node[above right,xshift=-0.1cm] {\tiny $d$} (v0);
  \draw[quiverarrow] (0.5,1.25) arc
	[start angle=170,
		end angle=230,
		x radius=1.7cm,
		y radius =0.88cm
	] ;
 \draw[thick, double, double distance=1.3pt] (0.6,1.29)--(0.8,1.1);
 \draw[thick, double, double distance=1.3pt] (1.2,0.55)--(1.1,0.8);
  \draw [<->] (2.5,1.2) to node[above] {$\mu_2$} (3.5,1.2);
 \end{scope}
  \begin{scope}[xshift=4cm]
   \draw[thick, double, double distance=1.3pt] (0.6,1.29)--(0.8,1.1);
 \draw[thick, double, double distance=1.3pt] (1.2,0.55)--(1.1,0.8);
  \draw[dashed] (2,2) arc
	[start angle=90,
		end angle=180,
		x radius=2cm,
		y radius =2cm
	] ;
    \draw[dashed] (0,0) arc
	[start angle=-90,
		end angle=0,
		x radius=2cm,
		y radius =2cm
	] ;
 \draw[dashed] (2,2) arc
	[start angle=-10,
		end angle=-60,
		x radius=2cm,
		y radius =1.45cm
	] ;
  \draw[dashed, red, thick] (0,0) arc
	[start angle=170,
		end angle=120,
		x radius=2cm,
		y radius =1.45cm
	] ;
  \node[point] (p0) at (0,0) {$\bullet$};
 \node[point] (p1) at (2,2) {$\bullet$};
 \node[conepoint] (p3) at (1,1) {\scriptsize $d$};
 \node[vertex, label=
{[label distance=-22pt]-90:{\tiny $3$}}] (v3) at (0.55,1.35) {$\bullet$};
 \node[vertex, label=
{[label distance=-20pt]90:{\tiny $1$}}] (v1) at (1.25,0.45) {$\bullet$};
 \node[vertex, label=
{[label distance=-21pt]-90:{\tiny $0$}}] (v2) at (1.64,1.32) {$\bullet$};
 \node[vertex, label=
{[label distance=-19pt]90:{\tiny $2$}}] (v0) at (0.4,0.65) {$\bullet$};
  \draw[quiverarrow] (0.5,1.24)--(0.38,0.8);
  \draw[quiverarrow] (v0)--(v1);
  \draw[quiverarrow] (v1)--(v2);
  \draw[quiverarrow] (1.52,1.38)--(0.68,1.41);
  \draw (0.63,0.98) node{$\scriptstyle d$};
 \draw[-] (0.45,0.8) arc
	[start angle=175,
		end angle=70,
		x radius=0.75cm,
		y radius =0.55cm
	] ;
 \end{scope}
 \end{tikzpicture}}
 \caption{Mutation of a tagged to an untagged arc and viceversa, case $2$.}
\label{fig:mutation_untagged}
\end{figure}
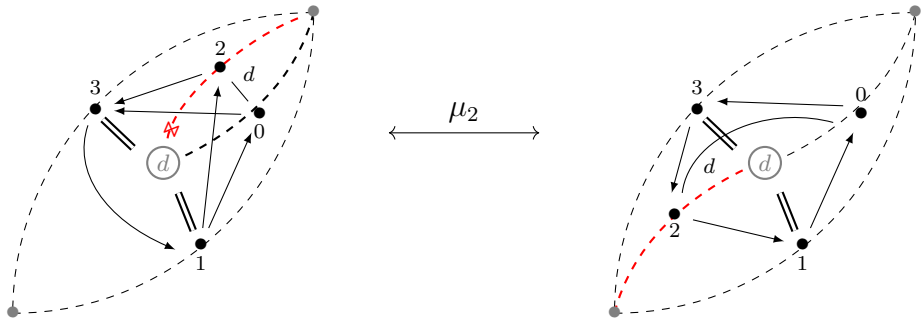

\begin{lemma}
\label{lem:mutation_untagged}
Let $Q$ be the quiver on the left of Figure~\ref{fig:mutation_untagged} and $Q'$ the quiver on the right. Let $G_Q$ (respectively $G_{Q'}$) be the group with generators $s_i$ (respectively $t_i$) with $0\leq i\leq 3$ satisfying the relations associated with $Q$ (respectively $Q'$). Then there is a group homomorphism:
$\varphi^Q_2: G_{Q}\rightarrow G_{Q'}$ defined as:
\begin{align*}
\varphi^Q_2(s_1)=S_1=t_2t_1t_2^{-1}, \, \varphi^Q_2(s_0)=S_0=t_2t_0t_2^{-1}, \, \varphi^Q_2(s_i)=S_i=t_i \text{ for } i=2,3.
\end{align*}
\end{lemma}

\begin{proof}
In order to check that $\varphi^Q_2$ is well-defined, it is enough to prove that the elements $S_i$, for $0\leq i\leq 3$, satisfy the defining relations of $G_Q$.
The relations $S_2S_3S_2=S_3S_2S_3$, $S_1S_3S_1=S_3S_1S_3$, $S_2S_1S_2=S_1S_2S_1$ and $S_1S_2S_3S_1=S_2S_3S_1S_2=S_3S_1S_2S_3$ hold by Proposition~\ref{prop:GMcase}, while the relation 
    \begin{align*}
        \underbrace{S_2S_0\cdots }_{\text{$d$ terms}} &=\underbrace{S_0 S_2\cdots }_{\text{$d$ terms}}
    \end{align*}
    holds by Lemma~\ref{lem:GM_with_d}.
    The relation $t_1t_0t_1=t_0t_1t_0$  implies that
    \begin{align*}
        S_2^{-1}S_1\cancel{ S_2 S_2^{-1}}S_0 \cancel{S_2 S_2^{-1}}S_1S_2=S_2^{-1}S_0\cancel{S_2 S_2^{-1}}S_1\cancel{S_2S_2^{-1}}S_0S_2.
    \end{align*}
    Since the first and last term are equal on the two sides, we can cancel them and obtain $S_1S_0S_1=S_0S_1S_0$.
    The relation for the double edge at $1$ in $Q'$
$$t_0t_2t_1t_0t_2t_1=t_1t_0t_2t_1t_0t_2$$
    implies that
\begin{align*}
&S_2^{-1}S_0S_2\cancel{S_2S_2^{-1}}S_1\cancel{S_2S_2^{-1}}S_0S_2\cancel{S_2S_2^{-1}}S_1S_2
=S_2^{-1}S_1\cancel{S_2S_2^{-1}}S_0S_2\cancel{S_2S_2^{-1}}S_1\cancel{S_2S_2^{-1}}S_0S_2S_2\\
&\iff S_0S_2S_1S_0S_2S_1
=S_1S_0S_2S_1S_0S_2,
\end{align*}
that is the double edge relation corresponding to vertex $1$ in $Q$. Similarly, the relation for the double edge at $3$ in $Q'$:
$$t_2t_0t_3t_2t_0t_3=t_3t_2t_0t_3t_2t_0$$
implies that
\begin{align*}
\cancel{S_2S_2^{-1}}S_0S_2S_3\cancel{S_2S_2^{-1}}S_0S_2S_3=S_3\cancel{S_2S_2^{-1}}S_0S_2S_3\cancel{S_2S_2^{-1}}S_0S_2
    \iff S_0S_2S_3S_0S_2S_3=S_3S_0S_2S_3S_0S_2,
\end{align*}
that is the double edge relation corresponding to vertex $3$ in $Q$. Using the defining relations of $Q'$, we have that
\begin{align*}
    S_0S_3S_0&=t_2t_0t_2^{-1}t_3t_2t_0t_2^{-1}    =t_2t_0t_3t_2t_0t_3t_3^{-1}t_0^{-1}t_3^{-1}t_0t_2^{-1}\\    &=t_3t_2t_0t_3t_2\cancel{t_0t_0^{-1}}t_3^{-1}\cancel{t_0^{-1}t_0}t_2^{-1}
    =t_3t_2t_0t_2^{-1}t_3\cancel{t_2t_2^{-1}}=S_3S_0S_3.
\end{align*}

The relation $t_2t_1t_0t_3t_2t_1=t_3t_2t_1t_0t_3t_2$, together with the relations proved above, implies that
\begin{align*}
\cancel{S_2S_2^{-1}}S_1\cancel{S_2S_2^{-1}}S_0S_2S_3\cancel{S_2S_2^{-1}}S_1S_2    &=S_3\cancel{S_2S_2^{-1}}S_1\cancel{S_2S_2^{-1}}S_0S_2S_3S_2 \\
    \iff S_1S_0S_2S_3S_1S_2&=S_3S_1S_0S_2S_3S_2\\
    \iff S_1S_0S_3S_1\hcancel{S_2S_3}&=S_3S_1S_0S_3\hcancel{S_2S_3}.\\
\end{align*}
By \cite[Lem. 2.4]{GM}, we conclude that $S_1S_0S_3S_1=S_3S_1S_0S_3=S_0S_3S_1S_0$. Hence all the defining relations of $G_Q$ hold, as required.
\end{proof}

\begin{lemma}
\label{lem:mutation_untagged_reverse}
Let $Q$ be the quiver on the right of Figure~\ref{fig:mutation_untagged} and $Q'$ the quiver on the left. Let $G_Q$ (respectively $G_{Q'}$) be the group with generators $s_i$ (respectively $t_i$) with $0\leq i\leq 3$ satisfying the relations associated with $Q$ (respectively $Q'$). Then there is a group homomorphism as follows:
\begin{align*}
    \varphi^Q_2: G_Q\rightarrow G_{Q'} \text{ defined as } \varphi^Q_2(s_3)=S_3=t_2t_3t_2^{-1}, \,\varphi^Q_2(s_i)=S_i=t_i \text{ for } i=0,1,2.
\end{align*}
\end{lemma}

\begin{proof}
    In order to check that $\varphi^Q_2$ is well-defined, it is enough to prove that the elements $S_i$, for $0\leq i\leq 3$, satisfy the defining relations of $G_Q$.

The relation $S_1S_0S_1=S_0S_1S_0$ trivially follows from the relation $t_1t_0t_1=t_0t_1t_0$.
 The relations $S_2S_3S_2=S_3S_2S_3$, $S_2S_1S_2=S_1S_2S_1$ and $S_1S_3=S_3S_1$ hold by Proposition~\ref{prop:GMcase}, while the relation
 \begin{align*}
        \underbrace{S_2S_0\cdots }_{\text{$d$ terms}} &=\underbrace{S_0 S_2\cdots }_{\text{$d$ terms}}.
    \end{align*}
is trivial to check.
The relation  $S_0S_2S_1S_0S_2S_1= S_1S_0S_2S_1S_0S_2$ for the double edge at $1$ follows trivially from the corresponding for $Q'$. Moreover, the relation $t_0t_2t_3t_0t_2t_3= t_3t_0t_2t_3t_0t_2$ for the double edge at $3$ in $Q'$ implies that
   \begin{align*}
         S_0\cancel{S_2 S_2^{-1}}S_3S_2S_0 \cancel{S_2 S_2^{-1}}S_3 S_2 
         =S_2^{-1}S_3S_2S_0 \cancel{S_2 S_2^{-1}}S_3S_2S_0 S_2
         \iff S_2  S_0S_3S_2S_0 S_3 =S_3S_2S_0 S_3S_2S_0,
     \end{align*}
    that is, the relation for the double edge in $Q$ follows.

    Using the relations found above, we have that
    \begin{align*}
        t_0t_3t_0=t_3t_0t_3&\iff S_0S_2^{-1}S_3S_2S_0= S_2^{-1}S_3S_2 S_0 S_2^{-1}S_3S_2
        \iff S_0S_2^{-1}S_3S_2S_0= \underline{S_2^{-1}S_3S_2 S_0 S_3S_2} S_3^{-1}\\
        &\iff S_0S_2^{-1}S_3S_2S_0= \underline{ S_0 S_3S_2S_0S_3S_0^{-1}} S_3^{-1}
        \iff S_3S_2S_0 S_3 S_0=S_2 S_3S_2S_0S_3\\
        &\iff  \hcancel{S_3S_2}S_0 S_3 S_0=\hcancel{S_3S_2} S_3S_0S_3
        \iff S_0 S_3 S_0=S_3S_0S_3,
    \end{align*}
    where the underlined relation follows from the double edge relation at vertex $3$ in $Q$.
The last two relations left to prove correspond to the $4$-cycle around the cone point in $Q$. The relation $t_1t_0t_3t_1=t_0t_3t_1t_0=t_3t_1t_0t_3$ implies that
\begin{align*}
S_1S_0S_2^{-1}S_3S_2S_1=S_0S_2^{-1}S_3S_2S_1S_0=S_2^{-1}S_3S_2S_1S_0S_2^{-1}S_3S_2.
\end{align*}

Using the relations already proved, we see that the equality of the first two expressions above is true if and only if
\begin{align*}
    S_1S_0S_3S_2S_3^{-1}S_1=S_0S_3S_2S_3^{-1}S_1S_0
    &\iff S_1S_0S_3S_2S_1S_3^{-1}=S_0S_3S_2S_1S_3^{-1}S_0\\
    &\iff S_1S_0S_3S_2S_1=S_0S_3S_2S_1S_0S_3S_0^{-1}\\
    &\iff S_1S_0S_3S_2S_1S_0=S_0S_3S_2S_1S_0S_3.
\end{align*}

Similarly, the equality of the first and third expression above is true if and only if
\begin{align*}
    S_2S_1S_0S_3S_2S_3^{-1}S_1=S_3S_2S_1S_0S_3S_2S_3^{-1}
    &\iff S_2S_1S_0S_3S_2S_1S_3^{-1}=S_3S_2S_1S_0S_3S_2S_3^{-1}\\
    &\iff S_2S_1S_0S_3S_2S_1=S_3S_2S_1S_0S_3S_2.
\end{align*}
Hence all the defining relations for $Q$ hold, as required.
\end{proof}

\begin{remark}
    \label{rem:flippedtags}
Flipping all of the tags in Figure~\ref{fig:cornercase}, Figure~\ref{fig:mutation_tagged} or~\ref{fig:mutation_untagged} does not affect the corresponding quivers or groups, so it follows that Lemmas~\ref{lem:cornercasen3}, \ref{lem:cornercasen3back},
\ref{lem:cornercasen2},
\ref{lem:cornercasen2back},
\ref{lem:mutation_tagged}, \ref{lem:mutation_tagged_reverse}, \ref{lem:mutation_untagged} and \ref{lem:mutation_untagged_reverse} hold for these cases too.
\end{remark}

\begin{proof}[Proof of Theorem~\ref{thm:mutation_invariance}]
Let $T$ be a tagged triangulation of $S$, and $Q_T$ the associated quiver as in Definition~\ref{defn_quiver_on_disc}. Let $G_{Q_T}$ be the associated group as in Definition~\ref{defn_GQ_associated_group}, with generators $s_i$.
Let $\alpha$ be an arc in $T$ and $k$ the corresponding vertex of $Q$. Let $T'$ be the tagged triangulation obtained by flipping $T$ at $\alpha$, so that
$Q_{T'}=\mu_k(Q)$ is the mutation of $Q$ at $k$ as in Definition~\ref{defn_mutationrules}, by Lemma~\ref{lem:flipmutation}(b).
By Lemma~\ref{lem:flipmutation}(a), the flip of $\alpha$ is given locally by one of the mutations in Figures~\ref{fig:typeAmutation},
\ref{fig:cornercase},~\ref{fig:mutationnearcone},~\ref{fig:mutation_tagged} or~\ref{fig:mutation_untagged} (from left to right or right to left), or by one of the mutations
from Figure~\ref{fig:cornercase}, \ref{fig:mutation_tagged} or \ref{fig:mutation_untagged} with all of the tags flipped.

It follows from Lemmas~\ref{lem:mutationtypeA}, \ref{lem:cornercasen3}, \ref{lem:cornercasen3back}, \ref{lem:cornercasen2},
\ref{lem:cornercasen2back}, \ref{lem:mutationnearcone}, \ref{lem:mutationnearcone_otherway}, \ref{lem:mutation_tagged}, \ref{lem:mutation_tagged_reverse}, \ref{lem:mutation_untagged}, \ref{lem:mutation_untagged_reverse} and Remark~\ref{rem:flippedtags} that there are group homomorphisms $\varphi_k^Q:G_Q\rightarrow G_{Q'}$ and $\varphi_k^{Q'}:G_{Q'}\rightarrow G_Q$.

Note that arcs appearing on the boundary in each figure could be on the actual boundary of $S$.
Each defining relation in $G_Q$ to be checked involves a certain collection of unmutated vertices, plus possibly the mutated vertex. It is easy to check that, in each case, the proof that this relation holds involves relations involving only the same collection of vertices.
If we consider the same situation where one of the vertices corresponds to an arc on the boundary of $S$, the corresponding relation does not appear and therefore does not need to be shown. It follows that the corresponding results hold in the situation where some or all of the dashed diagonals on the boundary of the figure are on the boundary of the disk.

It is easy to check in each case that $\varphi^{Q'}_k\varphi^Q_k(s_i)=s_k^{-1}s_is_k$ for all $i$, so $\varphi^{Q'}_k\varphi^Q_k$ is an isomorphism. By the same argument,
the other composition
$\varphi^Q_k\varphi^{Q'}_k$ is also an isomorphism,
and hence so is $\varphi^Q_k$.
There is a sequence of flips connecting any two tagged triangulations of $S$ by~\cite[Prop.\ 7.10]{FST}, and the result follows.
\end{proof}

\begin{theorem}
\label{thm:Bddnpresentation}
Let $T$ be any tagged triangulation of $(X,M)$, and $Q_T$ the associated quiver. Then $G_{Q_T}$ is isomorphic to the braid group $B(d,d,n)$, and thus gives a presentation for $B(d,d,n)$.
\end{theorem}
\begin{proof}
There is a tagged triangulation of $(X,M)$ for which the corresponding group $G_Q$ is the presentation of $B(d,d,n)$ from~\cite[Thm.\ 2.27]{BMR98} (see Figure~\ref{fig:BMR_embedded}). The result follows from combining this with Theorem~\ref{thm:mutation_invariance}.
\end{proof}

\section{Geometric interpretation of the new presentations}

We work with the same surface $(X,M)$ as in the previous section: $X$ is the disk $S$ with an interior marked point, interpreted as a cone point of degree $d\geq 2$, and $M$ a set of $n\geq 2$ marked points on the boundary of $X$.
In Section~\ref{subsection_associating_quiver}, we defined a way to associate a quiver $Q_T$ to any tagged triangulation $T$ of $(X,M)$ and a group $G_{Q_T}$.
As in \cite[Defn~3.1]{GM}, we associate another graph to $T$ as follows.

\begin{defn} \label{def:braidgraph}
Let $T$ be a tagged triangulation of $(X,M)$. We define the \textit{braid graph} $D_T$ of $T$ to be the geometric dual of $T$ regarded as a graph embedded in the disk.
Thus $D_T$ has a vertex in each connected component of the complement of $T$ and, whenever two connected components share an edge of $T$, there is a corresponding edge in $D_T$ between the two corresponding vertices. Note that, in general, $D_T$ can have multiple edges between vertices.

If we regard $T$ as a graph embedded in the plane, then $D_T$ is the geometric dual in the plane with the vertex corresponding to the external face removed. Note that this geometric dual is isomorphic to the combinatorial dual of $T$ by~\cite[Theorem 30]{whitney32} (see also~\cite[\S11, page 115]{harary69}), since $T$ is a non-separable graph, so $D_T$ also is well-defined as an abstract graph.
\end{defn}

Moreover, note that the interior of $X$ is isomorphic as an orbifold to $\mathcal{O}_d:=\mathbb{C}/C_d$ and hence we identify the two orbifolds in our arguments.

Given a set of $n$ vertices $V$ in $\mathcal{O}_d^\circ$, that is $\mathcal{O}_d$ minus the cone point, one can define the corresponding \textit{braid group} $Z_n(\mathcal{O}_d)$ (denoted $\Gamma (\mathcal{O}_d, V)$ in \cite{GM}). Each element of $Z_n(\mathcal{O}_d)$, also called \textit{braid},  can be regarded as a permutation $g$ of $V$ together with a tuple $\gamma=(\gamma_v)_{v\in V}$ of paths $\gamma_v:[0,1]\rightarrow \mathcal{O}^\circ_d$ with $\gamma_v(0)=v$ and $\gamma_v(1)=g(v)$ for each $v\in V$ and, for each $t\in[0,1]$, the points $\gamma_v(t)$ for $v\in V$ are all distinct for all $v\in V$. See \cite{A} and \cite[Section~3]{GM} for further details. Moreover, each path $\pi$ in $\mathcal{O}^\circ_d$ with endpoints in $V$ determines a braid $\sigma_\pi$ in  $Z_n(\mathcal{O}_d)$ (see e.g.~\cite[Defn~3.3]{GM}).

Let $T$ be a tagged triangulation of $(X,M)$, and note that this is a collection of $n$ (tagged) arcs. Each (tagged) arc $\alpha_i$ in $T$ corresponds to a vertex $i$ in the quiver $Q_T$ and to an edge $\pi_i$ in the braid graph $D_T$. Following the same notation as \cite{GM}, we let $\sigma_i:=\sigma_{\pi_i}$ denote the corresponding braid in $Z_n(\mathcal{O}_d)$ and $B_T$ be the subgroup of $Z_n(\mathcal{O}_d)$ generated by the braids $\sigma_i$, for $i$ vertex in $Q_T$.

\begin{proposition}\label{prop_initial_triang}
    Let $T_0$ be the triangulation of $(X,M)$ shown in Figure~\ref{fig:BMR_embedded}. Then there is an isomorphism from $B_{T_0}$ to $G_{Q_{T_0}}$ taking the braid $\sigma_i$ to the generator $s_i$ of $G_{Q_{T_0}}$. Furthermore, the subgroup $B_{T_0}$ is a subgroup of $Z_n(\mathcal{O}_d)$ of index $d$.
\end{proposition}

\begin{proof}
    Note that, via an isomorphism of the kind in \cite[Remark~3.2]{GM}, the element $\sigma_i$ coincides with $h_i$ for $1\leq i\leq n$, for the braids $h_i$ in $Z_n(\mathcal{O}_d)$ illustrated in Figure~\ref{fig:braids_hi}. Since by construction $G_{Q_{T_0}}=B(d,d,n)$ with the presentation from \cite{BMR98}, the monomorphism $\beta$ from Theorem~\ref{mainthmA} (which sends $s_i$ to $h_i$) gives the embedding of $B_{T_0}$ in $Z_n(\mathcal{O}_d)$ as a subgroup of index $d$.
\end{proof}

Before stating and proving our final result, we recall a result from \cite[Théorème, part(iv)]{Sergiescu93}, see also \cite[Lemma 3.5]{GM}.

\begin{lemma}\label{lemma_sergiescu}
    Let $A,\,B,\, C$ be three distinct points in $\mathcal{O}^\circ_d$ and suppose there is a topological disk in $\mathcal{O}^\circ_d$ with $A,\,B,\, C$ lying in order anticlockwise around its boundary. Let $AB$ denote the arc on its boundary between $A$ and $B$, and define $BC$ and $CA$ similarly. Then $\sigma_{AB}\sigma_{BC}=\sigma_{BC}\sigma_{CA}$.
\end{lemma}

\begin{defn}
Let $T$ be a tagged triangulation of $(X,M)$.
Suppose that there is an arrow $i\rightarrow k$ in $Q_T$.
Then there are vertices $X_i$, $Y_i$ and $Z_i$ and edges $X_iY_i$ and $Y_iZ_i$ in $D_T$ such that $\sigma_i=\sigma_{X_iY_i}$ and $\sigma_k=\sigma_{Y_iZ_i}$.
We say that an embedding of $D_T$ into $(X,M)$ is \textit{good} at $k$ if, for every such $i$, the vertices $X_i,Y_i$ and $Z_i$ are in anticlockwise order.
\end{defn}

\begin{remark}
Note that our convention for orienting the arrows of the quiver of a triangulation is opposite to the convention used in \cite{GM}. However, the proof of~\cite[Thm.\ 3.6]{GM} actually requires this opposite convention: then it goes through as stated (provided we regard $k$ as the mutation vertex). For example, in the notation used there (see~\cite[Fig.\ 6]{GM}), we have
$\sigma_{\widetilde{\pi}_2}=\sigma_1^{-1}\sigma_2\sigma_1$, while $\tilde{\tau}_2=\sigma_2$ as there is no arrow from $2$ to $1$ in the quiver used there. Thus the claim that $\sigma_{\widetilde{\pi}_2}=\tilde{\tau}_2$ does not hold, but this is resolved if the opposite convention for orienting the quiver is adopted, as we do here.
\end{remark}

\begin{theorem}\label{thm_braid_interpretation}
    Let $T$ be a tagged triangulation of $(X,M)$. Then there is an isomorphism from $B_T$ to $G_{Q_{T}}$ taking the braid $\sigma_i$ to the generator $s_i$ of $G_{Q_{T}}$. Furthermore, $B_{T}$ is a subgroup of index $d$ of $Z_n(\mathcal{O}_d)$.
\end{theorem}

\begin{proof}
By Proposition~\ref{prop_initial_triang}, the result holds for the triangulation $T=T_0$. Since any tagged triangulation can be obtained by flipping $T_0$ a finite number of times, it is enough to show that if the theorem holds for a tagged triangulation $T$, then it also holds for the flip of $T$ at any of its tagged arcs.

Assume that the result holds for a tagged triangulation $T$, that is there is an isomorphism $\psi_T: B_T\rightarrow G_{Q_{T}}$ sending the braid $\sigma_i$ to the generator $s_i$. Let the triangulation $T'$ be obtained by flipping $T$ at the arc $\alpha_k$ and the generators of $B_{T'}$ and $G_{Q_{T'}}$ be denoted by $\tau_i$ and $t_i$ respectively. 
In the following arguments, we label the paths in $\mathcal{O}^\circ_d$ as the corresponding braids in $B_{T}$ and $B_{T'}$, to avoid heavy notation.

Define 
$\widetilde{\tau_i}=\sigma_k^{-1}\sigma_i\sigma_k$ if
$i\rightarrow k$ in $Q_T$ or the vertices $i$ and $k$ correspond to the only two arcs incident with the conepoint in $T$
and the arc corresponding to $k$ is rotated anti-clockwise to the flipped arc; and $\widetilde{\tau_i}=\sigma_i$ otherwise.
Note that $B_T$ is generated by the $\widetilde{\tau_i}$, for $i$ running through the vertices of $Q_T$.

The possible types of flips that can occur are given locally by one of the mutations in Figures~\ref{fig:typeAmutation},
\ref{fig:cornercase},~\ref{fig:mutationnearcone},~\ref{fig:mutation_tagged} or~\ref{fig:mutation_untagged} (from left to right or right to left), or by one of the mutations
from Figure~\ref{fig:cornercase}, \ref{fig:mutation_tagged} or \ref{fig:mutation_untagged} with all of the tags flipped.

\begin{figure}
\begin{tikzpicture}[scale=1.3,
  quiverarrow/.style={black,-latex,shorten <=-4pt, shorten >=-4pt},
  mutationarc/.style={dashed, red, very thick, shorten <=-4pt, shorten >=-4pt},
  mutationarccone/.style={dashed, red, very thick, shorten >=-4pt},
  arc/.style={dashed, black, shorten <=-4pt, shorten >=-4pt},
  arccone/.style={dashed, black, shorten >=-4pt},
  line/.style={black, thick, shorten <=-8pt, shorten >=-8pt},
  point/.style={gray},
  vertex/.style={black},
  conepoint/.style={gray, circle, draw=gray!100, fill=white!100, thick, inner sep=1.5pt},
  db/.style={thick, double, double distance=1.3pt, shorten <=-6pt}
  ]
   \begin{scope}
   \node[point] (p1) at (0,0) {$\bullet$};
   \node[point] (p2) at (2,0) {$\bullet$};
   \node[point] (p3) at (0,2) {$\bullet$};
   \node[point] (p4) at (2,2) {$\bullet$};
  \draw[arccone] (p1)--(p2);
  \draw[arccone] (p1)--(p3);
  \draw[arccone] (p2)--(p4);
  \draw[arccone] (p3)--(p4);
  \draw[mutationarccone] (p1)--(p4);
  \draw[gray] (0,0) node{$\bullet$};
  \draw[gray] (2,0) node{$\bullet$};
  \draw[gray] (0,2) node{$\bullet$};
  \draw[gray] (2,2) node{$\bullet$};
  \node[vertex, label=
{[label distance=-5pt]-90:{\tiny $A$}}] (A) at (0.5,1) {$\bullet$};
\node[vertex, label=
{[label distance=-10pt]70:{\tiny $B$}}] (B) at (1.6,1.2) {$\bullet$};
\node[vertex, label=
{[label distance=-10pt]60:{\tiny $C$}}] (C) at (1,2.2) {$\bullet$};
\node[vertex, label=
{[label distance=-5pt]90:{\tiny $D$}}] (D) at (-0.2,1.7) {$\bullet$};
\node[vertex, label=
{[label distance=-10pt]-60:{\tiny $E$}}] (E) at (0.8,-0.2) {$\bullet$};
\node[vertex, label=
{[label distance=-5pt]90:{\tiny $F$}}] (F) at (2.2,0.7) {$\bullet$};
\draw[line] (A)--(B);
\draw[line] (A)--(D);
\draw[line] (A)--(C);
\draw[line] (E)--(B);
\draw[line] (F)--(B);
 \node[below] at  ($(A)! 0.5!(D)$) {\tiny $\sigma_1$};
 \node[left] at  ($(A)! 0.5!(C)$) {\tiny $\sigma_4$};
 \node[above] at  ($(A)! 0.5!(B)$) {\tiny $\sigma_5$};
 \node[below] at  ($(B)! 0.5!(F)$) {\tiny $\sigma_3$};
 \node[left] at  ($(B)! 0.5!(E)$) {\tiny $\sigma_2$};
  \begin{scope}[xshift=0.3cm, yshift=-2.3cm, scale=0.7]
  \node[vertex, label=
{[label distance=-5pt]180:{\tiny $1$}}] (v1) at (0,1) {$\bullet$};
\node[vertex, label=
{[label distance=-5pt]-90:{\tiny $2$}}] (v2) at (1,0) {$\bullet$};
\node[vertex, label=
{[label distance=-5pt]0:{\tiny $3$}}] (v3) at (2,1) {$\bullet$};
\node[vertex, label=
{[label distance=-5pt]90:{\tiny $4$}}] (v4) at (1,2) {$\bullet$};
\node[vertex, red, label=
{[label distance=-5pt]120:{\tiny $5$}}] (v5) at (1,1) {$\bullet$};
  \draw[quiverarrow] (v1)--(v5);
  \draw[quiverarrow] (v5)--(v4);
  \draw[quiverarrow] (v4)--(v1);
  \draw[quiverarrow] (v3)--(v5);
  \draw[quiverarrow] (v5)--(v2);
  \draw[quiverarrow] (v2)--(v3);
  \end{scope}
\draw [->] (2.5,1) to (3.5,1);
 \end{scope}
  \begin{scope}[xshift=4cm]
   \node[point] (p1) at (0,0) {$\bullet$};
   \node[point] (p2) at (2,0) {$\bullet$};
   \node[point] (p3) at (0,2) {$\bullet$};
   \node[point] (p4) at (2,2) {$\bullet$};
  \draw[arccone] (p1)--(p2);
  \draw[arccone] (p1)--(p3);
  \draw[arccone] (p2)--(p4);
  \draw[arccone] (p3)--(p4);
  \draw[mutationarccone] (p1)--(p4);
  \draw[gray] (0,0) node{$\bullet$};
  \draw[gray] (2,0) node{$\bullet$};
  \draw[gray] (0,2) node{$\bullet$};
  \draw[gray] (2,2) node{$\bullet$};
\draw [->] (2.5,1) to (3.5,1);
 \node[vertex, label=
{[label distance=-5pt]-90:{\tiny $A$}}] (A) at (0.5,1) {$\bullet$};
\node[vertex, label=
{[label distance=-10pt]70:{\tiny $B$}}] (B) at (1.6,1.2) {$\bullet$};
\node[vertex, label=
{[label distance=-10pt]60:{\tiny $C$}}] (C) at (1,2.2) {$\bullet$};
\node[vertex, label=
{[label distance=-5pt]90:{\tiny $D$}}] (D) at (-0.2,1.7) {$\bullet$};
\node[vertex, label=
{[label distance=-10pt]-60:{\tiny $E$}}] (E) at (0.8,-0.2) {$\bullet$};
\node[vertex, label=
{[label distance=-5pt]90:{\tiny $F$}}] (F) at (2.2,0.7) {$\bullet$};
\draw[line] (A)--(B);
\draw[line] (B)--(D);
\draw[line] (A)--(C);
\draw[line] (E)--(B);
\draw[line] (F)--(A);
 \node[above left] at  ($(B)! 0.6!(D)$) {\tiny $\widetilde{\tau_1}$};
 \node[right] at  ($(A)! 0.6!(C)$) {\tiny $\widetilde{\tau_4}$};
 \node[above] at  ($(A)! 0.35!(B)$) {\tiny $\widetilde{\tau_5}$};
 \node[below right] at  ($(A)! 0.6!(F)$) {\tiny $\widetilde{\tau_3}$};
 \node[left] at  ($(B)! 0.5!(E)$) {\tiny $\widetilde{\tau_2}$};
 \end{scope}
  \begin{scope}[xshift=8cm]
  \node[point] (p1) at (0,0) {$\bullet$};
   \node[point] (p2) at (2,0) {$\bullet$};
   \node[point] (p3) at (0,2) {$\bullet$};
   \node[point] (p4) at (2,2) {$\bullet$};
  \draw[arccone] (p1)--(p2);
  \draw[arccone] (p1)--(p3);
  \draw[arccone] (p2)--(p4);
  \draw[arccone] (p3)--(p4);
  \draw[mutationarccone] (p2)--(p3);
  \draw[gray] (0,0) node{$\bullet$};
  \draw[gray] (2,0) node{$\bullet$};
  \draw[gray] (0,2) node{$\bullet$};
  \draw[gray] (2,2) node{$\bullet$};
    \node[vertex, label=
{[label distance=-5pt]90:{\tiny $A$}}] (A) at (0.6,1.6) {$\bullet$};
\node[vertex, label=
{[label distance=-10pt]-30:{\tiny $B$}}] (B) at (1,0.6) {$\bullet$};
\node[vertex, label=
{[label distance=-10pt]60:{\tiny $C$}}] (C) at (1,2.2) {$\bullet$};
\node[vertex, label=
{[label distance=-5pt]90:{\tiny $D$}}] (D) at (-0.2,1.7) {$\bullet$};
\node[vertex, label=
{[label distance=-10pt]-60:{\tiny $E$}}] (E) at (0.8,-0.2) {$\bullet$};
\node[vertex, label=
{[label distance=-5pt]90:{\tiny $F$}}] (F) at (2.2,0.7) {$\bullet$};
\draw[line] (A)--(B);
\draw[line] (B)--(D);
\draw[line] (A)--(C);
\draw[line] (E)--(B);
\draw[line] (F)--(A);
 \node[below] at  ($(B)! 0.5!(D)$) {\tiny $\tau_1$};
 \node[right] at  ($(A)! 0.5!(C)$) {\tiny $\tau_4$};
 \node[right] at  ($(A)! 0.5!(B)$) {\tiny $\tau_5$};
 \node[above] at  ($(A)! 0.5!(F)$) {\tiny $\tau_3$};
 \node[right] at  ($(B)! 0.5!(E)$) {\tiny $\tau_2$};
  \begin{scope}[xshift=0.3cm, yshift=-2.3cm, scale=0.7]
  \node[vertex, label=
{[label distance=-5pt]180:{\tiny $1$}}] (v1) at (0,1) {$\bullet$};
\node[vertex, label=
{[label distance=-5pt]-90:{\tiny $2$}}] (v2) at (1,0) {$\bullet$};
\node[vertex, label=
{[label distance=-5pt]0:{\tiny $3$}}] (v3) at (2,1) {$\bullet$};
\node[vertex, label=
{[label distance=-5pt]90:{\tiny $4$}}] (v4) at (1,2) {$\bullet$};
\node[vertex, red, label=
{[label distance=-5pt]60:{\tiny $5$}}] (v5) at (1,1) {$\bullet$};
  \draw[quiverarrow] (v5)--(v1);
  \draw[quiverarrow] (v4)--(v5);
  \draw[quiverarrow] (v1)--(v2);
  \draw[quiverarrow] (v5)--(v3);
  \draw[quiverarrow] (v2)--(v5);
  \draw[quiverarrow] (v3)--(v4);
  \end{scope}
 \end{scope}
  \end{tikzpicture}
  \caption{Flip of an arc: type $A$ situation.}
    \label{fig:typeA_flip_braid}
  \end{figure}
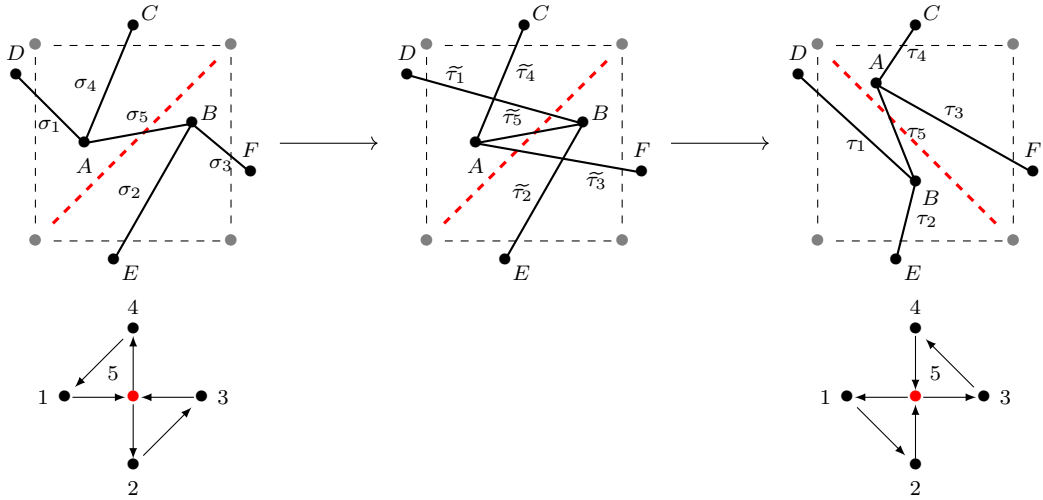

Consider first the flip in Figure~\ref{fig:typeAmutation}, that is a flip far from the cone point. The left hand side of Figure~\ref{fig:typeA_flip_braid} shows an embedding of  $D_T$ which is good at $5$. Applying Lemma~\ref{lemma_sergiescu}, the middle figure shows the paths corresponding to the braids $\widetilde{\tau_i}$. Rotating the vertices $A$ and $B$ clockwise, we get the diagram on the right of Figure~\ref{fig:typeA_flip_braid}, where we used the isomorphism  $\phi:B_{T'}\rightarrow B_T$ given by $\tau_i\mapsto\widetilde{\tau_i}$ from \cite[Remark~3.2]{GM}. Consider the composition of isomorphisms $\varphi_5\circ\psi_T\circ\phi$, where $\varphi_5$ is the isomorphism from Theorem~\ref{thm:mutation_invariance}. If $i\rightarrow 5$ in $Q_T$, then
\begin{align*}
  \varphi_5\circ\psi_T\circ\phi(\tau_i)=  \varphi_5\circ\psi_T(\sigma_5^{-1}\sigma_i\sigma_5)=
  \varphi_5(s_5^{-1}s_is_5)=t_5^{-1}t_5t_it_5^{-1}t_5=t_i,
\end{align*}
and otherwise
\begin{align*}
    \varphi_5\circ\psi_T\circ\phi(\tau_i)=\varphi_5\circ\psi_T(\sigma_i)=
    \varphi_5(s_i)=t_i.
\end{align*}
Hence we obtain an isomorphism $B_{T'}\rightarrow G_{Q_T'}$ sending $\tau_i$ to $t_i$ as required.
A symmetric argument works if we start with the triangulation on the right, with the shown embedding of the dual graph good at $5$, and flip it to the triangulation on the left of Figure~\ref{fig:typeA_flip_braid}.

For the flips in Figures~\ref{fig:cornercase} and~\ref{fig:mutationnearcone} (from left to right or right to left or with all the tags flipped), the result follows by arguments very similar to the above case, hence we omit the details. Since flipping the tags does not affect the arguments, it only remains to show that the result holds for the flips in Figures~\ref{fig:mutation_tagged} and~\ref{fig:mutation_untagged}.

\begin{figure}[ht]\scalebox{1}{
    \centering
 \begin{tikzpicture}[scale=2,
  quiverarrow/.style={black, -latex},
  mutationarc/.style={dashed, red, very thick},
  arc/.style={dashed, black},
  point/.style={gray},
  vertex/.style={black},
  line/.style={black, thick, shorten <=-8pt, shorten >=-8pt},
  conepoint/.style={gray, circle, draw=gray!100, fill=white!100, thick, inner sep=1.5pt},
  db/.style={thick, double, double distance=1.3pt, shorten <=-6pt}
  ]
 \begin{scope}
 \draw[dashed] (2,2) arc
	[start angle=90,
		end angle=180,
		x radius=2cm,
		y radius =2cm
	] ;
    \draw[dashed] (0,0) arc
	[start angle=-90,
		end angle=0,
		x radius=2cm,
		y radius =2cm
	] ;
    \draw[dashed] (2,2) arc
	[start angle=100,
		end angle=150,
		x radius=1.45cm,
		y radius =2cm
	] ;
 \node[point] (p0) at (0,0) {$\bullet$};
 \node[point] (p1) at (2,2) {$\bullet$};
 \node[conepoint] (p3) at (1,1) {\scriptsize $d$};
  \draw [->] (2.2,1.2) to (2.7,1.2);
\draw[dashed,red, thick, shorten <=3pt, shorten >=-1pt] (p3) .. controls +(10:0.6) and +(250:0.6) .. coordinate[pos=0.12](t1) coordinate[midway](w2) (p1);
\node[vertex,rotate=-63, red](tt1) at (t1) {$\bowtie$};
\node[vertex, label=
{[label distance=-5pt]180:{\tiny $A$}}] (A) at (1.3,1.2) {$\bullet$};
\node[vertex, label=
{[label distance=-10pt]80:{\tiny $B$}}] (B) at (1.8,1.3) {$\bullet$};
\node[vertex, label=
{[label distance=-5pt]90:{\tiny $C$}}] (C) at (-0.1,1.2) {$\bullet$};
\node[vertex, label=
{[label distance=-10pt]70:{\tiny $D$}}] (D) at (1.5,0.3) {$\bullet$};
\draw[line] (A)--(B);
\draw[line] (B)--(D);
\draw[thick] (-0.1,1.2) arc
	[start angle=-150,
		end angle=-25,
		x radius=1.06cm,
		y radius =1.3cm
	] ;
\draw[thick] (1.3,1.2) arc
	[start angle=30,
		end angle=300,
		x radius=0.3cm,
		y radius =0.3cm
	] ;
\draw[line] (1.28,0.87)--(B);
   \node[above] at  ($(A)! 0.5!(B)$) {\tiny $\sigma_2$};
  \node[right] at  ($(B)! 0.6!(D)$) {\tiny $\sigma_1$};
  \node[below] at  ($(A)! 0.5!(C)$) {\tiny $\sigma_0$};
  \node[left] at  (0.5,0.5) {\tiny $\sigma_3$};
\begin{scope}[yshift=-2cm]
 \node[vertex, label=
{[label distance=-22pt]-90:{\tiny $3$}}] (v3) at (0.55,1.35) {$\bullet$};
 \node[vertex, label=
{[label distance=-20pt]90:{\tiny $1$}}] (v1) at (1.25,0.45) {$\bullet$};
 \node[vertex, label=
{[label distance=-10pt]175:{\tiny $0$}}] (v0) at (1.38,1.63) {$\bullet$};
 \node[vertex, red, label=
{[label distance=-7pt]270:{\tiny $2$}}] (v2) at (1.64,1.32) {$\bullet$};
  \draw[quiverarrow] (v0)--(v3);
  \draw[quiverarrow] (v2)--(v3);
  \draw[quiverarrow] (v1)--(v0);
  \draw[quiverarrow] (v1)--(v2);
  \draw [-, shorten <=-2pt, shorten >=-2pt] (v2) to node[above right,xshift=-0.1cm] {\tiny $d$} (v0);
  \draw[quiverarrow] (0.5,1.25) arc
	[start angle=170,
		end angle=230,
		x radius=1.7cm,
		y radius =0.88cm
	] ;
 \draw[thick, double, double distance=1.3pt] (0.6,1.29)--(0.8,1.1);
 \draw[thick, double, double distance=1.3pt] (1.2,0.55)--(1.1,0.8);
\end{scope}
\end{scope}
 \begin{scope}[xshift=3cm]
 \draw[dashed] (2,2) arc
	[start angle=90,
		end angle=180,
		x radius=2cm,
		y radius =2cm
	] ;
    \draw[dashed] (0,0) arc
	[start angle=-90,
		end angle=0,
		x radius=2cm,
		y radius =2cm
	] ;
    \draw[dashed] (2,2) arc
	[start angle=100,
		end angle=150,
		x radius=1.45cm,
		y radius =2cm
	] ;
 \node[point] (p0) at (0,0) {$\bullet$};
 \node[point] (p1) at (2,2) {$\bullet$};
 \node[conepoint] (p3) at (1,1) {\scriptsize $d$};

\draw[dashed,red, thick, shorten <=3pt, shorten >=-1pt] (p3) .. controls +(10:0.6) and +(250:0.6) .. coordinate[pos=0.12](t1) coordinate[midway](w2) (p1);
\node[vertex,rotate=-63, red](tt1) at (t1) {$\bowtie$};
\node[vertex, label=
{[label distance=-5pt]180:{\tiny $A$}}] (A) at (1.3,1.2) {$\bullet$};
\node[vertex, label=
{[label distance=-10pt]80:{\tiny $B$}}] (B) at (1.8,1.3) {$\bullet$};
\node[vertex, label=
{[label distance=-5pt]90:{\tiny $C$}}] (C) at (-0.1,1.2) {$\bullet$};
\node[vertex, label=
{[label distance=-10pt]70:{\tiny $D$}}] (D) at (1.5,0.3) {$\bullet$};
\draw[line] (A)--(B);
\draw[line] (A)--(D);
\draw[thick] (-0.1,1.2) arc
	[start angle=-150,
		end angle=-25,
		x radius=1.06cm,
		y radius =1.3cm
	] ;
\draw[thick] (1.3,1.2) arc
	[start angle=30,
		end angle=300,
		x radius=0.3cm,
		y radius =0.3cm
	] ;
\draw[line] (1.28,0.87)--(B);
   \node[above] at  ($(A)! 0.5!(B)$) {\tiny $\widetilde{\tau_2}$};
  \node[right] at  ($(A)! 0.6!(D)$) {\tiny $\widetilde{\tau_1}$};
  \node[below] at  ($(A)! 0.5!(C)$) {\tiny $\widetilde{\tau_0}$};
  \node[left] at  (0.5,0.5) {\tiny $\widetilde{\tau_3}$};
  \draw [->] (2.2,1.2) to (2.7,1.2);
 \end{scope}
  \begin{scope}[xshift=6cm]
  \draw[dashed] (2,2) arc
	[start angle=90,
		end angle=180,
		x radius=2cm,
		y radius =2cm
	] ;
    \draw[dashed] (0,0) arc
	[start angle=-90,
		end angle=0,
		x radius=2cm,
		y radius =2cm
	] ;
    \draw[dashed] (2,2) arc
	[start angle=100,
		end angle=150,
		x radius=1.45cm,
		y radius =2cm
	] ;
 \draw[dashed, red, thick] (0,0) arc
	[start angle=-80,
		end angle=-30,
		x radius=1.45cm,
		y radius =2cm
	] ;
  \node[point] (p0) at (0,0) {$\bullet$};
 \node[point] (p1) at (2,2) {$\bullet$};
 \node[conepoint] (p3) at (1,1) {\scriptsize $d$};
\node[vertex, label=
{[label distance=-5pt]90:{\tiny $C$}}] (C) at (-0.1,1.2) {$\bullet$};
\node[vertex, label=
{[label distance=-10pt]70:{\tiny $D$}}] (D) at (1.5,0.3) {$\bullet$};
\node[vertex, label=
{[label distance=-10pt]80:{\tiny $A$}}] (A) at (1.4,1.2) {$\bullet$};
\node[vertex, label=
{[label distance=-10pt]200:{\tiny $B$}}] (B) at (0.6,0.8) {$\bullet$};
\draw[line] (C)--(B);
\draw[line] (A)--(D);
\draw[thick] (1.4,1.2) arc
	[start angle=40,
		end angle=200,
		x radius=0.48cm,
		y radius =0.4cm
	] ;
  \draw[thick] (0.6,0.8) arc
	[start angle=240,
		end angle=370,
		x radius=0.53cm,
		y radius =0.35cm
	] ;
     \node[left] at  ($(C)! 0.5!(B)$) {\tiny $\tau_3$};
  \node[right] at  ($(A)! 0.6!(D)$) {\tiny $\tau_1$};
  \node[below] at  (1,1.55) {\tiny $\tau_0$};
  \node[left] at  (1.2,0.7) {\tiny $\tau_2$};
 \begin{scope}[yshift=-2cm]
 \node[vertex, label=
{[label distance=-22pt]-90:{\tiny $3$}}] (v3) at (0.55,1.35) {$\bullet$};
 \node[vertex, label=
{[label distance=-20pt]90:{\tiny $1$}}] (v1) at (1.25,0.45) {$\bullet$};
 \node[vertex, label=
{[label distance=-10pt]330:{\tiny $0$}}] (v0) at (1.38,1.63) {$\bullet$};
 \node[vertex, red, label=
{[label distance=-5pt]180:{\tiny $2$}}] (v2) at (0.6,0.35) {$\bullet$};
  \draw[quiverarrow] (v0)--(v3);
  \draw[quiverarrow] (v1)--(v0);
  \draw[quiverarrow] (0.7,0.35)--(1.15,0.45);
  \draw[quiverarrow] (v0)--(v3);
  \draw (0.9,1.1) node{$\scriptstyle d$};
  \draw (v0)--(v2);
 \draw[quiverarrow] (0.5,1.25) arc
	[start angle=170,
		end angle=198,
		x radius=1.4cm,
		y radius =1.7cm
	] ;
 \draw[thick, double, double distance=1.3pt] (0.6,1.29)--(0.8,1.1);
 \draw[thick, double, double distance=1.3pt] (1.2,0.55)--(1.1,0.8);
 \end{scope}
 \end{scope}
 \end{tikzpicture}}
 \caption{Mutation of a tagged to an untagged arc or viceversa, case $1$.}
\label{fig:flip_braid_case1}
\end{figure}
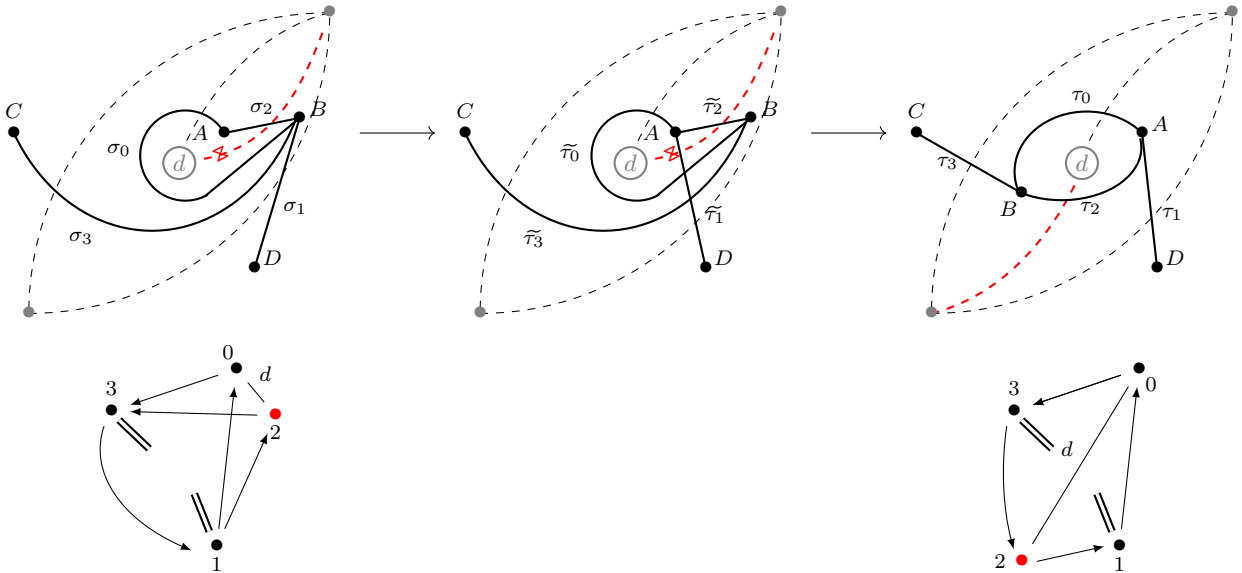

Consider the flip in Figure~\ref{fig:mutation_tagged} with $T$ on the left. 
The left hand side of Figure~\ref{fig:flip_braid_case1} shows an embedding of $D_T$ which is good at $2$.
Applying Lemma~\ref{lemma_sergiescu}, the middle figure shows the paths corresponding to the braids $\widetilde{\tau_i}$. Note that, following the definition of $\widetilde{\tau_i}$, the only conjugated element is $\sigma_1$ in this case as the mutated arc $\alpha_2$ is not rotated anti-clockwise to the flipped arc and so $\sigma_0$ is not conjugated.
Rotating the vertices $A$ and $B$ clockwise, we get the diagram on the right of Figure~\ref{fig:flip_braid_case1}, where we used the isomorphism  $B_{T'}\rightarrow B_T$ given by $\tau_i\mapsto\widetilde{\tau_i}$ from \cite[Remark~3.2]{GM}. Composing this with the isomorphism $\varphi_2\circ\psi_T$, where $\varphi_2$ is the isomorphism from Theorem~\ref{thm:mutation_invariance}, we obtain an isomorphism $B_{T'}\rightarrow G_{Q_T'}$ sending $\tau_i$ to $t_i$ as required.

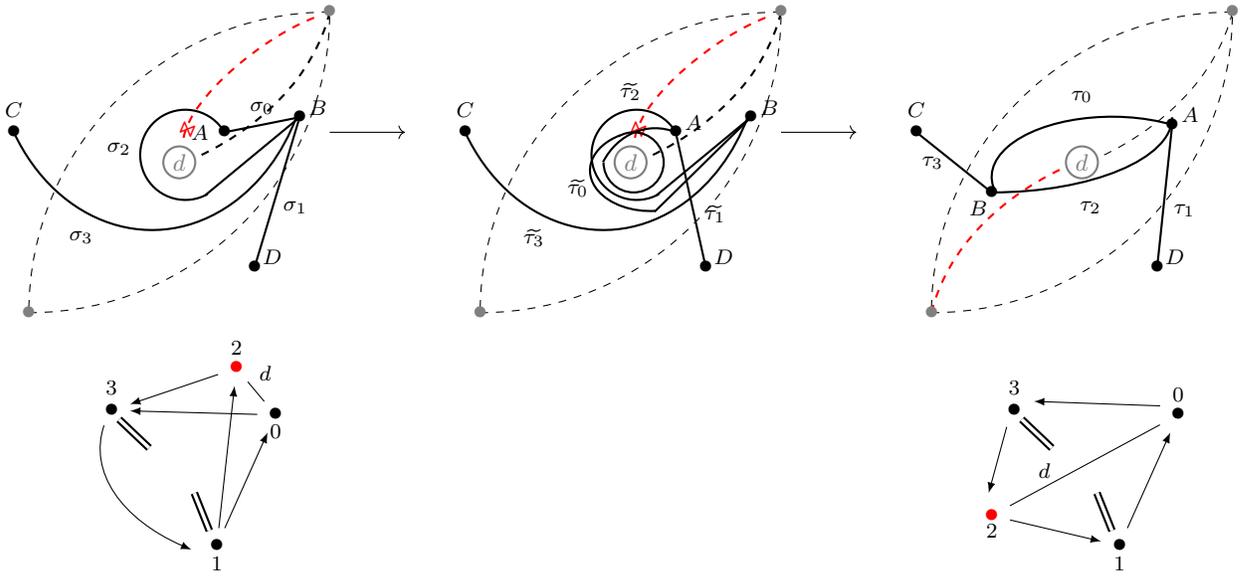
\begin{figure}[ht]\scalebox{1}{
    \centering
 \begin{tikzpicture}[scale=2,
  quiverarrow/.style={black, -latex},
  iquiverarrow/.style={black, -latex, <-},
  mutationarc/.style={dashed, red, very thick},
  arc/.style={dashed, black},
  point/.style={gray},
  vertex/.style={black},
  line/.style={black, thick, shorten <=-8pt, shorten >=-8pt},
  conepoint/.style={gray, circle, draw=gray!100, fill=white!100, thick, inner sep=1.5pt},
  db/.style={thick, double, double distance=1.3pt, shorten <=-6pt}
  ]
 \begin{scope}
 \draw[dashed] (2,2) arc
	[start angle=90,
		end angle=180,
		x radius=2cm,
		y radius =2cm
	] ;
    \draw[dashed] (0,0) arc
	[start angle=-90,
		end angle=0,
		x radius=2cm,
		y radius =2cm
	] ;
    \draw[dashed, thick] (2,2) arc
	[start angle=-10,
		end angle=-60,
		x radius=2cm,
		y radius =1.45cm
	] ;
 \node[point] (p0) at (0,0) {$\bullet$};
 \node[point] (p1) at (2,2) {$\bullet$};
 \node[conepoint] (p3) at (1,1) {\scriptsize $d$};
\draw[dashed,red, thick, shorten <=3pt, shorten >=-2pt] (p3) .. controls +(80:0.45) and +(200:0.45) .. coordinate[pos=0.11](t1) coordinate[midway](w2) (p1);
\node[vertex,rotate=-23, red](tt1) at (t1) {$\bowtie$};
\node[vertex, label=
{[label distance=-5pt]180:{\tiny $A$}}] (A) at (1.3,1.2) {$\bullet$};
\node[vertex, label=
{[label distance=-10pt]80:{\tiny $B$}}] (B) at (1.8,1.3) {$\bullet$};
\node[vertex, label=
{[label distance=-5pt]90:{\tiny $C$}}] (C) at (-0.1,1.2) {$\bullet$};
\node[vertex, label=
{[label distance=-10pt]70:{\tiny $D$}}] (D) at (1.5,0.3) {$\bullet$};
\draw[line] (A)--(B);
\draw[line] (B)--(D);
\draw[thick] (-0.1,1.2) arc
	[start angle=-150,
		end angle=-25,
		x radius=1.06cm,
		y radius =1.3cm
	] ;
\draw[thick] (1.3,1.2) arc
	[start angle=30,
		end angle=300,
		x radius=0.3cm,
		y radius =0.3cm
	] ;
\draw[line] (1.28,0.87)--(B);
   \node[above] at  ($(A)! 0.5!(B)$) {\tiny $\sigma_0$};
  \node[right] at  ($(B)! 0.6!(D)$) {\tiny $\sigma_1$};
  \node[below] at  ($(A)! 0.5!(C)$) {\tiny $\sigma_2$};
  \node[left] at  (0.5,0.5) {\tiny $\sigma_3$};
 \begin{scope}[yshift=-2cm]
 \node[vertex, label=
{[label distance=-22pt]-90:{\tiny $3$}}] (v3) at (0.55,1.35) {$\bullet$};
 \node[vertex, label=
{[label distance=-20pt]90:{\tiny $1$}}] (v1) at (1.25,0.45) {$\bullet$};
 \node[vertex, red, label=
{[label distance=-21pt]-90:{\tiny $2$}}] (v0) at (1.38,1.63) {$\bullet$};
 \node[vertex, label=
{[label distance=-7pt]270:{\tiny $0$}}] (v2) at (1.64,1.32) {$\bullet$};
  \draw[quiverarrow] (v0)--(v3);
  \draw[quiverarrow] (v2)--(v3);
  \draw[quiverarrow] (v1)--(v0);
  \draw[quiverarrow] (v1)--(v2);
  \draw [-, shorten <=-2pt, shorten >=-2pt] (v2) to node[above right,xshift=-0.1cm] {\tiny $d$} (v0);
  \draw[quiverarrow] (0.5,1.25) arc
	[start angle=170,
		end angle=230,
		x radius=1.7cm,
		y radius =0.88cm
	] ;
 \draw[thick, double, double distance=1.3pt] (0.6,1.29)--(0.8,1.1);
 \draw[thick, double, double distance=1.3pt] (1.2,0.55)--(1.1,0.8);
 \end{scope}
\draw [->] (2,1.2) to (2.5,1.2);
 \end{scope}
 \begin{scope}[xshift=3cm]
  \draw[dashed] (2,2) arc
	[start angle=90,
		end angle=180,
		x radius=2cm,
		y radius =2cm
	] ;
    \draw[dashed] (0,0) arc
	[start angle=-90,
		end angle=0,
		x radius=2cm,
		y radius =2cm
	] ;
 \draw[dashed, thick] (2,2) arc
	[start angle=-10,
		end angle=-60,
		x radius=2cm,
		y radius =1.45cm
	] ;
 \node[point] (p0) at (0,0) {$\bullet$};
 \node[point] (p1) at (2,2) {$\bullet$};
 \node[conepoint] (p3) at (1,1) {\scriptsize $d$};

\draw[dashed,red, thick, shorten <=3pt, shorten >=-2pt] (p3) .. controls +(80:0.45) and +(200:0.45) .. coordinate[pos=0.11](t1) coordinate[midway](w2) (p1);
\node[vertex,rotate=-23, red](tt1) at (t1) {$\bowtie$};
\node[vertex, label=
{[label distance=-10pt]80:{\tiny $A$}}] (A) at (1.3,1.2) {$\bullet$};
\node[vertex, label=
{[label distance=-10pt]80:{\tiny $B$}}] (B) at (1.8,1.3) {$\bullet$};
\node[vertex, label=
{[label distance=-5pt]90:{\tiny $C$}}] (C) at (-0.1,1.2) {$\bullet$};
\node[vertex, label=
{[label distance=-10pt]70:{\tiny $D$}}] (D) at (1.5,0.3) {$\bullet$};
\draw[line] (A)--(D);
\draw[thick] (-0.1,1.2) arc
	[start angle=-150,
		end angle=-25,
		x radius=1.06cm,
		y radius =1.3cm
	] ;
\draw[thick] (1.3,1.2) arc
	[start angle=30,
		end angle=300,
		x radius=0.3cm,
		y radius =0.3cm
	] ;
\draw[line] (1.28,0.87)--(B);
  \node[above] at  (0.65,0.7) {\tiny $\widetilde{\tau_0}$};
  \node[right] at  ($(A)! 0.6!(D)$) {\tiny $\widetilde{\tau_1}$};
  \node[below] at  (1,1.6) {\tiny $\widetilde{\tau_2}$};
  \node[left] at  (0.5,0.5) {\tiny $\widetilde{\tau_3}$};
\draw[thick] (1.3,1.2) arc
	[start angle=70,
		end angle=150,
		x radius=0.4cm,
		y radius =0.45cm
	] ;
  \draw[thick] (0.82,1) arc
	[start angle=180,
		end angle=480,
		x radius=0.2cm,
		y radius =0.2cm
	] ;
  \draw[thick] (0.93,1.18) arc
	[start angle=120,
		end angle=275,
		x radius=0.4cm,
		y radius =0.27cm
	] ;
   \draw[line] (B)--(1.26,0.77);
  \draw [->] (2,1.2) to (2.5,1.2);
 \end{scope}
  \begin{scope}[xshift=6cm]
  \draw[dashed] (2,2) arc
	[start angle=90,
		end angle=180,
		x radius=2cm,
		y radius =2cm
	] ;
    \draw[dashed] (0,0) arc
	[start angle=-90,
		end angle=0,
		x radius=2cm,
		y radius =2cm
	] ;
 \draw[dashed] (2,2) arc
	[start angle=-10,
		end angle=-60,
		x radius=2cm,
		y radius =1.45cm
	] ;
  \draw[dashed, red, thick] (0,0) arc
	[start angle=170,
		end angle=120,
		x radius=2cm,
		y radius =1.45cm
	] ;
  \node[point] (p0) at (0,0) {$\bullet$};
 \node[point] (p1) at (2,2) {$\bullet$};
 \node[conepoint] (p3) at (1,1) {\scriptsize $d$};
\node[vertex, label=
{[label distance=-5pt]90:{\tiny $C$}}] (C) at (-0.1,1.2) {$\bullet$};
\node[vertex, label=
{[label distance=-10pt]70:{\tiny $D$}}] (D) at (1.5,0.3) {$\bullet$};
\node[vertex, label=
{[label distance=-10pt]80:{\tiny $A$}}] (A) at (1.6,1.25) {$\bullet$};
\node[vertex, label=
{[label distance=-10pt]200:{\tiny $B$}}] (B) at (0.4,0.8) {$\bullet$};
\draw[line] (C)--(B);
\draw[line] (A)--(D);
\draw[thick] (1.6,1.25) arc
	[start angle=60,
		end angle=195,
		x radius=0.8cm,
		y radius =0.4cm
	] ;
  \draw[thick] (0.4,0.8) arc
	[start angle=270,
		end angle=360,
		x radius=1.2cm,
		y radius =0.5cm
	] ;
     \node[left] at  ($(C)! 0.5!(B)$) {\tiny $\tau_3$};
  \node[right] at  ($(A)! 0.6!(D)$) {\tiny $\tau_1$};
  \node[below] at  (1,1.55) {\tiny $\tau_0$};
  \node[left] at  (1.2,0.7) {\tiny $\tau_2$};
 \begin{scope}[yshift=-2cm]
 \node[vertex, label=
{[label distance=-22pt]-90:{\tiny $3$}}] (v3) at (0.55,1.35) {$\bullet$};
 \node[vertex, label=
{[label distance=-20pt]90:{\tiny $1$}}] (v1) at (1.25,0.45) {$\bullet$};
 \node[vertex, label=
{[label distance=-21pt]-90:{\tiny $0$}}] (v2) at (1.64,1.32) {$\bullet$};
 \node[vertex, red, label=
{[label distance=-19pt]90:{\tiny $2$}}] (v0) at (0.4,0.65) {$\bullet$};
  \draw[quiverarrow] (0.5,1.24)--(0.38,0.8);
  \draw[quiverarrow] (v0)--(v1);
  \draw[quiverarrow] (v1)--(v2);
  \draw[quiverarrow] (1.52,1.38)--(0.68,1.41);
  \draw (0.75,0.95) node{$\scriptstyle d$};
   \draw (v0)--(v2);
    \draw[thick, double, double distance=1.3pt] (0.6,1.29)--(0.8,1.1);
    \draw[thick, double, double distance=1.3pt] (1.2,0.55)--(1.1,0.8);
 \end{scope}
 \end{scope}
 \end{tikzpicture}}
 \caption{Mutation of a tagged to an untagged arc or viceversa, case $2$.}
\label{fig:flipbraid_case_2}
\end{figure}

Consider now the flip in Figure~\ref{fig:mutation_untagged} with $T$ on the left. The left hand side of Figures~\ref{fig:flipbraid_case_2} show an embedding of the braid graph $D_T$ which is good at $2$.
Applying Lemma~\ref{lemma_sergiescu}, the middle figure shows the paths corresponding to the braids $\widetilde{\tau_i}$. Note that this time, the conjugated elements are $\sigma_1$, as there is an arrow $1\rightarrow 2$, and $\sigma_0$, as the mutated arc $\alpha_2$ is rotated anti-clockwise to the flipped arc and $\alpha_0$ is also an arc at the cone point. Rotating the vertices $A$ and $B$ clockwise (with $A$ moving anti-ckockwise around the cone point) we obtain the diagram on the right of Figure~\ref{fig:flipbraid_case_2}, where we used the isomorphism $B_{T'}\rightarrow B_T$ given by $\tau_i\mapsto \widetilde{\tau_i}$ from \cite[Remark 3.2]{GM}. Composing this with the isomorphism $\varphi_2\circ\psi_T$, where $\varphi_2$ is the isomorphism from Theorem~\ref{thm:mutation_invariance}, we obtain an isomorphism $B_{T'}\rightarrow G_{Q_T'}$ sending $\tau_i$ to $t_i$ as required.
Note that not conjugating $\sigma_0$ would result in a final braid $\tau'_0$ swirling around the conepoint. In the case $d=2$, $\tau'_0$ is isotopic to the braid $\tau_0$ and hence conjugating $\sigma_0$ or not does not make a difference, see \cite[proof of Thm~3.6]{GM}. However, for $d>2$, $\tau'_0$ is not isotopic to the braid $\tau_0$ appearing in $B_{T'}$ and hence it is necessary to conjugate $\sigma_0$.

Note that the right hand side diagrams in Figures~\ref{fig:flip_braid_case1} and~\ref{fig:flipbraid_case_2} coincide. Consider this as $T$ and mutate the red arc. There is now a choice on whether rotating this arc anti-clockwise and obtain the picture on the left hand side of Figure~\ref{fig:flip_braid_case1}, or clockwise and obtain the picture on the left hand side of Figure~\ref{fig:flipbraid_case_2}. The two options correspond respectively to conjugating or not conjugating the braid corresponding to vertex $0$ in the associated quiver. Following arguments similar to the above two cases, one can check that in both cases the result holds.

As the above covers all the possible mutations, and the theorem follows.
\end{proof}

We now have the ingredients we need to complete Remark~\ref{rem:tworelations}.
\begin{remark} \label{rem:tworelationsB}
In the situation of Remark~\ref{rem:tworelations} with $d>2$, we have $s_is_k\not=s_ks_i$. The situation described there is the right hand diagram of Figure~\ref{fig:flip_braid_case1} with $i=0$ and $k=2$. The element $\sigma=\tau_2^{-1}\tau_0^{-1}\tau_2\tau_0$ is a pure braid and $\xi(\sigma)\in\pi_1(\MO_d)$ is a single strand winding around the pole twice, where
$\xi$ is the map from the proof of Lemma~\ref{lem:lorderd}.
Arguing as in the proof of Lemma~\ref{lem:lorderd}, we have that $\sigma$ is not equal to the identity, so $\tau_0\tau_2\not=\tau_0\tau_2$ and hence $s_0s_2\neq s_2s_0$ by Theorem~\ref{thm_braid_interpretation}.
\end{remark}

\section{Presentations of \texorpdfstring{$G(d,d,n)$}{G(d,d,n)}}
Let $T$ be a tagged triangulation of $(X,M)$, and $Q_T$ the associated quiver. Let $G'_{Q_T}$ be the group defined in the same way as $G_{Q_T}$ (see Definition~\ref{defn_GQ_associated_group}) with the additional relations $s_i^2=e$ for all $i$. Then we have:

\begin{theorem}
\label{thm:Gddnpresentation}
Let $T$ be a tagged triangulation of $(X,M)$. Then $G'_{Q_T}$ is isomorphic to $G(d,d,n)$, thus giving a presentation of $G(d,d,n)$.   
\end{theorem}
\begin{proof}
This follows from Theorem~\ref{thm:Bddnpresentation} and the presentation of $G(d,d,n)$ given in~\cite{ariki} (see~\cite[Prop.\ 3.2]{BMR98}).
Note that the braid diagram that gives the presentation of $B(d,d,n)$ in~\cite[Table 5]{BMR98} is the opposite of the diagram that gives the presentation of $G(d,d,n)$ in~\cite[Table 2]{BMR98} (see~\cite[Thm.\ 2.27]{BMR98}), but the presentation corresponding to the opposite diagram in this case is equivalent to that corresponding to the original diagram: passing to the opposite diagram amounts only to a relabelling, and therefore does not change the isomorphism class of the presented group.
\end{proof}

Finally, we will show that, by applying a result from~\cite{shi2005}, the generators can be regarded as reflections in $G(d,d,n)$, and explain how this can be done explicitly.

We first recall the result of Shi~\cite{shi2005} that we need. For this we need the following definition of a graph from the paper (slightly modified in our discussion here).

\begin{defn} \label{def:reflectiongraph}
Let $R$ be a set of reflections in $G(d,d,n)$.
Then $\overline{\Gamma}_R$ is the edge-labelled directed (multi-)graph with vertex set $\{1,2,\ldots ,n\}$.
We take $R$ as the set of directed edges. An element $r=s(a,b;c)$ in $R$ (see Section~\ref{s:buildingbeta}) has start vertex $a$, end vertex $b$ and is labelled $c$. We adopt the convention that such a directed edge is equivalent to a directed edge from $b$ to $a$ labelled $-c$. This convention ensures that $\overline{\Gamma}_R$ is well-defined, noting that $s(a,b;c)=s(b,a;-c)$.
We define $\Gamma_R$ to be the underlying unoriented graph of $\overline{\Gamma}_R$ with the labels removed: note that this graph is also well-defined.
\end{defn}

Suppose that $\Gamma_R$ is connected and contains precisely one cycle. By reversing some directed edges in $\overline{\Gamma}_R$ if necessary (and thus also negating their labels), we may assume that the corresponding directed edges in $\overline{\Gamma}_R$ form an oriented cycle, $C$.
Set $\delta(R)$ to be the absolute value of the sum of the labels on the directed edges in $C$. 
Note that by taking the absolute value here, we ensure that $\delta(R)$ is well-defined. Then, we have the following:
                            
\begin{theorem}
[{\cite[Thm.\ 2.19]{shi2005}}]\label{thm_shi_generating}
    Let $R$ be a set of reflections in $G(d,d,n)$ such that $\Gamma_R$ is connected and contains precisely one cycle. Then $R$ generates $G(d,d,n)$ if and only if $\delta(R)$ and $d$ are coprime.
\end{theorem}

\begin{lemma}\label{lemma_braidgraph_onecycle}
Let $T$ be a tagged triangulation of $(X,M)$, and let $D_T$ be the corresponding braid graph (see Definition~\ref{def:braidgraph}). Then $D_T$ contains a unique cycle. In fact, it can be obtained from a cycle by adjoining a binary tree (possibly consisting of a single vertex) to each of its vertices.
\end{lemma}
\begin{proof}
By Remark~\ref{rem:conepointportion}, the induced subgraph of the braid graph on the vertices associated to
the connected components of the complement of $T$ incident with the cone point will be a cycle of length at least two. The vertices on the boundary of this union $U$ of connected components must be on the boundary of $X$ (since they are not the cone point), so the tagged triangulation must be built up from $U$ by adding a triangulated polygon to each boundary edge of $U$ intersecting $U$ only in that edge (where we allow a degenerate case consisting of an edge only, i.e. where no polygon is attached).

It follows that the braid graph can be obtained from an oriented cycle by adjoining a binary tree (possibly consisting only of a single vertex) to each of the vertices of the cycle. In particular, it contains a unique cycle as claimed.
\end{proof}

For the rest of this section, we will work in the following setup.

\begin{setup}\label{setup_reflections}
Let $T$ be a tagged triangulation of $(X,M)$, and let $D_T$ be the corresponding braid graph (see Definition~\ref{def:braidgraph}). Fix a numbering $1,2,\dots, n$ of the $n$ vertices of $D_T$. Let $\overline{D}_T$ be a directed graph with underlying unoriented graph $D_T$, chosen so that the unique cycle in $D_T$ is an oriented cycle in $\overline{D}_T$.

We then associate a reflection $s(e)=s(a,b;c(e))$ to each edge $e$ of $D_T$, where the corresponding directed edge in $\overline{D}_T$ has initial vertex $a$ and end vertex $b$ and $c(e)$ is an integer. We do this in such a way that, if the unique oriented cycle in $\overline{D}_T$ consists of vertices
$a_1,a_2,\dots, a_r$ with a directed edge $f_m$ from $a_m$ to $a_{m+1}$ for all $m$ (with $a_{r+1}$ interpreted as $a_1$), then
$\left| \sum_{m=1}^r c_m \right|$ and $d$ are coprime, where $c_m=c(f_m)$.
We define $R$ to be the set of all reflections $s(e)$ for $e$ an edge of $D_T$. Note that, by construction, $\Gamma_{R_T}$ can be identified with $D_T$.
\end{setup}

\begin{proposition}\label{prop:reflections_gen_relns}
Let $T$ be a tagged triangulation of $(X,M)$. Then the set $R$ of reflections satisfies the defining relations of $G'_{Q_T}$, where each $s(e)$ is identified with the generator $s_i$ associated to the vertex $i$ of $Q_T$ corresponding to the edge $e$ of $D_T$.
\end{proposition}

\begin{proof}
Note first that each $s(e)$ squares to the identity by definition. The commuting and braid relations from Definition~\ref{defn_GQ_associated_group}, parts (\ref{item_commutation}) and (\ref{item_braid}) respectively, are then satisfied by \cite[Sec.~3.3(2),(3)]{shi2005}. 
Moreover, if the unique cycle in $D_T$ has length $r=2$ then, by \cite[Sec.~3.3(4)]{shi2005} we have that

$$
\underbrace{s(a_1,a_2;c_1)s(a_2,a_1; c_2)s(a_1,a_2;c_1)\cdots\ }_{\text{$d$ terms}} = \underbrace{s(a_2,a_1;c_2)s(a_1,a_2;c_1)s(a_2,a_1;c_2)\cdots\ }_{\text{$d$ terms}}.
$$
Hence, the relation from Definition~\ref{defn_GQ_associated_group}(\ref{item_dreln}) is satisfied.
    
    Consider the situation of the left drawing in Figure~\ref{fig:puzzlepiecesbraids}, where $s,\, t$ and $u$ are the associated reflections. Then, by \cite[Sec.~4.5(v)]{shi2005}, we have $sutu=utus$. Using the braid relations and the fact that the reflections square to the identity, we have that
    $$
    sutu=utus \iff stut=tuts \iff stutts=tutsts \iff stus=tuttst \iff stus=tust.
    $$
    Similarly, one can see that the third equality in the cycle relation is true and $stus=tust=ustu$, i.e. the relations from Definition~\ref{defn_GQ_associated_group}(\ref{item_cycle}) are satisfied.

    As pointed out before, double edges in $Q_T$ only appear if there are exactly two arcs at the conepoint in $T$, or equivalently the only cycle in $D_T$ has length $2$. Suppose that this is the case.
    Let $j$ be a vertex connected to one of two vertices $a_1$ and $a_2$ in the $2$-cycle, say $a_2$ via an edge $g$. In other words, there is a double edge at the vertex in $Q_T$ corresponding to the edge $g$ between $a_2$ and $j$ in $D_T$.
    Denote the vector with a $1$ in the $i^{th}$-entry and $0$ everywhere else by $e_i$ and recall that $\omega_d=e^{2\pi i/d}$.
    By direct computation, one can check that 
\begin{align*}(s(a_2,j;c_g)s(a_1,a_2;c_1)s(a_2,a_1;c_2))^2=
(s(a_1,a_2;c_1)s(a_2,a_1;c_2) s(a_2,j;c_g))^2,
\end{align*}

that is, the expected double edge relation from Definition~\ref{defn_GQ_associated_group}(\ref{item_doubleedge}) holds. In fact, both products clearly only affect the entries in $a_1,\,a_2$ and $j$ positions and we compute that under both the left and the right hand side products of reflections, we have
\begin{align*}
    e_{a_1}\mapsto \omega_d^{-2c_2-2c_1}e_{a_1},\,\,
    e_{a_2}\mapsto \omega_d^{c_2+c_1}e_{a_2},\,\,
    e_{j}\mapsto \omega_d^{c_2+c_1}e_{j}.
\end{align*}

    Similarly, it is easy to see by direct computation that when there are exactly two arcs at the conepoint, incident with two different vertices on the boundary, then the corresponding reflections satisfy the relations from Definition~\ref{defn_GQ_associated_group}(\ref{item_broken4cycle}).
    
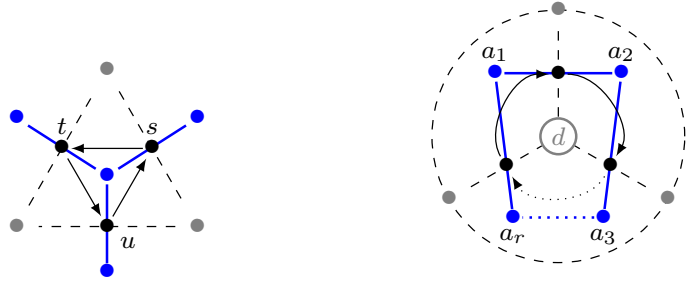
\begin{figure}[ht]
\scalebox{1.2}{
    \centering
 \begin{tikzpicture}[scale=2,
  quiverarrow/.style={black,-latex,shorten <=-4pt, shorten >=-4pt},
  mutationarc/.style={dashed, red, very thick},
  arc/.style={dashed, black},
  bluearc/.style={thick, blue,shorten <=-4pt, shorten >=-4pt},
  point/.style={gray},
  vertex/.style={black},
  conepoint/.style={gray, circle, draw=gray!100, fill=white!100, thick, inner sep=1.5pt},
  db/.style={thick, double, double distance=1.3pt, shorten <=-6pt}
  ]
 \begin{scope}
  \node[point] (p0) at (0,0) {$\bullet$};
  \node[point] (p1) at (0.5,0.866) {$\bullet$};
  \node[point] (p2) at (1,0) {$\bullet$};
  \node[vertex,blue] (b0) at (0.5,0.28) {$\bullet$};
  \node[vertex,blue] (b1) at (1,0.6) {$\bullet$};
  \node[vertex,blue] (b2) at (0,0.6) {$\bullet$};
  \node[vertex,blue] (b3) at (0.5,-0.25) {$\bullet$};
  \draw[bluearc](b0) -- (b1);
  \draw[bluearc](b0) -- (b2);
  \draw[bluearc](b0) -- (b3);
  \draw[arc](p0) -- (p1);
  \draw[arc](p1) -- (p2);
  \draw[arc](p2) -- (p0); 
  \node[vertex] (v0) at (0.25,0.433) {$\bullet$};
  \node[vertex] (v1) at (0.75,0.433) {$\bullet$};
  \node[vertex] (v2) at (0.5,0) {$\bullet$};
  \draw[quiverarrow](v0) -- (v2);
  \draw[quiverarrow](v2) -- (v1);
  \draw[quiverarrow](v1) -- (v0);
  \node[above] at  ($(b0)! 0.5!(b1)$) {\tiny $s$};
  \node[above] at  ($(b0)! 0.5!(b2)$) {\tiny $t$};
  \node[below right] at  ($(b0)! 0.5!(b3)$) {\tiny $u$};
\end{scope}
\begin{scope}[xshift=3cm, yshift=-0.2cm]
\draw[arc] (0,0.7) circle (0.7);
\node[point](q1) at (0,1.4) {$\bullet$};
\node[point](q2) at (0.606,0.35) {$\bullet$};
\node[point](q3) at (-0.606,0.35) {$\bullet$};
\node[conepoint](c) at (0,0.7) {\tiny $d$};
\draw[arc](q1) -- coordinate[midway](m1) (c);
\draw[arc](q2) -- coordinate[midway](m2) (c);
\draw[arc](q3) -- coordinate[midway](m3) (c);
\node[vertex, blue, label=
{[label distance=-7pt]90:{\tiny $a_1$}}] (b0) at (-0.35,1.05) {$\bullet$};
\node[vertex, blue, label=
{[label distance=-7pt]90:{\tiny $a_2$}}] (b1) at (0.35,1.05) {$\bullet$};
\node[vertex, blue, label=
{[label distance=-7pt]-90:{\tiny $a_3$}}] (b2) at (0.25,0.25) {$\bullet$};
\node[vertex, blue, label=
{[label distance=-7pt]-90:{\tiny $a_r$}}] (b3) at (-0.25,0.25) {$\bullet$};
\draw[bluearc](b0) -- (b1);
\draw[bluearc](b2) -- (b1);
\draw[bluearc](b0) -- (b3);
\draw[bluearc,dotted](b2) -- (b3);
\node[vertex](z1) at (m1) {$\bullet$};
\node[vertex](z2) at (m2) {$\bullet$};
\node[vertex](z3) at (m3) {$\bullet$};
\draw[quiverarrow](z1) ..
controls +(0:0.25) and +(60:0.25) .. (z2);
\draw[quiverarrow,dotted](z2) ..
controls +(240:0.25) and +(300:0.25) .. (z3);
\draw[quiverarrow](z3) ..
controls +(120:0.25) and +(180:0.25) .. (z1);
\end{scope}
\end{tikzpicture}}
\caption{Two possible local behaviours appearing in $T$. The dashed arcs are arcs in $T$ (in the right hand picture the arcs shown on the boundary are also allowed to be boundary segments). The blue edges are edges in $D_T$ and $Q_T$ is indicated in black.}
\label{fig:puzzlepiecesbraids}
\end{figure}

Consider now the situation on the right hand side of Figure~\ref{fig:puzzlepiecesbraids}, showing the unique cycle in $D_T$ of length $r$, where $r\geq 3$. We show that the reflections satisfy the relations from Definition~\ref{defn_GQ_associated_group}(\ref{item_cycled}).
In the following computations, subscripts are taken modulo $r$ and addition of the elements $c_j$ is carried out modulo $d$.
For $1\leq m\leq r$, consider the product of $d(r-1)$ reflections
\begin{align}
\label{reln_cycled_reflections}
s(a_m,a_{m+1};c_m)s(a_{m+1},a_{m+2};c_{m+1})\,\,\cdots\,\, s(a_{m+d(r-1)-1},a_{m+d(r-1)};c_{m+d(r-1)-1}).
\end{align}

Fix $j\in \{0,1,\ldots ,r-2\}$.
When computing the image of $e_{a_{m+d(r-1)-1-j}}$ under the product~\eqref{reln_cycled_reflections}, only the reflections $s(a_{m+d(r-1)-i(r-1)}, a_{m+d(r-1)-i(r-1)}; c_{m+d(r-1)-i(r-1)})$ for $i=0,1,\ldots d-1$ act non-trivially. We compute that:
$$
    e_{a_{m+d(r-1)-1-j}}\mapsto \omega_d^{\sum_{i=0}^{d-1}c_{m+d(r-1)-1-j-i(r-1)}} e_{a_{m-1-j}}= \omega_d^{\sum_{i=0}^{d-1}c_{m+d(r-1)-1-j+i}} e_{a_{m-1-j}}.
$$

When computing the image of $e_{a_{m+d(r-1)}}$ under the product~\eqref{reln_cycled_reflections}, every reflection in the product acts non-trivially, and we compute that:
\begin{align*}
    e_{a_{m+d(r-1)}}\mapsto \omega_d^{-\sum_{i=m}^{m+d(r-1)-1}c_i} e_{a_{m}}.
\end{align*}
Note that
$$
    -\sum_{i=m}^{m+dr-1}c_i\equiv 0\mod d,
$$
since this involves adding up a multiple of $d$ copies of each $c_i$.
Hence, we have:
\begin{align*}
    0\equiv -\sum_{i=m}^{m+dr-1}c_i
    \iff
    -\sum_{i=m}^{m+d(r-1)-1}c_i\equiv \sum_{i=m+dr-d}^{m+dr-1}c_i
    \iff -\sum_{i=m}^{m+d(r-1)-1}c_i\equiv \sum_{i=0}^{d-1}c_{m+d(r-1)+i} \mod d.
\end{align*}
It follows that
$$e_{a_{m+d(r-1)}}\mapsto \omega_d^{\sum_{i=0}^{d-1}c_{m+d(r-1)+i}} e_{a_{m}}.$$

Hence, for $p=1,2, \ldots ,r$ we have:
$$e_{a_p}\mapsto 
\omega_d^{\sum_{i=0}^{d-1} c_{p+i}} e_{a_{p-d(r-1)}},$$
and we see that the products (\ref{reln_cycled_reflections}) are equal for all $1\leq m\leq r$ and hence the reflections satisfy the relations from Definition~\ref{defn_GQ_associated_group}(\ref{item_cycled}). This completes the proof of the result.
\end{proof}

\begin{theorem}\label{thm_reflection_interpretation}
    In the situation of Setup~\ref{setup_reflections}, there is an isomorphism of groups $\nu: G'_{Q_T}\rightarrow G(d,d,n)$ sending the generator of $G'_{Q_T}$ associated to vertex $v$ in $Q_T$ to the reflection associated to the edge in $D_T$ that is the dual of $v$.
\end{theorem}

\begin{proof}
    Note first that $\nu$ is a group homomorphism as the images of the generators of $G'_{Q_T}$ satisfy the generating relations of $G'_{Q_T}$ by Proposition~\ref{prop:reflections_gen_relns}. Moreover, by combining Theorem~\ref{thm_shi_generating} and Lemma~\ref{lemma_braidgraph_onecycle}, we conclude that the images of the generators generate $G(d,d,n)$ and hence $\nu$ is surjective. As $G(d,d,n)$ is a finite group, we conclude that $\nu$ is a group isomorphism.
\end{proof}

\begin{remark}
    Note that, in contrast to the braid group case, in the complex reflection group presentation the double edge relations can be made more symmetric.
    In fact, for reflections chosen as in Setup~\ref{setup_reflections}, by direct computations one can check that in the situation of Definition~\ref{defn_GQ_associated_group}(\ref{item_doubleedge}), we have
    \begin{align*}
        s_ks_is_js_ks_j=s_is_js_ks_is_js_k=s_js_ks_is_js_ks_i.
    \end{align*}

    On the other hand, in the situation of Definition~\ref{defn_GQ_associated_group}(\ref{item_broken4cycle}), for $d>2$, we still get two separate relations also in the complex reflection group.
    To see this, let $s_i=s(a_1,a_2;c_1)$, $s_j=s(a_1,j;c(g))$, $s_k=s(a_2,a_1;c_2)$ and $s_l=s(a_2,l;c(h))$
    be reflections satisfying Setup~\ref{setup_reflections} and corresponding respectively to the vertices $i,j,k$ and $l$ in Definition~\ref{defn_GQ_associated_group}(\ref{item_broken4cycle}). Then, it is easy to compute that
    \begin{align*}
        &s_is_js_ks_ls_is_j: e_{a_1}\mapsto \omega_d^{c_1}e_{a_2},\\
        &s_js_ks_ls_is_js_k: e_{a_1}\mapsto \omega_d^{-2c_2-c_1}e_{a_2}.
    \end{align*}   
    As, by assumption,
    $c_1+c_2$ is coprime to $d$, and hence invertible modulo $d$, we have that
    \begin{align*}
        -2c_2-c_1= c_1 \text{ mod } d\iff 2(c_1 +c_2)=0 \text{ mod } d\iff d=2.
    \end{align*}
    Hence, for $d>2$ we have that $s_is_js_ks_ls_is_j\neq s_js_ks_ls_is_js_k$.
\end{remark}

\bibliographystyle{alphaurl}
\bibliography{local.bib}

\begin{thebibliography}{HHLP17}

\bibitem[All02]{A}
D.~Allcock.
\newblock Braid pictures for {A}rtin groups.
\newblock {\em Transactions of the American Mathematical Society}, 354.9:3455--3474, 2002.

\bibitem[Ari95]{ariki}
S.~Ariki.
\newblock Representation theory of a {H}ecke algebra of {$G(r,p,n)$}.
\newblock {\em J. Algebra}, 177(1):164--185, 1995.

\bibitem[Bes15]{bessis15}
D.~Bessis.
\newblock Finite complex reflection arrangements are {$K(\pi,1)$}.
\newblock {\em Ann. of Math. (2)}, 181(3):809--904, 2015.

\bibitem[BM15]{BM}
M.~Barot and B.~R. Marsh.
\newblock Reflection group presentations arising from cluster algebras.
\newblock {\em Transactions of the American Mathematical Society}, 367.3:1945--1967, 2015.

\bibitem[BMR98]{BMR98}
M.~Brou\'{e}, G.~Malle, and R.~Rouquier.
\newblock Complex reflection groups, braid groups, {H}ecke algebras.
\newblock {\em J. Reine Angew. Math.}, 500:127--190, 1998.

\bibitem[Bri71]{brieskorn1971fundamentalgruppe}
E.~Brieskorn.
\newblock Die fundamentalgruppe des raumes der regul{\"a}ren orbits einer endlichen komplexen spiegelungsgruppe.
\newblock {\em Inventiones mathematicae}, 12:57--61, 1971.

\bibitem[Fle23a]{F23A}
J.~Flechsig.
\newblock Braid groups and mapping class groups for $2$-orbifolds.
\newblock preprint, 2023.
\newblock URL: \url{https://arxiv.org/abs/2305.04273}.

\bibitem[Fle23b]{F23B}
J.~Flechsig.
\newblock Orbifold braid groups and complex braid groups.
\newblock preprint, 2023.
\newblock URL: \url{https://arxiv.org/abs/2312.10498}.

\bibitem[FLST21]{FLST21}
A.~Felikson, J.~W. Lawson, M.~Shapiro, and P.~Tumarkin.
\newblock Cluster algebras from surfaces and extended affine {W}eyl groups.
\newblock {\em Transform. Groups}, 26(2):501--535, 2021.

\bibitem[FST08]{FST08}
S.~Fomin, M.~Shapiro, and D.~Thurston.
\newblock Cluster algebras and triangulated surfaces. {I}. {C}luster complexes.
\newblock {\em Acta Math.}, 201(1):83--146, 2008.

\bibitem[FST12]{FST}
A.~Felikson, M.~Shapiro, and P.~Tumarkin.
\newblock Cluster algebras and triangulated orbifolds.
\newblock {\em Advances in Mathematics}, 231.5:2953--3002, 2012.

\bibitem[FT16]{FT16}
A.~Felikson and P.~Tumarkin.
\newblock Coxeter groups and their quotients arising from cluster algebras.
\newblock {\em Int. Math. Res. Not. IMRN}, 2016(17):5135--5186, 2016.

\bibitem[FZ02]{FZ02}
S.~Fomin and A.~Zelevinsky.
\newblock Cluster algebras. {I}. {F}oundations.
\newblock {\em J. Amer. Math. Soc.}, 15(2):497--529, 2002.

\bibitem[FZ03]{FZ}
S.~Fomin and A.~Zelevinsky.
\newblock Cluster algebras {II}: Finite type classification.
\newblock {\em Inventiones Mathematicae}, 154:63--121, 2003.

\bibitem[Gar23]{garnier23}
O.~Garnier.
\newblock Regular theory in complex braid groups.
\newblock {\em J. Algebra}, 620:534--557, 2023.

\bibitem[GM17]{GM}
J.~Grant and B.~R. Marsh.
\newblock Braid groups and quiver mutation.
\newblock {\em Pacific Journal of Mathematics}, 290.1:77--116, 2017.

\bibitem[Har69]{harary69}
Frank Harary.
\newblock {\em Graph theory}.
\newblock Addison-Wesley Publishing Co., Reading, Mass.-Menlo Park, Calif.-London, 1969.

\bibitem[Hat02]{hatcher02}
A.~Hatcher.
\newblock {\em Algebraic topology}.
\newblock Cambridge University Press, Cambridge, 2002.

\bibitem[HHLP17]{HHLP}
J.~Haley, D.~Hemminger, A.~Landesman, and H.~Peck.
\newblock Artin group presentations arising from cluster algebras.
\newblock {\em Algebras and Representation Theory}, 20:629--653, 2017.

\bibitem[HHQ24]{HHQ24}
Z.~Han, P.~He, and Y.~Qiu.
\newblock Cluster braid groups of {C}oxeter-{D}ynkin diagrams.
\newblock {\em J. Combin. Theory Ser. A}, 208:Paper No. 105935, 21, 2024.

\bibitem[Jr.19]{caramello}
F.~C.~Caramello Jr.
\newblock Introduction to orbifolds.
\newblock preprint, 2019.
\newblock URL: \url{https://arxiv.org/abs/1909.08699}.

\bibitem[KQ20]{KingQiu20}
A.~King and Y.~Qiu.
\newblock Cluster exchange groupoids and framed quadratic differentials.
\newblock {\em Invent. Math.}, 220(2):479--523, 2020.

\bibitem[Lee13]{lee13}
J.~M. Lee.
\newblock {\em Introduction to smooth manifolds}, volume 218 of {\em Graduate Texts in Mathematics}.
\newblock Springer, New York, second edition, 2013.

\bibitem[LM21]{lewismorales21}
J.~B. Lewis and A.~H. Morales.
\newblock Factorization problems in complex reflection groups.
\newblock {\em Canad. J. Math.}, 73(4):899--946, 2021.

\bibitem[Qiu16]{Qiu16}
Y.~Qiu.
\newblock Decorated marked surfaces: spherical twists versus braid twists.
\newblock {\em Math. Ann.}, 365(1-2):595--633, 2016.

\bibitem[Qiu19]{qiuyusummary19}
Y.~Qiu.
\newblock The braid group for a quiver with superpotential.
\newblock {\em Sci. China Math.}, 62(7):1241--1256, 2019.

\bibitem[QZ20]{QZ20}
Y.~Qiu and Y.~Zhou.
\newblock Finite presentations for spherical/braid twist groups from decorated marked surfaces.
\newblock {\em J. Topol.}, 13(2):501--538, 2020.

\bibitem[Rou21]{roushon21}
S.~K. Roushon.
\newblock Configuration {L}ie groupoids and orbifold braid groups.
\newblock {\em Bull. Sci. Math.}, 171:Paper No. 103028, 35, 2021.

\bibitem[Ser93]{Sergiescu93}
V.~Sergiescu.
\newblock Graphes planaires et pr\'{e}sentations des groupes de tresses.
\newblock {\em Math. Z.}, 214(3):477--490, 1993.

\bibitem[Shi05]{shi2005}
J.~Y. Shi.
\newblock Congruence classes of presentations for the complex reflection groups g (m, 1, n) and g (m, m, n).
\newblock {\em Indagationes Mathematicae}, 16(2):267--288, 2005.

\bibitem[ST54]{shephard1954finite}
G.~C. Shephard and J.~A. Todd.
\newblock Finite unitary reflection groups.
\newblock {\em Canadian Journal of Mathematics}, 6:274--304, 1954.

\bibitem[Thu22]{thurston22}
W.~P. Thurston.
\newblock {\em The geometry and topology of three-manifolds. {V}ol. {IV}}.
\newblock American Mathematical Society, Providence, RI, [2022] \copyright2022.
\newblock Edited and with a preface by S. P. Kerckhoff and a chapter by J. W. Milnor.

\bibitem[Whi32]{whitney32}
H.~Whitney.
\newblock Non-separable and planar graphs.
\newblock {\em Trans. Amer. Math. Soc.}, 34(2):339--362, 1932.

\end{thebibliography}
\end{document}